\documentclass[a4paper]{article}
\usepackage[utf8]{inputenc}
\usepackage[english]{babel}
\usepackage[usenames, dvipsnames]{color}
\usepackage[toc,page]{appendix}

\usepackage{algorithm, algorithmic}
\usepackage{hyperref}
\hypersetup{colorlinks=true, linkcolor=black, citecolor = red, filecolor=magenta, urlcolor=cyan} 
\usepackage{graphicx,pict2e,amssymb,amsmath,bbm,bm,mathabx, mathtools}
\usepackage{array}
\newcolumntype{R}{>{$}r<{$}}
\newcolumntype{L}{>{$}l<{$}}
\newcolumntype{M}{R@{${}\le{}$}L}
\newcolumntype{E}{R@{${}={}$}L}
\newcolumntype{D}{R@{${}:={}$}L}
\usepackage{ntheorem,delarray}
\usepackage{nicefrac}
\usepackage{srcltx}
\usepackage{caption}
\usepackage[a4paper, right = 2in, top = 2cm, left = 1.25in]{geometry}

\usepackage{makeidx}
\usepackage[short]{optidef}
\makeindex\nonstopmode
\floatname{algorithm}{Algorithm}

\long\def\killtext#1{}

\theoremseparator{.}
\newtheorem{theorem}{Theorem}[section]

\newtheorem{proposition}[theorem]{Proposition}
\newtheorem{lemma}[theorem]{Lemma}

\newtheorem{corollary}[theorem]{Corollary}
\newtheorem{definition}{Definition}[section]

\theorembodyfont{\rmfamily}
\newtheorem{remark}[theorem]{Remark}
\newtheorem{exa}[theorem]{Example}

\newtheorem{assumption}[theorem]{Assumption}
\theoremindent.8cm \theorembodyfont{\small}

\newenvironment{proofx}{\noindent{\bf Proof.}}{}
\newenvironment{proof*}[1]{\noindent{\bf Proof of #1.}}{\hfill$\blacksquare$\medskip}

\def\eps{\varepsilon}\def\cvar{{\sf CVaR}}

\def\AA{\mathcal{A}}\def\BB{\mathcal{B}}

\def\HH{\mathcal{H}}
\def\LL{\mathcal{L}}
\def\MM{\mathcal{M}}

\def\Gbf{\mathbf{G}}

\def\Ebb{\mathbb{E}}

\def\Rbb{\mathbb{R}}

\def \zero{\boldsymbol{0}}
\def \prox{\textbf{prox}}

\DeclarePairedDelimiter\abs{\lvert}{\rvert}
\DeclarePairedDelimiter{\bracket}{ [ }{ ] }
\DeclarePairedDelimiter{\paran}{(}{)}
\DeclarePairedDelimiter{\braces}{\lbrace}{\rbrace}
\DeclarePairedDelimiterX{\gnorm}[3]{\lVert}{\rVert_{#2}^{#3}}{#1}

\DeclarePairedDelimiterX{\inprod}[2]{\langle}{\rangle}{#1, #2}
\DeclarePairedDelimiterX{\inprodo}[3]{\langle}{\rangle_{#3}}{#1, #2}

\providecommand{\br}[1]{{\left(#1\right)}}
\providecommand{\sbr}[1]{{\ast\left(#1\right)}}

\newcommand{\tsum}{\textstyle{\sum}}

\newcommand{\wt}[1]{{\widetilde{#1}}}
\newcommand{\wb}[1]{{\widebar{#1}}}
\newcommand{\wh}[1]{{\widehat{#1}}}

\newcommand{\relu}[1]{\left [#1\right]_+}
\newcommand\tab[1][1cm]{\hspace*{#1}}
\newcommand\grad{\nabla}
\newcommand\subgrad[1]{#1'}
\newcommand\pgrad{\partial}
\newcommand{\mleq}[1]{\stackrel{\mathclap{\mbox{\normalfont \scriptsize #1}}}{\leq}}

\DeclareMathOperator*{\argmin}{argmin}

\DeclareMathOperator*{\dom}{dom}

\DeclareMathOperator*{\inte}{int}
\DeclareMathOperator*{\rinte}{rint}
\DeclareMathOperator*{\lin}{lin}

\def\eqnok#1{(\ref{#1})}
\numberwithin{equation}{section}

\newlength{\maxwidth}
\newcommand{\algalign}[2]% #1 = text to left, #2 = text to right
{\makebox[\maxwidth][r]{$#1{}$}${}#2$}

\begin{document}
\title{\Large\bf {Stochastic First-order Methods for Convex and \\
Nonconvex Functional Constrained Optimization}
\thanks{
Boob and Lan was partially supported by the NSF grant CCF 1909298. Deng was partially supported by the NSFC grant 11831002.}}
\author{ 
	Digvijay Boob \thanks{digvijaybb40@gatech.edu, Industrial and Systems Engineering, Georgia Institute of Technology.} \and 
	Qi Deng \thanks{qideng@sufe.edu.cn, School of Information Management and Engineering, Shanghai University of Finance and Economics.} \and 
		Guanghui Lan \thanks{george.lan@isye.gatech.edu, Industrial and Systems Engineering, Georgia Institute of Technology.}}
\date{\vspace{-5ex}}
\maketitle

\begin{abstract}
Functional constrained optimization is becoming more and more important in machine learning and operations research. Such problems have potential applications in risk-averse machine learning, semisupervised learning and robust optimization among others. In this paper, we first present a novel Constraint Extrapolation (ConEx) method for solving convex functional constrained problems, which utilizes linear approximations of the constraint functions to define the extrapolation (or acceleration) step. We show that this method is a unified algorithm that achieves the best-known rate of convergence for solving different 
functional constrained convex composite problems, including convex or strongly convex, and smooth or nonsmooth problems with stochastic objective and/or stochastic constraints. Many of these rates of convergence were in fact obtained for the first time in the literature. In addition, ConEx is a single-loop algorithm that does not involve any penalty subproblems. Contrary to existing primal-dual methods, it does not require the projection of Lagrangian multipliers into a (possibly unknown) bounded set. {Second, for  nonconvex functional constrained problems, we introduce a new proximal point method which transforms the initial nonconvex problem into a sequence of convex 
 problems by adding quadratic terms to both the objective and constraints. Under certain MFCQ-type assumption, we establish the convergence and rate of convergence of this method to KKT points  when the convex subproblems are solved exactly or inexactly. For large-scale and stochastic problems,  we present a more practical proximal point method in which the approximate solutions of the subproblems are computed by the aforementioned ConEx method. Under a strong feasibility assumption, we establish the total iteration complexity of ConEx required by this inexact proximal point method for a variety of problem settings, including nonconvex smooth or nonsmooth problems with stochastic objective and/or stochastic constraints. To the best of our knowledge, most of these convergence and complexity results of the proximal point method for nonconvex problems also seem to be new in the literature.}

\vspace{0.1in}


\end{abstract}

\vspace{0.1in}


\allowdisplaybreaks

\section{Introduction}\label{sec:intro}
In this paper, we study the following  composite optimization problem with functional constraints:
%\vspace{-1em}
\begin{equation}\label{main-prob}
\begin{split}
\min_{x \in X} \tab & \psi_0(x) :=  f_0(x) + \chi_0(x) \\[-1.5mm]
\text{s.t.} \tab & \psi_i(x) := f_{i}(x) + \chi_{i}(x) \le 0, \quad i=1,\dots,m.	
\end{split}
\end{equation}
%\begin{align} \min\limits_{x \in X}\tab  &f(x)  \\
%\text{s.t.}\tab &\phi^{(i)}(x) \le 0\nonumber
%\end{align}
Here, $X \subseteq \Rbb^n$ is a convex compact set,  $f_0: X \to \Rbb$ and  $f_i: X \to \Rbb,\ i =1, \dots, m$ are continuous functions 
which are not necessarily convex, %$\chi_0:X\rightarrow \mathbb{R}$ is \added{convex continuous} function,
 and $\chi_i:X\rightarrow \mathbb{R}, i =0,1, \dots, m$ are convex and continuous functions.
Problem~\eqnok{main-prob} covers different convex and nonconvex settings depending on 
the assumptions on $f_i$ and $\chi_i$,  $i = 0, \dots, m$.

In the convex setting, we assume that $f_i$, $i=0, \ldots, m$, are convex or strongly convex functions,
which can be either smooth, nonsmooth or the sum of smooth and nonsmooth
components. 

We also assume that $\chi_i$, $i = 0, \ldots, m$, are ``simple" functions in the sense that, for any given vector $v \in \Rbb^n$ and non-negative weight vector $w \in \Rbb^m$, 
a certain proximal operator associated with the function $\chi_0(x) + \tsum_{i=1}^m w_i\chi_i(x) + \inprod*{v}{x}$ can be computed efficiently. 
For such problems, 
Lipschitz smoothness properties of $\chi_i$'s is of no consequence due to the simplicity of this proximal operator.

For the nonconvex case, 
we assume that $f_i$, $i =0, \ldots, m$, are smooth functions, which 
are not necessarily convex, but satisfying 
a certain lower curvature condition (c.f. \eqref{eq:weak-convexity}).
However, we do not put the simplicity assumption about the proximal operator associated with
convex functions $\chi_i$, $i=0, \ldots, m$,
in order to cover a broad class of nonconvex problems, including those with non-differentiable objective 
functions or constraints. 

Constrained optimization problems of the above form are prevalent in data science. One such example arises from risk averse machine learning.
Let $\ell(\cdot, \xi):\Rbb^n \times \Xi \to \Rbb$ models the loss for a random data-point $\xi \in \Xi$.  Our goal is to
minimize a certain risk measure~\cite{rockafellar00cvar,shapirobook}, e.g., the so-called conditional value at risk that penalizes only the positive deviation of the loss function, subject
to the constraint that the expected loss is less than a threshold value. Therefore, one can formulate
this problem as 
\begin{equation}\label{eq:motivation_pb_1}
\begin{split}
\min_{x \in X} \tab &\cvar [\ell(x,\omega)]\\[-1.5mm]
\text{s.t.}\tab &\Ebb [\ell(x,\omega)] \le c,
\end{split}
\end{equation}
where $\cvar$ denotes conditional value at risk
and $c$ is the tolerance on the average loss that one can consider as acceptable. 
In many practical situations, the loss function $\ell(x, \omega)$ is nonconvex w.r.t. $x$.
Other examples of problem~\eqref{main-prob} can also be found in semi-supervised learning, where one would like to
minimize the loss function defined over the labeled samples, subject to certain proximity type constraints
for the unlabeled samples.

There exists a variety of literature on solving convex functional constrained optimization problems \eqref{main-prob}. 
One research line focuses on primal methods, e.g., cooperative subgradient methods~\cite{Polyak67,lan-zhou2016stoch} 
and level-set methods \cite{LNN,nesterov2018lectures,lin2018level,aravkin2018level,lin_svrg}, which do not involve the Lagrange multipliers. 
One possible limitation of these methods is the difficulty to directly achieve accelerated
rate of convergence when the objective or constraint functions are smooth. 
Constrained convex optimization problems can also be solved by reformulating them as saddle point problems which will then be solved 
by using primal-dual type algorithms (see \cite{Nemirovski2005mirror,aybat2018primal}).
The main hurdle for existing primal-dual methods exists
in that they require the projection of dual multipliers onto a ball whose diameter
is usually unknown.
Other approaches for constrained convex problems include
the classical exact penalty, quadratic penalty and augmented Lagrangian methods~\cite{bertsekasnonlinear,LanMon13-1,LanMon16-1,Xu2019-1}.
These approaches however require the solutions of penalty subproblems and hence
are more complicated than primal and primal-dual methods. Recently, 
research effort has also been directed to 
stochastic optimization problems with functional constraints \cite{lan-zhou2016stoch,aravkin2018level}. 
In spite of many interesting findings, existing methods for solving these problems are still limited: a) many primal methods solve only stochastic
problems with deterministic constraints~\cite{lan-zhou2016stoch},
and  the convergence for accelerated primal-dual methods \cite{Nemirovski2005mirror,aybat2018primal}  has not been studied
for stochastic functional constrained problems; and
b) a few algorithms for solving problems with expectation constraints  require either
a constraint evaluation step~\cite{lan-zhou2016stoch}, 
or stochastic lower bounds on the optimal value~\cite{aravkin2018level}, 
thus relying on a light-tail assumption for the stochastic noise and 
conservative sampling estimates based on Bernstein inequality.
Some other algorithms require even more restrictive assumptions that the
noise associated with stochastic constraints has to be bounded~\cite{YuNeelyWei17}.

The past few years has also seen a resurgence of interest 
in the design of efficient algorithms 
 for nonconvex stochastic optimization, especially for stochastic and finite-sum problems
due to their importance in machine learning.
Most of these studies need to assume that the constraints are convex, and
focus on the analysis of  iteration complexity,
i.e.,  the number of iterations required to find an approximate stationary point, as well as possible ways to accelerate convergence to
such approximate solutions. 
If the nonconvex functional constraints do not appear, one approach for solving \eqref{main-prob}
is to directly generalize stochastic gradient descent type methods (see \cite{GhaLan12,RN97,reddi2016stochastic,allen-zhu2016variance,FangLiLinZhang18-1,ZhouPanGu2008,WangJiZhouLiangTa18-1,NguyeLiuSchTak17-1,PhamNguyenPhanTran18-1,Lan19})
for solving problems with nonconvex objective functions.
An alternative approach is to indirectly utilize convex optimization methods
within the framework of proximal-point methods
which transfer nonconvex optimization problems into a series
of convex ones (see \cite{guler1992new,bertsekas2015convex,RN304,davis2017proximally,kong2018complexity,lan2018accelerated,rafique2018convex,lee19nonconvexminmax}).
While direct methods are simpler and hence easier to implement, indirect methods may provide stronger theoretical performance
guarantees under certain circumstances, e.g., when the problem has a large conditional number, many components and/or multiple blocks~\cite{lan2018accelerated}. 
However, if nonconvex functional constraints $ \psi_i(x) \le 0$ do appear in \eqref{main-prob},
the study on its solution methods is scarce.
While there is a large body of work on the asymptotic analysis and the  optimality conditions of penalty-based approaches  for general constrained 
nonlinear programming (for example, see \cite{bertsekasnonlinear,martinez2003a,andreani2018strict,andreani2011on,Quoc2011}), only a few works   discussed
the complexity of these methods for solving problems with nonconvex
functional constraints \cite{cartis2014on,WangMaYuan17-1,Facchinei19}. However, these techniques are not applicable to our setting because they cannot guarantee the feasibility of the generated
solutions, but certain local non-increasing properties for the constraint functions. On the other hand, 
the feasibility of the nonconvex functional constraints appear to be important in our problems of interest.

In this paper, we attempt to address some of the aforementioned significant issues associated with both convex and nonconvex
functional constrained optimization. 
Our contributions mainly exist in the following several aspects.

First, for solving convex functional constrained problems, we present
a novel primal-dual type method, referred to as the 
 Constraint Extrapolation (ConEx) method.
 One distinctive feature of this method from existing primal-dual methods is that it 
utilizes linear approximations of the constraint functions to define the extrapolation (or acceleration/momentum) step. 
As a consequence, contrary to the well-known Nemirovski's mirror-prox method~\cite{Nemirovski2005mirror}
and an interesting primal-dual method recently developed by Hamedani and Aybat~\cite{aybat2018primal}, ConEx does not require the projection 
of Lagrangian multipliers onto a (possibly unknown) bounded set. In addition, ConEx is a single-loop algorithm that does not involve any penalty subproblems.
Due to the built-in acceleration step, this method can explore problem structures and hence achieve
better rate of convergence than primal methods. In fact,
we show that this method is a unified algorithm that achieves the best-known rate of convergence for solving different convex
functional constrained problems, including convex or strongly convex, and smooth or non-smooth problems with stochastic objective and/or stochastic constraints.

\begin{table}[H]
	\centering
\begin{tabular}{|c|c|c|c|c|}\hline
	&\multicolumn{2}{c|}{Strongly convex \eqref{main-prob}} &\multicolumn{2}{c|}{Convex \eqref{main-prob}}\\ \hline
	Cases& Smooth* & Nonsmooth & Smooth* & Nonsmooth \\ \hline
	Deterministic & $O(1/\sqrt{\eps})$ & $O(1/\eps)$ & $O(1/{\eps})$ & $O(1/\eps^2)$ \\ \hline
	Semi-stochastic & $O(1/{\eps})$ & $O(1/\eps)$ & $O(1/\eps^2)$ & $O(1/\eps^2)$\\ \hline
	Fully-stochastic & $O(1/{\eps^2})$ & $O(1/\eps^2)$& $O(1/\eps^2)$ & $O(1/\eps^2)$\\ \hline
\end{tabular}
\caption{Different convergence rates of the ConEx method for strongly convex/convex, and smooth/nonsmooth objective and/or constraints. 
Deterministic means both objective and constraints are deterministic, semi-stochastic means objective is stochastic but constraints are deterministic, 
fully-stochastic means both objective and constraints are stochastic.\\
{(*) Results for the smooth case hold if $B \ge \gnorm{y^*}{2}{}+1$. Otherwise, the ConEx method converges at the rate of nonsmooth problems.}  
}\label{table1}
\end{table}
Table~\ref{table1} provides a brief summary for the iteration complexity of the ConEx method for solving different
functional constrained problems. For the strongly convex case, ConEx can obtain convergence to an $\eps$-approximate solution
 (i.e., optimality gap and infeasibility are $O(\eps)$) as well as convergence of the last iterate to the optimal solution. 
 The complexity bounds provided in Table~\ref{table1} for the strongly convex case hold for both types of convergence criterions. 
 For semi- and fully-stochastic case, we use the notion of expected convergence instead of exact convergence used in the deterministic case.
It should be noted that in Table~\ref{table1}, we ignore the impact of various Lipschitz constants and/or stochastic noises for the sake
of simplicity. In fact, the ConEx method achieves quite a few new complexity results by
reducing the impact of these Lipschitz constants. Moreover, to the best of our knowledge, it attains for the first time
 the optimal iteration and sampling
complexity for solving general stochastic constrained problems without requiring
the boundedness or light-tail assumptions on the
stochastic subgradients (see Theorems~\ref{cor:step_size_strong_cvx} and~\ref{thm:convergence_convex}
and discussions afterwards).
Even though ConEx is a primal-dual type method, 
we can show its convergence irrespective of the knowledge of the optimal Lagrange multiplier {$y^*$}, as it does not require the projection of multipliers onto the ball. 
In particular, convergence rates of the ConEx method for nonsmooth cases (either convex or strongly convex) in Table \ref{table1} holds irrespective of the knowledge of the optimal Lagrange multipliers. 
For smooth cases, if certain parameter {$B$ is not big enough, i.e., $B < \gnorm{y^*}{2}{}+1$}, then 
it converges at the rates for nonsmooth problems of the respective case. As one can see from Table \ref{table1}, such a change would cause a suboptimal convergence rate in terms of $\eps$ only for the deterministic case,
but complexity will be the same for both
semi- and fully-stochastic cases.
It is worth mentioning that faster convergence rates for the smooth deterministic case can still be attained by incorporating certain line search procedures.
To the best of our knowledge, this is the first time in the literature that a simple single-loop
algorithm was developed for solving all different types of convex functional constrained problems
in an optimal manner.
  
Second, 
we extend the ConEx method for the
 nonconvex setting and present a new framework of proximal point method for solving the nonconvex functional  constrained  optimization problems,
which otherwise seem to be difficult to solve by using direct approaches. 
The key component of our method is to exploit the structure of the nonconvex objective and constraints $\psi_i$, $i =0, \ldots, m$, thereby turning 
the original problem into a sequence of functional constrained subproblems with a strongly convex objective and strongly convex constraints. 
We show that if the proximal point method has a strictly feasible initial solution and its subproblem is exactly solved, then the whole generated sequence remains strictly feasible. Hence, Slater's condition guarantees the existence of Lagrange multipliers and strong duality for each subproblem.
Moreover, we show that the need of an exact optimal subproblem solution can be relaxed by a termination criterion based on the distance to the optimal solution, leading to a more general inexact proximal point method that still preserves the appealing property of strict feasibility.
Under the Mangasarian-Fromovitz constraint qualification (MFCQ), we show that the inexact proximal point method converges  asymptotically to the KKT points, and provide the first iteration complexity of such proximal point method. More specific, we show that this method requires $O(1/\eps)$ iterations to obtain an appropriately defined 
$(\eps,\eps^2)$-KKT point (see Theorem~\ref{thm:uniform-bound-and-rate} and discussions afterwards). % (see Section~\ref{sec:prelim} for more details).

For large-scale and stochastic optimization, 
 we propose an inexact proximal point method for which the subproblems are approximately solved by first-order methods such as ConEx. 
 However, due to the optimization challenge in this setting, it is in general difficult to obtain highly accurate solutions such that the whole sequence remains strictly feasible.
To overcome this difficulty, we propose a new and verifiable strong feasibility assumption  which alleviates the need of generating strict feasible solutions. We develop the inexact proximal point method using a termination criterion in terms of the functional optimality gap and the constraint violation, showing that this method requires $O(\Delta/\eps)$ iterations to obtain some $(\eps,\eps)$-KKT point solutions, where $\Delta$ depends on the inexactness errors summed over the iterations. When the proximal point subproblems are solved by ConEx, we present the total iteration count of ConEx to achieve the approximate  solutions under different smooth and nonsmooth settings (see Theorem~\ref{thm:inexact-pp}, Corollary~\ref{cor:iteration_complexity_inexact}
and discussions afterwards).

Close to the completion of our paper, we notice 
that Ma et. al.~\cite{MaLinYang19} also worked independently
on the analysis of the proximal-point methods for nonconvex functional constrained problems.
In fact, the initial version of ~\cite{MaLinYang19} 
was released almost at the same time as ours. In spite of some overlap, there exist
a few essential differences between our work and~\cite{MaLinYang19}.  
First, we establish the convergence/complexity of
the proximal point method under a variety of constraint qualification conditions, including
MFCQ, a stronger notion of MFCQ (Assumption \ref{ass:strong-mangafromo}), and strong feasibility.
Hence, our work covers a broader class of nonconvex problems,
while \cite{MaLinYang19}  only consider problems that satisfy a uniform Slater's condition. Strong feasibility condition is stronger than the
uniform Slater's condition  but it is easier to verify.
 Second,  \cite{MaLinYang19} uses a different definition of subdifferential than ours and
 the definition of the KKT conditions in \cite{MaLinYang19}  comes from convex optimization problems.
 While it is unclear under what constraint qualification this KKT condition is necessary for local optimality,
it is possible to put their problem into our composite framework
in \eqnok{main-prob} and compute the subdifferential that provably yields our KKT condition 
under the aforementioned MFCQ.
Third, for solving the convex subproblems we provide a unified
algorithm, i.e., ConEx, that
can achieve the best-known rate of convergence for solving
different problem classes, including deterministic, semi-
and fully-stochastic, smooth and nonsmooth problems.
On the other hand, 
different methods were suggested for 
solving different types of problems in \cite{MaLinYang19}. 
In particular, a variant of the switching subgradient method, which was firstly presented by Polyak in \cite{Polyak67} for the general convex case,
and later extended by~\cite{lan-zhou2016stoch} for the stochastic and strongly convex
cases, was suggested for solving deterministic problems. For the stochastic case they directly apply the algorithm in \cite{YuNeelyWei17} and
hence require stochastic gradients to be bounded. These subgradient methods
do not necessarily yield the best possible rate of convergence if the objective/constraint
functions are smooth or contain certain smooth components.

\paragraph{Outline}This paper is organized as follows. Section~\ref{sec-notn} describes notation and terminologies. Section~\ref{sec:saddle_reformulation} exclusively deals with the ConEx method for solving problem~\eqref{main-prob} in the convex setting. 
Section~\ref{subsec:main_theorems} states the main convergence results of the ConEx method and Section~\ref{subsec:conv_analysis} shows the details of the convergence analysis.
Section~\ref{sec-alg1} presents  exact and inexact proximal point methods for solving problem~\eqref{main-prob} in the nonconvex setting. Section~\ref{sec_exact_ppt} and \ref{sec:inexact_ppt} establishes their convergence behavior and iteration complexity under the MFCQ type assumptions.  
Section~\ref{sec:inexact-strong-feas} investigates the inexact proximal point method under the strong feasibility assumption, and shows an overall iteration complexity result
when the subproblems are solved by the ConEx method.

\subsection{Notation and terminologies \label{sec-notn}}

\setlength{\abovedisplayskip}{2pt}\setlength{\belowdisplayskip}{0.8pt}Throughout the paper, we use the following notations. Let  $[m] := \{1 , \dots, m\}$, $\psi(x) := [\psi_{1}(x), \dots, \psi_{m}(x)]^T$, $f(x) :=[f_1(x), \dots, f_m(x)]^T$ and $\chi(x) := [\chi_1(x), \dots, \chi_m(x)]^T$ and the constraints in \eqref{main-prob} be expressed as $\psi(x)\le \zero$. Here bold $\zero$ denotes the vector of elements $0$. Size of the vector is left unspecified whenever it is clear from the context.
$\gnorm{\cdot}{}{}$ denotes a general norm and $\gnorm{\cdot}{\ast}{}$ denotes its dual norm defined as $\gnorm{z}{\ast}{} := \sup\{z^Tx : \gnorm{x}{}{} \le 1\}$. From this definition, we obtain the $a^Tb \le \gnorm{a}{}{} \gnorm{b}{\ast}{}$. Euclidean norm is denoted as $\gnorm{\cdot}{2}{}$ and standard inner product is denoted as $\inprod{\cdot}{\cdot}$. Let $\BB^2(r) := \{x: \gnorm{x}{2}{} \le r \}$ be the Euclidean ball of radius $r$ centered at origin. Non-negative orthant of this ball is denoted as $\BB^2_+(r)$.
For a convex set $X$, we denote the normal cone at $x \in X$ as $N_X(x)$ and its dual cone as $N^*_X(x)$, interior as $\inte{X}$ and relative interior as $\rinte{X}$. For a scalar valued function $f$ and a scalar $t$, the notation $\{f\le t \}$ stands for the set $\{x:f(x) \le t\}$. The “$+$” operation on sets denotes the Minkowski sum of the sets. We refer to the distance between two sets $A, B \subset \Rbb^n$ as $d(A,B) := \inf_{a \in A, b \in B} \gnorm{a-b}{}{}$.

$\relu{x} := \max\{x,0\}$ for any $x \in \Rbb$. For any vector $x \in \Rbb^k$, we define $\relu{x}$ as element-wise application of the operator $ \relu{\cdot}$. The $i$-th element of vector $x$ is denoted as $x^{(i)}$ unless otherwise explicitly specified a different notation for certain special vectors.

 A function $r(\cdot)$ is \emph{$\lambda$-Lipschitz smooth} if the gradient $\grad r(x)$ is a $\lambda$-Lipschitz function, i.e. for some $\lambda \ge 0$
\begin{equation*}
\gnorm{\nabla r(x) -\nabla r(y)}{*}{} \le \lambda \gnorm{x-y}{}{}, \ \ \forall x, y\in \dom r.
\end{equation*}
An equivalent form is:
\begin{equation*}
 - \tfrac{\lambda}{2} \gnorm{x-y}{}{2} \le r(x) -r(y) - \inprod{\grad r(y)}{x-y} \le  \tfrac{\lambda}{2} \gnorm{x-y}{}{2}, \tab \forall x, y \in \dom{r}.
\end{equation*} 
Note that in the above relation, the upper and lower curvatures need not be same. A refined version of the above property differentiates between negative and positive curvatures. 
\begin{equation}\label{eq:weak-convexity}
r(y) + \inprod{\grad r(y)}{x-y} - \tfrac{\nu}{2} \gnorm{x-y}{}{2} \le r(x), \tab \forall x, y \in \dom{r}.
\end{equation}
Here, we say that $r$ satisfies \eqref{eq:weak-convexity} with parameter $\nu$ with respect to $\gnorm{\cdot}{}{}$. 
In many cases, it is possible that a convex function $r$ is a combination of Lipschitz smooth and nonsmooth functions. 

Let $\omega: X \to \Rbb$ be continuously differentiable with  $L_\omega$-Lipschitz continuous gradient and $1$-strongly convex with respect to $\gnorm{\cdot}{}{}$. We define the prox-function associated with $\omega(\cdot)$ as
\begin{equation} \label{eq:primal_prox}
W(y, x) := \omega(y)-\omega(x)-\inprod{\grad \omega(x)}{y-x}, \tab \forall x,y \in X.
\end{equation}
Based on the smoothness and strong convexity of $\omega(x)$, we have the following relation
\begin{equation} \label{wxy-bound}
W(y, x)\le 	\tfrac{L_\omega}{2}\|x-y\|^2 \le L_\omega W(x,y), \tab \forall x, y\in X.
\end{equation}
{We define the diameter of the compact set $X$ by}
\begin{equation}\label{eq:diameter}
	{D_X \coloneq \max_{x,y \in X} \sqrt{2W(x,y)}.}
\end{equation}
{It immediately follows from the strong convexity of $\omega(\cdot,\cdot)$ that $\|x-y\|\le D_X$ for any $x,y\in X$.}
For any convex function $h$, we denote the subdifferential as $\pgrad h$ as follows: at a point $x$ in the relative interior of $X$, $\pgrad h$ is comprised of all subgradients $h'$ of $h$ at $x$ which are in the linear span of $X-X$. For a point $x \in X \setminus \rinte{X}$, the set $\pgrad h(x)$ consists of all vectors $h'$, if any, such that there exists $x_i \in \rinte{X}$ and $h'_i \in \pgrad h(x_i),  i =1,2, \dots,$ with $x = \lim\limits_{i \to \infty}x_i,\ h' = \lim\limits_{i \to \infty} h'_i$. With this definition, it is well-known that, if a convex function $h:X\to \Rbb$ is Lipschitz continuous, with constant $\MM$, with respect to a norm $\gnorm{\cdot}{}{}$, then the set $\pgrad h(x)$ is nonempty for any $x\in X$ and
\[h' \in \pgrad h(x) \Rightarrow \abs{ \inprod{h'}{d} } \le \MM \gnorm{d}{}{}, \forall d \in \lin{(X-X)},\] 
which also implies
\[h' \in \pgrad h(x) \Rightarrow \gnorm{h'}{\ast}{} \le \MM,\] 
where $\gnorm{\cdot}{\ast}{}$ is the dual norm. See \cite{Ben-Tal05} for more details.
We say that a function $r(\cdot)$ %\comment[id=qd]{$r(\cdot)$ can be nonsmooth?} 
is $\beta$-strongly convex with respect to $W(\cdot,\cdot)$ if 
\begin{equation}
	r(x)\ge r(y) + \inprod{ r'(y)}{x-y}+\beta W(x, y),\tab \forall x, y \in X.
\end{equation}

\section{Constraint Extrapolation for Convex Functional Constrained Optimization} \label{sec:saddle_reformulation}
In this section, we present a novel constraint extrapolation (ConEx) method for solving problem~\eqref{main-prob} in the convex setting.
To motivate our proposed method, consider the equivalent Lagrangian saddle point problem:
\begin{equation}\label{eq:saddle_pb}
\min_{x\in X} \max_{y\ge \zero} \left\{ \LL(x,y):=\psi_0(x)+ \tsum_{i=1}^m y^\br{i} \psi_{i}(x) \right\}.
\end{equation} Let $(x^*,y^*)$ be a \emph{saddle point} solution of \eqref{eq:saddle_pb}, then it satisfies 
\begin{equation}\label{eq:saddle_def}\LL(x^*,y) \le \LL(x^*, y^*) \le \LL(x, y^*),
\end{equation}
for all $x \in X, y \ge \zero$. Moreover, we have that $x^*$ is the optimal solution of \eqref{main-prob}. 
Throughout this section, we assume the
existence of $y^*$ satisfying \eqref{eq:saddle_def}. Denote the optimal value $\psi_0^* := \psi_0(x^*)$.
Then, the following definition describes a widely used optimality measure for the convex problem~\eqref{main-prob}.
\begin{definition} \label{define-apprx-opt}
	A point $\wb{x}\in X$ is called a $(\delta_o, \delta_c)$-optimal solution of problem \eqref{main-prob} if
	%there exists an optimal solution $x^{*}\in X$ such that,
	\begin{equation*}
	\psi_0(\wb{x})-\psi_0^* \le \delta_0 \quad \text{and} \quad
	\gnorm{\relu{\psi(\wb{x})} }{2}{} \le\delta_c.
	\end{equation*}
	A stochastic $(\delta_o, \delta_c)$-approximately optimal solution satisfies
	\[ \Ebb \bracket{\psi_0(\wb{x})-\psi_0^*}\le \delta_0 \quad \text{and} \quad \Ebb \bracket{ \gnorm{ \relu{\psi(\wb{x})} }{2}{} } \le \delta_c. \]	
\end{definition}
As mentioned earlier, for the convex composite case, we assume that $\chi_i, i =0, \dots, m$, are ``simple" functions in the sense that, for any vector $v \in \Rbb^n$ and nonnegative $w \in \Rbb^m$, we can efficiently compute the following \prox \ operator  
\begin{equation}\label{eq:W-prox}
{\prox(w,v,\wt{x}, \eta) := \argmin\limits_{ x \in X} \braces[\big]{\chi_0(x)+\tsum_{i=1}^m w_i\chi_i(x) + \inprod*{v}{x} + \eta W(x, \wt{x})}.}
\end{equation}
The $\prox$ operator above uses Bregman divergence instead of Euclidean norm used in standard prox operator.

\subsection{The ConEx method} \label{subsec:main_theorems}
ConEx is a single-loop primal-dual type method for functional constrained optimization.
It evolves from the primal-dual methods for solving bilinear saddle point point problems (e.g., \cite{chambolle2011first,chen2014pd,lan2018optrand,Lan2018comm,Lan19}).
Recently Hamedani and Aybat~\cite{aybat2018primal} show that these methods can also 
handle more general functional coupling term. However, as discussed earlier, existing primal-dual methods ~\cite{Nemirovski2005mirror,aybat2018primal} for
general saddle point problems, when applied to functional constrained problems, require the projection
of dual multipliers onto a possibly unknown bounded set in order to ensure the boundedness of the operators, as well as
the proper selection of stepsizes. One distinctive feature of ConEx is to
use value of linearized constraint functions in place of exact function values when defining %the operator of the saddle point problem and 
the extrapolation/momentum step. With this modification, we show that the ConEx method still converges
even though the feasible set of $y$ in problem \eqref{eq:saddle_pb} is unbounded.
In addition, we show that ConEx method is a unified algorithm for
functional constrained optimization in the following sense. First, we establish an explicit rate of convergence of the ConEx method for solving functional constrained stochastic optimization problems where either the objective and/or constraints are given in the form of expectation. Second, we consider the composite constrained optimization problem in which the objective function $f_0$ and/or constraints $f_i, i =1, \dots, m$ can be nonsmooth. Third, we consider the two cases of convex or strongly convex objective, $f_0$. For the strongly convex objective, we 
also establish the rate of convergence to the optimal solution $x^*$.

Before proceeding to the algorithm, we introduce the problem setup in more details. %including some notations for the stochastic setting.
First, we assume that $f_0$ satisfies the following Lipschitz smoothness and nonsmoothness condition for some constants $L_0, H_0 \ge 0$:
\begin{equation}\label{eq:A3}
f_0(x_1)- f_0(x_2)-\inprod{\subgrad{ f_0}(x_2)}{x_1-x_2}\le \tfrac{L_0}{2}\gnorm{x_1-x_2}{}{2} + H_0\gnorm{x_1-x_2}{}{}
\end{equation}
for all $x_1, x_2 \in X$ and for all $f'_0(x_2) \in \pgrad f_0(x_2)$. For constraints, we make a similar assumption as in \eqref{eq:A3}. Moreover, we make an additional assumption that 
the constraint functions are Lipschitz continuous. In particular, we have 
\begin{equation}\label{eq:A2}
f_{i}(x_1) - f_{i}(x_2) - \inprod{\subgrad{f_i}(x_2)}{x_1-x_2} 
\le \tfrac{L_i}{2}\gnorm{x_1-x_2}{}{2} + 
H_{i} \gnorm{x_1-x_2}{}{},
\end{equation}
for all $ x_1,x_2 \in X$ and for all $f'_i(x_2) \in \pgrad f_i(x_2), \ i = 1,\dots, m$, and
\begin{equation} \label{eq:A1} %\tag{\textbf{A1}}
\begin{aligned}
%\gnorm{\grad f_0(x_1) -\grad f_0(x_2)}{\ast}  &\le L_{0} \gnorm{x_1-x_2}{}, &\forall x_1, x_2 \in X,\\
%\gnorm{\grad f_i(x_1) - \grad f_i(x_2) }{\ast} &\le L_i\gnorm{x_1-x_2}{}, &\forall x_1, x_2 \in X, i = 1, \dots, m,\\
f_{i}(x_1) - f_{i}(x_2) %- \inprod{\grad f(x_2)}{x_1-x_2} 
&\le %\tfrac{L_{f,i}}{2}\gnorm{x_1-x_2}{}^2 + 
M_{f,i} \gnorm{x_1-x_2}{}{}, &\forall x_1,x_2 \in X, i =  1, \dots, m,\\
\chi_i(x_1) - \chi_i(x_2) %- \inprod{\grad \chi(x_2)}{x_1-x_2} 
&\le %\tfrac{L_{\chi,i}}{2}\gnorm{x_1-x_2}{}^2 + 
M_{\chi,i} \gnorm{x_1-x_2}{}{}, &\forall x_1,x_2 \in X, i =1, \dots, m.
\end{aligned}
\end{equation}
Note that the Lipschitz-continuity assumption in \eqref{eq:A1} is common in the literature when $f_i, i \in [m]$, are nonsmooth functions. If $f_i, i \in [m]$, are Lipschitz smooth then 
their gradients are bounded due to the compactness of $X$. Hence \eqref{eq:A1} is not a strong assumption for the given setting. Also note that due to definition of subgradient for convex function defined in Section \ref{sec-notn}, we have $\gnorm{f'_i(\cdot)}{\ast}{} \le M_{f,i}$ which implies $\abs{{f'_i(x_2)^T}\paran{x_1-x_2}} \le \gnorm{f_i'(x_2)}{\ast}{} \gnorm{x_1-x_2}{}{} \le M_{f,i} \gnorm{x_1-x_2}{}{}$.
Using this relation, \eqref{eq:A2} and \eqref{eq:A1}, we have the following four relations:
\begin{equation}\label{eq:lipschitz_relations}
\begin{split}
\gnorm{f(x_1)-f(x_2)}{2}{} &\le M_f\gnorm{x_1-x_2}{}{},\\
\gnorm{\chi(x_1)-\chi(x_2)}{2}{} &\le M_\chi\gnorm{x_1-x_2}{}{},\\
\gnorm{f(x_1)-f(x_2) - \subgrad{f}(x_2)^T(x_1-x_2)}{2}{} &\le \tfrac{L_f}{2}\gnorm{x_1-x_2}{}{2} + H_f\gnorm{x_1-x_2}{}{},\\
\gnorm{f'(x_2)^T(x_1-x_2)}{2}{} &\le M_{f} \gnorm{x_1-x_2}{}{},
\end{split}
\end{equation}
for all $x_1, x_2 \in X$.
Here $\subgrad{f}(\cdot) := [\subgrad{f}_1(\cdot), \dots, \subgrad{f}_m(\cdot)] \in \Rbb^{n\times m}$ and constants $M_f, M_\chi, H_f$ and $L_f$ are defined as 
\begin{equation}\label{eq:define_M_L}
\begin{tabular}{DD}
M_f & (\tsum_{i =1}^{m}M_{f,i}^2)^{1/2}, &M_\chi & (\tsum_{i = 1}^mM_{\chi,i}^2)^{1/2},\\[2mm]
H_f & (\tsum_{i =1}^mH_{i}^2)^{1/2}, &L_f & (\tsum_{i =1}^mL_{i}^2)^{1/2}.
\end{tabular}
\end{equation}
We denote $\alpha = (\alpha_1, \dots, \alpha_m)^T$ as the vector of moduli of strong convexity for $\chi_i, i \in[m]$, and $\alpha_{0}$ as the modulus of strong convexity for $\chi_0$. We say that convex problem \eqref{main-prob} is a smooth composite if \eqref{eq:A2} is satisfied with $H_i = 0$ for all $i = 1, \dots, m$ and \eqref{eq:A3} is satisfied with $H_0 =0$. Otherwise, \eqref{main-prob} is a nonsmooth problem. To be succinct, problem \eqref{main-prob} is composite smooth if $H_f = H_ 0 = 0$, otherwise it is a nonsmooth problem.

We assume that we can access the first-order information of functions $f_0, f_i$ and zeroth-order information of function $f_i$ using a stochastic oracle (SO). In particular, given $x \in X$, 
SO outputs $G_{0}(x, \xi), G_i(x, \xi)$, and $F(x, \xi)$ such that
\begin{equation}\label{eq:stochastic_oracle}%\tag{\textbf{A2}}
\begin{split}
\Ebb \bracket*{G_{0}(x,\xi)} &= \subgrad{ f_0}(x),\\
\Ebb \bracket*{G_i(x, \xi)} &= \subgrad{ f_i}(x), \tab i =1, \dots, m,\\
\Ebb[F(x, \xi)] &= f(x),\\
\Ebb \bracket*{\gnorm{G_0(x,\xi)-\subgrad{ f_0}(x) }{\ast}{2} } &\le \sigma_0^2,\\
\Ebb \bracket*{\gnorm{ G_i(x, \xi) - \subgrad{ f_i}(x)}{\ast}{2}} &\le \sigma_i^2, \tab i =1, \dots, m,\\
\Ebb[\gnorm{F(x, \xi)-f(x)}{2}{2}] &\le \sigma_f^2,
\end{split}
\end{equation} 
where $\xi$ is a random variable which models the source of uncertainty and is independent of the search point $x$. Note that
the last relation of \eqref{eq:stochastic_oracle} is satisfied if we have individual stochastic oracles $F_i(x, \xi)$ 
such that $\Ebb[\paran{F_i(x, \xi)-f_{i}(x)}^2] \le \sigma_{f, i}^2$. In particular, we can set $\sigma_f^2 = \tsum_{i= 1}^{m} \sigma_{f, i}^2$.
We call $G_i,  i = 0, \dots, m$, as stochastic subgradients of functions $f_i, i = 0, \dots, m$ at point $x$, respectively. We use stochastic subgradients $G_i(x_t, \xi_t),\ i = 0, \dots, m$, in 
the $t$-th iteration of the ConEx method where $\xi_t$ is a realization of random variable $\xi$ which is independent of the search point $x_t$. {We denote by $\sigma := [\sigma_1, \dots, \sigma_m]^T$ a vector of standard deviation of stochastic subgradients of  individual constraints.}

We denote $\ell^{t-1}_f(x_t)$ a linear approximation of $f(\cdot)$ at point $x_{t}$ with
\[\ell^{t-1}_f(x_t) := f(x_{t-1}) + \subgrad{ f}(x_{t-1})^T(x_t-x_{t-1}),\] 
where $\subgrad{ f}(x_{t-1}) = [\subgrad{ f}_{1}(x_{t-1}), \dots, \subgrad{ f}_{m}(x_{t-1})]$ as defined earlier. For ease of notation, we denote $\ell^{t-1}_f(x_t)$ as $ \ell_f(x_t)$. We can do this, since for all $t$, we approximate $f(x_t)$ with linear function approximation taken at $x_{t-1}$.
We use a stochastic version of $\ell_f$ in our algorithm, which is denoted as $\ell_F$. In particular, we have 
\[\ell_F(x_t) := F(x_{t-1}, \wb{\xi}_{t-1})+ \Gbf(x_{t-1}, \wb{\xi}_{t-1})^T(x_t-x_{t-1}),\]
where 
$\Gbf(x_{t-1}, \wb{\xi}_{t-1}) := [G_{1}(x_{t-1}, \wb{\xi}_{t-1}), \dots, {G}_{m}(x_{t-1}, \wb{\xi}_{t-1})] \in \Rbb^{n \times m}$. Here, we used $\wb{\xi}_t$ as an independent (of $\xi_t$) realization of random variable $\xi$. In other words, $G_i(x_t, \wb{\xi}_t)$ and $G_i(x_t, \xi_t)$ are conditionally independent estimates of $\subgrad{ f}_{i}(x_t)$ for $i = 1, \dots, m$ under the condition that $x_t$ is fixed. As we show later, independent samples of $\xi$ are required to show that $\ell_{F}(x_t)$ is an unbiased estimator of $\ell_f(x_t)$. % We denote $\ell_F(x_t)$ as $\ell_{F,t}$.

We are now ready to formally describe the constraint extrapolation method (see Algorithm~\ref{alg-without-guess}).%\comment[id=qd]{Line 1 of Algo1: $F(x_{-1})$ is a typo?}
\begin{algorithm}[h]
	\caption{\textbf{Con}straint \textbf{Ex}trapolation (ConEx) Method}
	\label{alg-without-guess}
	\begin{algorithmic}[1]
		\REQUIRE$(x_0, y_0), \{\gamma_t,\tau_t, \eta_t, \theta_t\}_{t\ge 0}, T.$
		\STATE $x_{-1} \gets x_0$, {$\ell_F(x_{0})\gets F(x_0, \wb{\xi}_0)$ and $\ell_{F}(x_{-1}) \gets \ell_{F}(x_{0})$.}%is not required since x_0 = x_{-1}
		\FOR {$ t = 0, \dots, T-1$}
		\STATE $s_t \gets (1 + \theta_t)\bracket{\chi(x_t) + \ell_F(x_t) }  - \theta_t \bracket{\chi(x_{t-1}) +\ell_F(x_{t-1}) }.$
		\STATE {$y_{t+1} \gets \argmin_{y \ge \zero} \inprod{-s_t}{y} + \tfrac{\tau_t}{2}\gnorm{y-y_t}{2}{2}$.}
		\STATE $x_{t+1} \gets \prox\paran[\big]{y_{t+1}, G_0(x_t, \xi_t)+\tsum_{i \in [m]}G_i(x_t, \xi_t) y_{t+1}^\br{i}, x_t, \eta_t}.$
		\ENDFOR
		\RETURN $\wb{x}_T = \paran[\big]{\tsum_{t = 0}^{T-1}\gamma_t}^{-1}\sum\limits_{t= 0}^{T-1} \gamma_tx_{t+1}.$
	\end{algorithmic}
\end{algorithm}
As mentioned earlier, the $\ell_F(x_t)$ term in Line 3 of Algorithm \ref{alg-without-guess} is an unbiased estimator of $\ell_f(x_t)$. Moreover, the term $\chi(x_t)+\ell_f(x_t)$ is an approximation to $\chi(x_t)+f(x_t) = \psi(x_t)$. Essentially, Line 3 represents a stochastic approximation for the term $\psi(x_t) + \theta_t(\psi(x_t)-\psi(x_{t-1}))$ which is an extrapolation of the constraints, hence justifying the name of the algorithm. Line 4 is the standard $\prox$ operator {which has a closed-form solution: $y_{t+1} = \bracket[\big]{y_t +\tfrac{1}{\tau_t}s_t}_+$}. 
Line 5 also uses a prox operator defined in \eqref{eq:W-prox}. %which uses Bregman divergence $W$ instead of standard Euclidean norm. 
The final output of the algorithm in Line 7 is the weighted average of all primal iterates generated. 
If we choose standard deviations $\sigma_f = \sigma_0 = \sigma_i = 0$ for $i = 1, \dots, m$ then we recover the deterministic gradients and function evaluation. Henceforth, we assume 
general non-negative values for all such standard deviations and provide a combined analysis for these settings. Later, we substitute appropriate values of standard deviations to finish the analysis 
for the following three different cases.
\begin{enumerate}
	\item [a)] Deterministic setting where both the objective and constraints are deterministic. Here $\sigma_0 = \sigma_i = \sigma_f = 0$ for all $i \in [m]$.
	\item [b)] Semi-stochastic setting where the constraints are deterministic but the objective is stochastic. Here, $\sigma_f = \sigma_i= 0$ for all $i \in [m]$. However, $\sigma_0 \ge 0$ can take arbitrary values.
	\item [c)] Fully-stochastic setting where both function and gradient evaluations are stochastic. Here, all $\sigma_f, \sigma_0, \sigma_i \ge 0$ can take arbitrary values.
\end{enumerate}

Below, we specify a stepsize policy and state
the convergence properties of Algorithm \ref{alg-without-guess} for solving problem \eqref{main-prob} in the strongly convex setting where $\chi_0$ is $\alpha_{0}$-strongly convex. {Note that in the stochastic case, stepsize $\tau_t$ depends on the total number of iterations $T$ which needs to be fixed beforehand. Similar policies are used in the stochastic approximation literature \cite{Lan19}.}
The proof of this theorem is involved and will be deferred to Section~\ref{subsec:conv_analysis}.
%\comment{Mentione that $T$ is given for the stepsize policy}
\begin{theorem}\label{cor:step_size_strong_cvx}
	Suppose \eqref{eq:A3}, \eqref{eq:A2}, \eqref{eq:A1} and \eqref{eq:stochastic_oracle} are satisfied. Let $B\ge 1$ be a constant, $t_0 := \tfrac{4(L_0 + BL_f)}{\alpha_0}+2$, 
	$\MM:= \max\{2M_f, M_\chi+M_f\}$, and $\sigma_{X, f} := (\sigma_{f}^2 + D_X^2\gnorm{\sigma}{2}{2})^{1/2}$.
	Set $y_0 = \zero$, $T \ge 1$  and $\braces*{\gamma_t, \theta_t, \eta_t, \tau_t}$  in Algorithm~\ref{alg-without-guess} according to the following:
	\begin{equation}\label{eq:step_size}
	\begin{tabular}{EE}
	\gamma_t & t +t_0 + 2, &\eta_t &\tfrac{\alpha_{0}\paran*{t+t_0+1}}{2}, \\
	\tau_t & \tfrac{1}{t + 1}\max\braces{\tfrac{32\MM^2}{\alpha_{0}}, \tfrac{384\gnorm{\sigma}{2}{2}T}{\alpha_0}, \tfrac{\sigma_{X,f} T^{3/2}}{B(t_0+2)^{1/2}}},  &\theta_t &\tfrac{t+t_0+1}{t+t_0+2}. 
	\end{tabular}
	\end{equation}
	Then, we have %\comment[id=qd]{where is $D_X$ defined}
	\begin{equation}\label{eq:bounding_optimality_cor}
	\begin{split}
	\Ebb [\psi_0(\wb{x}_T) -\psi_0(x^*)]&\le \tfrac{\alpha_0(t_0+1)(t_0+2)D_X^2}{T^2} + \tfrac{12B\sigma_{X,f}(t_0+1)(t_0+2)^{1/2}}{T^{3/2}} + \tfrac{16(\zeta^2+H_0^2)}{\alpha_0T} + \tfrac{8B(t_0+2)^{1/2}\sigma_{X,f}}{T^{1/2}}.
	\end{split}
	\end{equation} 
	and
	\begin{align}
	\Ebb \gnorm*{\relu{\psi(\wb{x}_T)}}{2}{} &\le  \tfrac{192(t_0+2)(\gnorm{y^*}{2}{}+1)^2\MM^2}{\alpha_0T^2}+ \tfrac{\alpha_{0}(t_0+1)(t_0+2)D_X^2}{T^2} +  \tfrac{13B\sigma_{X,f}(t_0+1)(t_0+2)^{1/2}}{T^{3/2}}  \nonumber\\
	&\quad +\tfrac{16\paran{\zeta^2+ \HH_*^2+144(t_0+2)(\gnorm{y^*}{2}{}+1)^2\gnorm{\sigma}{2}{2} } }{\alpha_{0}T} \nonumber\\
	&\quad+\braces[\big]{\tfrac{6(t_0+2)^{1/2}(\gnorm{y^*}{2}{}+1)^2\sigma_{X,f} }{B}  +\tfrac{26B(t_0+2)^{1/2}\sigma_{X, f}}{3}}\tfrac{1}{T^{1/2}}, \label{eq:bounding_infeasibility_cor}
	\end{align}
	where 
	\begin{align*}
	\HH_* &:= H_0+(\gnorm{y^*}{2}{}+1)H_f + \tfrac{L_fD_X[\gnorm{y^*}{2}{}+1-B]_+}{2},\\
	\zeta &:= 
	2e\braces[\Big]{\big[ \sigma_0^2 + 12(t_0+3)\gnorm{\sigma}{2}{2} \gnorm{y^*}{2}{2} + 96(t_0+2)B^2\gnorm{\sigma}{2}{2} + \tfrac{\HH_*^2}{2} + \tfrac{3\alpha_0B\sigma_{X, f}(t_0+2)^{3/2}}{2} \big]}^{1/2}.
	\end{align*}
	Moreover, we obtain the last iterate convergence
	\begin{align}
	\Ebb[W( x^*, X_T)] &\le \tfrac{192(t_0+2)(\gnorm{y^*}{2}{}+1)^2\MM^2}{\alpha_0^2T^2}+ \tfrac{(t_0+1)(t_0+2)D_X^2}{T^2} +  \tfrac{12B\sigma_{X,f}(t_0+1)(t_0+2)^{1/2}}{\alpha_0T^{3/2}} \nonumber \\
	&\quad +\tfrac{16\paran{\zeta^2+ \HH_*^2+144(t_0+2)(\gnorm{y^*}{2}{}+1)^2\gnorm{\sigma}{2}{2} }}{\alpha_{0}^2T} \nonumber\\
	&\quad +\tfrac{(t_0+2)^{1/2}\gnorm{y^*}{2}{2}\sigma_{X,f}}{B\alpha_0} \tfrac{1}{T^{1/2}} +\tfrac{8B(t_0+2)^{1/2}\sigma_{X, f}}{\alpha_0T^{1/2}}. \label{eq:bounding_last_iterate_cor}
	\end{align}
\end{theorem}

An immediate corollary of the above theorem is the following:
\begin{corollary}\label{cor:iteration_counts}
	We obtain an $(\eps, \eps)$-optimal solution of problem \eqref{main-prob} in $T_\eps$ iterations, where
	\begin{equation}\label{eq:bounding_T_cor}
	\begin{split}
	 T_\eps = &\max \Big\{
	\paran[\Big]{\tfrac{5\alpha_0(t_0+2)(t_0+1)D_X^2}{\eps} +\tfrac{960(t_0+2)(\gnorm{y^*}{2}{}+1)^2\MM^2}{\alpha_0\eps} }^{1/2}, \paran[\big]{\tfrac{65B\sigma_{X, f}(t_0+2)^{3/2}}{\eps}}^{2/3},\\
	&\tab \tfrac{80(\zeta^2+ \HH_*^2 +144(t_0+2)(\gnorm{y^*}{2}{}+1)^2\gnorm{\sigma}{2}{2} )}{\alpha_{0}\eps}, \paran[\big]{\tfrac{30(\gnorm{y^*}{2}{}+1)^2\sigma_{X,f}}{B}}\tfrac{t_0+2}{\eps^2}, \paran[\big]{\tfrac{130B\sigma_{X, f}}{3}}^2\tfrac{t_0+2}{\eps^2}
	\Big\}.
	\end{split}
	\end{equation}
	Moreover, we obtain $\Ebb[W(x^*, x_T)] \le \eps$ in at most
	\begin{equation}\label{eq:bounding_T_last_iterate}
	\begin{split}
	&\max \Big\{
	\paran[\Big]{\tfrac{5(t_0+2)(t_0+1)D_X^2}{\eps} +\tfrac{960(t_0+2)(\gnorm{y^*}{2}{}+1)^2\MM^2}{\alpha_0^2\eps} }^{1/2}, \paran[\big]{\tfrac{60B\sigma_{X, f}(t_0+2)^{3/2}}{\alpha_0\eps}}^{2/3},\\
	&\tab \tfrac{80(\zeta^2+ \HH_*^2+144(t_0+2)(\gnorm{y^*}{2}{}+1)^2\gnorm{\sigma}{2}{2})}{\alpha_{0}^2\eps}, \paran[\big]{\tfrac{5\gnorm{y^*}{2}{2}\sigma_{X,f}}{B\alpha_0}}^2\tfrac{t_0+2}{\eps^2}, \paran[\big]{\tfrac{40B\sigma_{X, f}}{\alpha_0} }^2\tfrac{t_0+2}{\eps^2} %\paran[\big]{\tfrac{40\sqrt{6}\gnorm{\sigma}{2}(\gnorm{y^*}{2}+1)^2}{B\alpha_0^2}}^2\tfrac{t_0+2}{\eps^2}
	\Big\}
	\end{split}
	\end{equation}
	iterations.
\end{corollary}

\begin{proofx}
	Using \eqref{eq:bounding_infeasibility_cor} and \eqref{eq:bounding_T_cor}, we have $\Ebb \gnorm*{\relu{\psi(\wb{x}_T)}}{2}{} \le \tfrac{\eps}{5} +\tfrac{\eps}{5}+\tfrac{\eps}{5}+\tfrac{\eps}{5}+\tfrac{\eps}{5} = \eps.$
	Similarly, using \eqref{eq:bounding_optimality_cor} and \eqref{eq:bounding_T_cor}, it is easy to observe that $\Ebb \bracket*{\psi_0(\wb{x}_T) - \psi_0(x^*)} \le \eps$.
	Using \eqref{eq:bounding_last_iterate_cor} and \eqref{eq:bounding_T_last_iterate}, we have $\Ebb[W(x^*,x_T)] \le \tfrac{\eps}{5} +\tfrac{\eps}{5}+\tfrac{\eps}{5}+\tfrac{\eps}{5}+\tfrac{\eps}{5} = \eps$. Hence we conclude the proof.
%	\qed
\end{proofx}

Theorem~\ref{cor:step_size_strong_cvx} and Corollary~\ref{cor:iteration_counts}
provide unified iteration complexity bounds for solving strongly convex functional constrained optimization problems.
These results will also be used later for solving subproblems arising from the
proximal point method for nonconvex problems in Section~\ref{sec-alg1}.
%The following remark discusses various consequences of these results in more details.
%\begin{remark}[Convergence for composite smooth vs nonsmooth problems] \label{remark_strongly_convex}
Below we derive from \eqref{eq:bounding_T_cor} 
the convergence rate of Algorithm \ref{alg-without-guess} for both 
nonsmooth problems, i.e., either $H_f$ or $H_0$ is strictly positive, 
and (composite) smooth problems, i.e., $H_f = 0, H_0 = 0$.

Let us start with nonsmooth problems for which  \eqref{eq:A3} is satisfied with $H_0 > 0$
or \eqref{eq:A2} is satisfied with $H_{i} > 0$ for at least one $i \in [m]$. 
In this case, we have 
\[
\HH_* = (\gnorm{y^*}{2}{}+1) H_f + H_0 +\tfrac{L_f D_X\bracket{\gnorm{y^*}{2}{}+1-B}_+}{2} > 0
\] 
irrespective of the value of $B$. Then, using \eqref{eq:bounding_T_cor}, we obtain the iteration 
complexity of \[O\Big( \tfrac{1}{\sqrt{\eps}}\big( \tfrac{(L_0+BL_f)D_X}{\sqrt{\alpha_0}} + \tfrac{\sqrt{L_0+BL_f}B\MM}{\alpha_0} \big) + \tfrac{\HH_*^2}{\alpha_0\eps}\Big)\] for
the deterministic case. 
%Note that since $\HH_* > 0$, we have 
%		an $O(1/\eps)$ overall rate of convergence.		%as opposed to $O(1/\sqrt{\eps})$ in the smooth deterministic case. 
		For the semi-stochastic case, the iteration complexity becomes 
		\[O\Big( \tfrac{1}{\sqrt{\eps}}\big( \tfrac{(L_0+BL_f)D_X}{\sqrt{\alpha_0}} + \tfrac{\sqrt{L_0+BL_f}B\MM}{\alpha_0} \big) + \tfrac{(\HH_*^2+\sigma_0^2)}{\alpha_0\eps}\Big).\] 
		Similarly, for the fully-stochastic case, the iteration complexity is given by
		\[O\Big( \tfrac{1}{\sqrt{\eps}}\big( \tfrac{(L_0+BL_f)D_X}{\sqrt{\alpha_0}} + \tfrac{\sqrt{L_0+BL_f}B\MM}{\alpha_0} \big) + \tfrac{(\HH_*^2+\zeta^2)}{\alpha_0\eps} + \tfrac{1}{\eps^2} \big( \tfrac{B^2(L_0+BL_f)(\sigma_0^2+D_X^2\gnorm{\sigma}{2}{2}) }{\alpha_0}  \big)\Big).\] 
		Observe that, due to the built-in acceleration
		scheme of the ConEx method, the Lipschitz constant $L_0$
		will barely impact the convergence since it appears only in the $O(1/\sqrt{\eps})$ term.
		Similarly, the impact of the Lipschitz constant $L_f$ will be minimized for a large enough $B$ (i.e., $B \ge \gnorm{y^*}{2}{}+1$).
		To the best of our knowledge,
		these complexity results with separate impact of Lipschitz constants appear 
		to be new for functional constrained optimization.
		Moreover, the iteration (and sampling) complexity for the fully-stochastic case, i.e.,
		general stochastic constrained problems requiring only bounded second moments on nosies, 
		 has not been obtained before in the literature.

Now let us consider smooth problems for which 
\eqref{eq:A3} and \eqref{eq:A2} are satisfied with $H_0 = 0$ and $H_i = 0$ for all $i = 1,\dots, m$, respectively.
We distinguish two different scenarios depending on whether $B \ge \gnorm{y^*}{2}{}+1$.
First, if $B \ge \gnorm{y^*}{2}{}+1$, then 
$
\HH_* = H_0 + H_f (\gnorm{y^*}{2}{}+1) + L_f D_X\bracket{\gnorm{y^*}{2}{}+1-B}_+/2 = 0
$ 
and the iteration complexity in \eqref{eq:bounding_T_cor} can be simplified as follows.
For the deterministic case, %i.e., $\sigma_0 = \gnorm{\sigma}{2}{} = \sigma_{f} =0$, 
the iteration complexity in \eqref{eq:bounding_T_cor} reduces to
		\begin{equation} \label{smooth_strong1}
		O\Big( \tfrac{1}{\sqrt{\eps}}\big( \tfrac{(L_0+BL_f)D_X}{\sqrt{\alpha_0}} + \tfrac{\sqrt{L_0+BL_f}B\MM}{\alpha_0} \big) \Big).
		\end{equation}
Moreover, the complexity bounds for the semi- and fully-stochastic cases are given by % i.e., $\sigma_0>0$ and $\sigma_i =\sigma_f=0$ for all $i = 1, \dots, m$, % implying $\sigma_{X, f} = (\sigma_{f}^2+D_X^2\gnorm{\sigma}{2}{2})^{1/2} = 0$. 
%Then the iteration complexity is given by
		\begin{equation} \label{smooth_strong2}
		O\Big( \tfrac{1}{\sqrt{\eps}}\big( \tfrac{(L_0+BL_f)D_X}{\sqrt{\alpha_0}} + \tfrac{\sqrt{L_0+BL_f}B\MM}{\alpha_0} \big) + \tfrac{\sigma_0^2}{\alpha_0\eps}\Big),
		\end{equation}
%Similarly, for the fully-stochastic case, the iteration complexity will reduce to
		\begin{equation} \label{smooth_strong3}
		O\Big( \tfrac{1}{\sqrt{\eps}}\big( \tfrac{(L_0+BL_f)D_X}{\sqrt{\alpha_0}} + \tfrac{\sqrt{L_0+BL_f}B\MM}{\alpha_0} \big) + \tfrac{\zeta^2}{\alpha_0\eps} + \tfrac{1}{\eps^2} \big( \tfrac{B^2(L_0+BL_f) (\sigma_0^2+D_X^2\gnorm{\sigma}{2}{2} )}{\alpha_0}  \big) \Big),
		\end{equation}
		respectively, where $\zeta^2 = O(\sigma_0^2 + B^2(L_0+BL_f)\gnorm{\sigma}{2}{2}/\alpha_{0})$.
		It is worth noting that a similar bound to \eqnok{smooth_strong1} has been obtained
		in \cite{aybat2018primal} with a slightly different termination criterion. %\footnote{The infeasibility in  \cite{aybat2018primal} is measured by $y^* \relu{\psi(\wb{x}_T)} $,and hence may vanish for constraints with $y_i^*=0$.}
		On the other hand, the complexity bounds in \eqnok{smooth_strong2}
		and \eqnok{smooth_strong3} for the semi-stochastic and fully-stochastic cases
		seem to be new in the literature.
		
Second, if $B < \gnorm{y^*}{2}{}+1$ for the smooth case, then $\HH_* > 0$ and the ConEx method 
converges at the rate of nonsmooth problems in all these three settings described above.
Hence, the ConEx method still converges albeit at a slower rate without knowing exact bound on $\gnorm{y^*}{2}{}$. 
On the other hand, existing primal-dual methods require correct upper-bound estimation on $\gnorm{y^*}{2}{}$ 
which is used as a bound on $y$ in order to define the projection operator and properly select stepsize.  To obtain a faster convergence rate,
one can possibly perform a line search for the right value of $B$ when specifying {the policy for $\{\gamma_t, \eta_t, \tau_t,\theta_t\}$} in the ConEx method, 
%\comment[id=qd]{Better: To obtain faster rate, one can possibly xxx} 
especially for the deterministic and semi-stochastic cases where the constraint violations $\gnorm{\relu{\psi(\cdot)}}{2}{}$ can be measured precisely. 

It is worth mentioning that for the complexity results discussed above, we do not require 
the constraints $\psi_i, \ i =1,\dots, m$, to be strongly convex. 
From \eqref{eq:step_size}, we can see that $\alpha_{0} > 0$ is enough to ensure
		the selection of stepsize policy which yields accelerated convergence rates. In particular, if $\alpha_i = 0$ for all $i \in [m]$ (implying $\psi_i$'s are merely convex functions) then $\eta_t$ in relation \eqref{eq:int_rel154} is required to satisfy the following more stringent relation: $\gamma_{t}\eta_t \le \gamma_{t-1}(\eta_{t-1}+\alpha_{0})$. Note that our stepsize policy already satisfies this relation. Hence Algorithm \ref{alg-without-guess} exhibits accelerated convergence rates even if the constraints are merely convex.

Now we provide another theorem which states the stepsize policy and 
the resulting convergence properties of the ConEx method for solving problem \eqref{main-prob} without any strong convexity assumptions. 
The proof of this result can be found in Section~\ref{subsec:conv_analysis}.

\begin{theorem} \label{thm:convergence_convex}
	Suppose \eqref{eq:A3}, \eqref{eq:A2}, \eqref{eq:A1} and \eqref{eq:stochastic_oracle} are satisfied. Let $B \ge 1$ be a given constant,
		 $\MM$, $\sigma_{X, f}$ and $\HH_*$ be defined as in Theorem \ref{cor:step_size_strong_cvx}.
Set $y_0 =\zero$ and $\{\gamma_t, \theta_t, \eta_t, \tau_t\}$ in Algorithm~\ref{alg-without-guess} according to the following:\\
	\begin{equation}\label{eq:step_size_conv}
	\begin{tabular}{EE}
	\gamma_t & 1, & \eta_t & L_0+BL_f+\eta,\\
	\theta_t & 1, & \tau_t & \tau,
	\end{tabular}
	\end{equation}
	where 
	\begin{align*}
	\eta &:= \tfrac{\sqrt{2T[\HH_*^2+\sigma_0^2 + 48B^2\gnorm{\sigma}{2}{2} ] } }{D_X} {+}\tfrac{6B\max\{ \MM, 4\gnorm{\sigma}{2}{} \} }{D_X} ,\\
	\tau &:= \max\braces[\big]{\tfrac{\sqrt{96T} \sigma_{X, f}}{B} , \tfrac{2D_X\max\{\MM, 4\gnorm{\sigma}{2}{} \}}{B} }.
	\end{align*}
	Then, we have
	\begin{align}
	\Ebb &[\psi_0(\wb{x}_T) -\psi_0(x^*)] \le \tfrac{(L_0+BL_f)D_X^2+\max\{6\MM, 24\gnorm{\sigma}{2}{} \}BD_X}{T} +   \tfrac{\sqrt{2}(\zeta^2+H_0^2)D_X}{\sqrt{T( \HH_*^2 + \sigma_0^2+48B^2\gnorm{\sigma}{2}{2}) } +3B\MM} + \tfrac{\sqrt{3}B\sigma_{X, f}}{\sqrt{2T}}  \label{eq:bounding_opt_convex}
	\end{align}
	and
	\begin{align}
	\Ebb[\gnorm{\relu{\psi(\wb{x}_T)}}{2}{} ] &\le\tfrac{ (L_0+BL_f)D_X^2+\max\{6\MM, 24\gnorm{\sigma}{2}{} \}D_X\big( B+ \tfrac{(\gnorm{y^*}{2}{}+1)^2}{B} \big) }{T} \nonumber \\
	& \quad
	+ \tfrac{1}{\sqrt{T}} \Big( \big(\tfrac{12\sqrt{6}(\gnorm{y^*}{2}{}+1)^2}{B} + \tfrac{13B}{4\sqrt{6}} \big) \sigma_{X, f}
	 \Big)\nonumber\\
	&\quad + +\tfrac{{\sqrt{2}D_X}\sqrt{\HH_*^2+\sigma_0^2+48B^2\gnorm{\sigma}{2}{2} }}{\sqrt{T}} + \tfrac{\sqrt{2}D_X(\zeta^2+\HH_*^2)}{\sqrt{T(\HH_*^2+\sigma_0^2+48B^2\gnorm{\sigma}{2}{2}) } +3B\MM} , \label{eq:bounding_infeas_convex}
	\end{align}
	where
	\begin{equation*}
	\zeta := 2e \paran[\big]{\sigma_0^2+ \gnorm{\sigma}{2}{2}(14\gnorm{y^*}{2}{2} +123B^2)+2\sqrt{3} \gnorm{\sigma}{2}{}(2B\HH_*+B\sigma_0)}^{1/2}.
	\end{equation*}
	As a consequence, the number of iterations performed by Algorithm~\ref{alg-without-guess} to find an $(\eps, \eps)$-optimal solution of problem \eqref{main-prob} can be bounded by
	\begin{equation}\label{eq:T_convex}
	\begin{split}
	\max\Big\{ &\tfrac{3(L_0+BL_f)D_X^2 + \max\{36\MM, 144\gnorm{\sigma}{2}{}\} (\gnorm{y^*}{2}{}+1) D_X}{\eps}, \tfrac{\sigma_{X, f}^2}{\eps^2} \big(\tfrac{36\sqrt{6}(\gnorm{y^*}{2}{}+1)^2}{B} + \tfrac{13\sqrt{3}B}{4\sqrt{2}}\big)^2,  \\
	&\tfrac{18}{\eps^2}  \big(D_X\sqrt{\HH_*^2+\sigma_0^2+48B^2\gnorm{\sigma}{2}{2}} + \tfrac{D_X(\zeta^2+\HH_*^2)}{\sqrt{\HH_*^2+\sigma_0^2+48B^2\gnorm{\sigma}{2}{2}}}\big)^2 \Big\}.
	\end{split}
	\end{equation}
\end{theorem} 

Theorem~\ref{thm:convergence_convex} provides unified iteration complexity bounds for solving convex functional constrained optimization problems.
%Below we add a few remarks about the iteration complexity result obtained in \eqref{eq:T_convex}.
Below we derive from \eqref{eq:T_convex} 
the convergence rate of Algorithm \ref{alg-without-guess} for solving both 
nonsmooth problems, i.e., either $H_f$ or $H_0$ is strictly positive, 
and (composite) smooth problems, i.e., $H_f = 0, H_0 = 0$.

Let us start with the more general nonsmooth problems. Since $H_i > 0$ for some $i = 0, \ldots, m$,
we have $\HH_* > 0$. Then,
the complexity bound in \eqref{eq:T_convex} 
for the deterministic, semi-stochastic and fully-stochastic cases, respectively,
will reduce to  
\begin{align}
	&O\big( \tfrac{L_0+BD_X(L_fD_X+ \MM)}{\eps} + \tfrac{D_X^2\HH_*^2}{\eps^2}\big),  \nonumber \\
	&O\big( \tfrac{L_0+BD_X(L_fD_X+ \MM)}{\eps} + \tfrac{D_X^2(\HH_*^2+\sigma_0^2)}{\eps^2} \big),\nonumber\\ \intertext{and} 
	&O\big( \tfrac{L_0+BD_X(L_fD_X+ \MM)}{\eps} + \tfrac{B^2(\sigma_f^2+D_X^2\gnorm{\sigma}{2}{2} ) + D_X^2(\sigma_0^2+\HH_*^2)}{\eps^2} \big). \label{cvx_nonsmooth3}
\end{align}
Similarly to the strongly convex case,
the separate impact of the Lipschitz constants ($L_0$ and $L_f$) on
these complexity bounds have not been obtained before. Moreover,
the iteration (and sampling) complexity for
the fully-stochastic case, i.e.,
general stochastic constrained problems requiring only bounded second moments on nosies, 
appears to be new in the literature.

Now let us consider smooth problems for which $H_f = H_0 = 0$.
We distinguish two different scenarios depending on whether $B \ge \gnorm{y^*}{2}{}+1$.
First, if $B \ge \gnorm{y^*}{2}{}+1$, then $\HH_* = 0$ and the complexity bound in \eqref{eq:T_convex} 
for the deterministic, semi-stochastic and fully-stochastic cases, respectively,
will reduce to  
\begin{align}
&O\big( \tfrac{L_0+BD_X(L_fD_X+ \MM)}{\eps} \big), \label{cvx_bnd1} \\
&O\big( \tfrac{L_0+BD_X(L_fD_X+ \MM)}{\eps} + \tfrac{\sigma_0^2D_X^2}{\eps^2} \big), \label{cvx_bnd2}\\
\intertext{and}
&O\big( \tfrac{L_0+BD_X(L_fD_X+ \MM)}{\eps} + \tfrac{B^2(\sigma_f^2+D_X^2\gnorm{\sigma}{2}{2}) + D_X^2\sigma_0^2}{\eps^2} \big), \label{cvx_bnd3}
\end{align}
where last bound is obtained
from \eqref{eq:T_convex}   by noting that $\zeta^2 =O(\sigma_0^2+ 48B^2\gnorm{\sigma}{2}{2})$ and replacing $\sigma_{X, f}^2 = \sigma_{f}^2+D_X^2\gnorm{\sigma}{2}{2}$.
Note that similar bound as in \eqref{cvx_bnd1} has been obtained before
by using more complicated algorithms (e.g., penalty method) or different criterion.
On the other hand the complexity bounds
in \eqref{cvx_bnd2} and \eqref{cvx_bnd3} appear to be new in the literature.
Second,  if $B < \gnorm{y^*}{2}{}+1$, then $\HH_* > 0$ and as a result,
the ConEx method 
still converges but at the rate of nonsmooth problems in all these three settings described above.

It should be noted that,
different from the strongly convex case (c.f. \eqref{eq:step_size}), the stepsize scheme in \eqref{eq:step_size_conv} 
depends on $\HH_*$, implying that we need to estimate whether $B > \gnorm{y*}{2}{} +1$. However,
we can replace $\HH_*$ in the definition of $\eta$ by $\HH_B := H_0 + BH_f$. In this way, 
similar complexity bounds will be obtained for most cases, including
nonsmooth deterministic,  nonsmooth semi-stochastic, nonsmooth fully-stochastic, as well as smooth semi-stochastic and 
smooth fully-stochastic problems. In particular, with this modification the last term 
 in \eqref{eq:T_convex} will change to
\begin{equation*}%\label{eq:int_rel169}
		\tfrac{18}{\eps^2}  \Big(D_X\sqrt{\HH_B^2+\sigma_0^2+48B^2\gnorm{\sigma}{2}{2} } + \tfrac{D_X(\zeta^2+\HH_*^2)}{\sqrt{\HH_B^2+\sigma_0^2+48B^2\gnorm{\sigma}{2}{2} }}\Big)^2.
\end{equation*}
The only exception that this modification would not work is for smooth deterministic problems. In this case, since
$\HH_B = 0$ but $\HH_* > 0$, 
the stepsize scheme \eqref{eq:step_size_conv} set according to replacing $\HH_*$ by $\HH_B$ does not yield 
convergence. In particular, the last term in the infeasibility bound \eqref{eq:bounding_infeas_convex} would change to $\HH_*^2/(\sqrt{T}\HH_B+B\MM)$ which is a constant when $\HH_B = 0$. 
One possible solution for this is to artificially set $\HH_B > 0$ in the definition of $\eta$ to be some large positive number and forgo of the faster convergence of $O(1/\eps)$. 
After this change, we would obtain a convergence rate of $O(1/\eps^2)$. 	An alternative approach would be to design
a line search procedure on $\HH_B$ for the right value of $\HH_*$,
since there exists a verifiable condition based on the constraint violation $\gnorm{\relu{\psi(\cdot)}}{2}{}$.

\subsection{Convergence analysis of the ConEx method} \label{subsec:conv_analysis}
In this section, we provide a combined analysis of Theorem \ref{thm:convergence_convex} and Theorem \ref{cor:step_size_strong_cvx}. 
Note that Algorithm~\ref{alg-without-guess} %\comment[id=qd]{use Algorithm$\sim$1 to prevent breaking lines} 
is essentially a dual type method. In order to analyze this algorithm, we define a \emph{primal-dual gap function} for the equivalent saddle point problem \eqref{eq:saddle_pb}. In particular, given a pair of feasible solutions $z = (x, y)$ and $\wb{z} = (\wb{x}, \wb{y})$ of \eqref{eq:saddle_pb}, we define the {primal-dual gap function} $Q(z, \wb{z})$ as
\begin{equation}\label{eq:saddle_conv1}Q(z, \wb{z}) := \LL(x, \wb{y}) - \LL(\wb{x}, y). \end{equation}
One can easily see from \eqref{eq:saddle_def} that $Q(z, z^*) \ge 0$ and $Q(z^*, z) \le 0$ for all feasible $z$. 
We use the gap function of the  saddle point formulation \eqref{eq:saddle_pb} to bound the optimality and infeasibility of the convex problem \eqref{main-prob} separately, in terms of Definition  \ref{define-apprx-opt}. 
We first develop an important upper-bound on the gap function in terms of primal, dual variables and randomness. This bound holds for all non-negative $\gamma_{t}, \eta_t$ and $\tau_t$. The precise statement is provided in  Lemma~\ref{thm:linearized_conex_1}.

The following technical result provides a simple form of the Three-point theorem (see, e.g., Lemma 3.5 of \cite{Lan19}) and will be used in the proof of  Lemma \ref{thm:linearized_conex_1}.
\begin{lemma}
	\label{lem:3-point}
	Assume that $g: S\to \Rbb$ %\comment[id=qd]{$g: S\rightarrow \Rbb$?}
	 satisfies
	\begin{equation} \label{eq:strong_conv_g}
	g(y) \ge g(x) + \inprod{g'(x)}{y-x} + \mu W(y,x) , \tab \forall x, y \in S
	\end{equation}
	for some $\mu \ge 0$, where $S$ is closed convex set in $\Rbb^n$. If \[ \wb{x} = \argmin_{x \in S} \braces{g(x) + W(x, \wt{x})},\]
	then 
	\[ g(\wb{x}) + W(\wb{x}, \wt{x}) + (\mu+1)W(x,\wb{x}) \le g(x) + W(x, \wt{x}), \tab \forall x \in S. \]
\end{lemma}
\begin{proofx}
	It follows from the definition of $W$ that $W(x, \wt{x}) = W(\wb{x}, \wt{x}) + \inprod{\grad W(\wb{x}, \wt{x})}{x-\wb{x}} + W(x, \wb{x})$. Using this relation, \eqref{eq:strong_conv_g} and 
	the optimality condition for $\wb{x}$, we have
	\begin{align*}
	g(x) + W(x, \wt{x}) &= g(x) + [W(\wb{x}, \wt{x}) + \inprod{\grad W(\wb{x}, \wt{x})}{x-\wb{x}} + W(x, \wb{x})] \\
	&\ge g(\wb{x}) + \inprod{g'(\wb{x})}{x-\wb{x}} + \mu W(x,\wb{x}) + [W(\wb{x}, \wt{x}) + \inprod{\grad W(\wb{x}, \wt{x})}{x-\wb{x}} + W(x, \wb{x})] \\
	&\ge g(\wb{x}) + W(\wb{x}, \wt{x}) + (\mu+1)W(x, \wb{x}).
	\end{align*}
	Hence we conclude the proof.
\end{proofx}

\begin{lemma}\label{thm:linearized_conex_1}
	Suppose \eqref{eq:A3}, \eqref{eq:A2}, \eqref{eq:A1} and \eqref{eq:stochastic_oracle} are satisfied. % with $L_{f, i} = L_{\chi,i} = 0$ for all $i \in [m]$. 
	Let $B \ge 0$ be a constant and assume that $\{\gamma_{t}, \eta_t, \tau_t, \theta_t\}$ is a non-negative sequence  satisfying {for all $t \ge 1$, }
	\begin{subequations}\label{eq:int_rel154}
	\begin{align}
	\gamma_t\theta_t &= \gamma_{t-1},\label{eq:int_rel154-a}\\
	\gamma_t\tau_t &\le \gamma_{t-1}\tau_{t-1},\label{eq:int_rel154-b}\\
	\gamma_t\eta_t &\le \gamma_{t-1}(\eta_{t-1}+\alpha_{0,t-1}),\label{eq:int_rel154-c}
	\end{align}
	\end{subequations}
	%{for all $t \ge 1$ and the last iteration $T\ (T\ge 1)$, }%\hspace{2cm }
	\begin{subequations}\label{eq:int_rel155-1}
	\begin{align}
	\theta_t(M_f+M_\chi)^2  &\le \tfrac{\tau_t(\eta_{t-1}-L_0-BL_f)}{12},\label{eq:int_rel155-1-a}\\
	(M_f+M_\chi)^2 &\le \tfrac{\tau_t(\eta_t-L_0-BL_f)}{12},\label{eq:int_rel155-1-b}\\
	(2M_f)^2\tfrac{1}{\theta_t} &\le \tfrac{{\tau_t(\eta_{t-1}-L_0-BL_f)}%{\tau_t(\eta_{t-2}-L_0-BL_f)}
	}{12}, \label{eq:int_rel155-2-b}
	\end{align}
	\end{subequations}
	{and, for all $ t \ge 2$,}% and the  last iteration $T\ (T \ge 2)$,}
	%\begin{subequations}
		\begin{align}
			(2M_f)^2\tfrac{\theta_t}{\theta_{t-1}} \le \tfrac{\tau_t(\eta_{t-2}-L_0-BL_f)}{12}, \label{eq:int_rel155-2}%\\
		\end{align}
	%\end{subequations}
	where $\alpha_{0,t} := \alpha_{0}+\alpha^Ty_{t+1}$ and $M_f, M_\chi, L_f$ are constants as defined in \eqref{eq:define_M_L}.
	Then, for all $t \ge 0$ and $z \in \{(x,y): x\in X, y \ge \zero\}$, we have
	\begin{align}
	&{\tsum_{i = 0}^t} \gamma_i Q(z_{i+1}, z) + {\tsum_{i =0}^t}\gamma_i[\inprod{\delta^G_i}{x_i-x} - \inprod{\delta^F_{i+1}}{y_{i+1}-y} ] \nonumber\\
	&\le \gamma_0\eta_0W(x,x_0) -\gamma_t(\eta_t+\alpha_{0, t})W(x, x_{t+1}) + \tfrac{\gamma_0\tau_0}{2}\gnorm{y-y_0}{2}{2} -\tfrac{\gamma_t\tau_t}{12}\gnorm{y-y_{t+1}}{2}{2} \nonumber \\
	&\quad+\tsum_{ i =0}^t \tfrac{2\gamma_i}{\eta_i-L_0-BL_f}\bracket[\big]{\gnorm{\delta_i^G}{*}{2} +(H_0+H_f\gnorm{y}{2}{}+ \tfrac{L_fD_X}{2}\bracket{\gnorm{y}{2}{}-B}_+)^2} \nonumber\\
	&\quad+ \tsum_{ i =1}^t\tfrac{3\gamma_i\theta_i^2}{2\tau_i}\gnorm{q_i-\wb{q}_i}{2}{2} + \tfrac{3\gamma_t}{2\tau_t}\gnorm{q_{t+1}-\wb{q}_{t+1}}{2}{2} \label{eq:int_rel153}.
	\end{align}
	Here $q_t := \ell_F(x_t)-\ell_F(x_{t-1}) + \chi(x_t)-\chi(x_{t-1})$, $\wb{q}_t := \ell_f(x_t)-\ell_f(x_{t-1}) + \chi(x_t)-\chi(x_{t-1})$, $\delta_i^F := \ell_F(x_t) -\ell_f(x_t)$ and $\delta_t^G := G_{0}(x_t, \xi_t) + \tsum_{j =1}^m G_j(x_t,\xi_t)y^\br{j}_{t+1}-\subgrad{ f}_0(x_t) - \tsum_{j = 1}^m\subgrad{ f_j}(x_{t})y_{t+1}^\br{j}$.
\end{lemma}
\newcounter{num}
\setcounter{num}{0} %0 if using i and t. Else using t and T.

\begin{proofx}
	Note that $y_{t+1} = \argmin\limits_{y \ge \zero} \inprod{-s_t}{y} + \tfrac{\tau_t}{2} \gnorm{y-y_t}{2}{2}$. Hence, using Lemma \ref{lem:3-point}, we have for all $y \ge \zero$,
	\begin{equation}\label{eq:int_rel139} %\label{eq:int_rel68}
	-\inprod{s_t}{y_{t+1} -y} \le \tfrac{\tau_t}{2} \bracket*{\gnorm{y-y_t}{2}{2} - \gnorm{y_{t+1}-y_t}{2}{2} - \gnorm{y-y_{t+1} }{2}{2} }.
	\end{equation}
	Let us denote $v_t := \subgrad{ f_0}(x_t) + \tsum_{j = 1}^m\subgrad{ f_j}(x_{t})y_{t+1}^\br{j}$ and $V_t := G_{0}(x_t, \xi_t) + \tsum_{j=1}^m G_{j}(x_t, \xi_t)y^\br{j}_{t+1}$. 
	Then, due to the strong convexity of $\chi_0$ and $\chi_j, j =1, \dots, m$, the optimality of $x_{t+1}$, Lemma \ref{lem:3-point} and the definition of $\alpha_{0,t}$, we have for all $x \in X$,
	\begin{equation} \label{eq:int_rel140} %\label{eq:int_rel69}
	\begin{split}
	\inprod{V_t}{x_{t+1} -x} &+ \chi_0(x_{t+1})-\chi_0(x) + \tsum_{j=1}^m \paran*{\chi_j(x_{t+1}) -\chi_j(x)}y^\br{j}_{t+1}\\
	&\le \eta_t [W(x, x_t) -W(x_{t+1}, x_t)] - (\eta_t + \alpha_{0,t}) W(x, x_{t+1}).
	\end{split}
	\end{equation}
	Due to the convexity of $f_0$ and $f_i$, \eqref{eq:A3}, the definition of $\ell_f$ and the fact that $y_{t+1} \ge \zero$, we have
	\begingroup
	\allowdisplaybreaks
	\begin{align}
	&\inprod{v_t}{x_{t+1} -x} = \inprod{\subgrad{ f_0}(x_t) + \tsum_{i \in [m]}\subgrad{ f_i}(x_{t} ) y^\br{i}_{t+1} } {x_{t+1} -x} \nonumber \\
	&\quad= \inprod*{\subgrad{ f_0}(x_t)}{x_{t+1} -x_t +x_t -x} + \inprod{\subgrad{ f}(x_{t})y_{t+1}} {x_{t+1}-x_t + x_t -x} \nonumber \\
	&\quad\ge f_0(x_t) -f_0(x) + f_0(x_{t+1}) - f_0(x_t) - \tfrac{L_0}{2}\gnorm{x_{t+1}-x_t}{}{2} -H_0\gnorm{x_{t+1}-x_t}{}{} \nonumber \\
	&\quad\quad +\inprod{y_{t+1}}{\ell_f(x_{t+1}) - f(x_t)} + \inprod{y_{t+1}}{f(x_t)-f(x)}\nonumber\\
	%&\qquad+ \tsum_{i \in [m]}y_{t+1}^\br{i}\bracket*{ \wb{\psi}^\br{i}(x_t) - \wb{\psi}^\br{i}(x) + \wb{\psi}^\br{i}(x_{t+1}) -\wb{\psi}^\br{i}(x_t) -\tfrac{L_i}{2} \gnorm{x_{t+1}- x_t }{}^2} \nonumber\\
	&\quad= f_0(x_{t+1}) - f_0(x) + \inprod{\ell_f(x_{t+1}) - f(x)}{y_{t+1}} -\underbracket{\paran[\big]{\tfrac{L_0}{2} \gnorm{x_{t+1}- x_t}{}{2}+H_0\gnorm{x_{t+1}-x_t}{}{}}}_{O_{t+1}} \label{eq:int_rel141},
	%&\ge f_0(x_{t+1}) - f_0(x) + \inprod{\wb{\psi}(x_{t+1}) - \wb{\psi}(x)}{y_{t+1}} -(L_0+ L^Ty_{t+1}) W(x_{t+1},x_t) \label{eq:int_rel141}%\label{eq:int_rel70}.
	\end{align}
	\endgroup
	where $O_{t+1} :=\tfrac{L_0}{2} \gnorm{x_{t+1}- x_t}{}{2}+H_0\gnorm{x_{t+1}-x_t}{}{}$ is a `Lipschitz'-like term for the objective.
	%By definition we have $\delta^G_t = V_t-v_t$. 
	Combining \eqref{eq:int_rel140}, \eqref{eq:int_rel141}, noting that $\delta^G_t =  V_t - v_t$ and using $\psi_0 = f_0+ \chi_0$, $\psi=f + \chi$, we have 
	\begin{equation} \label{eq:int_rel142}%\label{eq:int_rel71}
	\begin{split}
	&\psi_0(x_{t+1}) -\psi_0(x)+ \inprod{\ell_f(x_{t+1})+\chi(x_{t+1}) -\psi(x)}{y_{t+1}} + \inprod{\delta^G_t}{x_{t+1}-x} \\
	&\le \eta_t W(x, x_t) - \eta_tW(x_{t+1}, x_t) -(\eta_t+\alpha_{0,t})W(x, x_{t+1}) + O_{t+1}.
	\end{split}
	\end{equation}
	Noting the definition of $Q(\cdot,\cdot)$ in \eqref{eq:saddle_conv1} and, adding \eqref{eq:int_rel139} and \eqref{eq:int_rel142}, we obtain
	\begingroup
	\allowdisplaybreaks
	\begin{align}
	%	\bracket*{ g(x_{t+1}) &+ \chi_0(x_{t+1}) + \inprod*{\psi(x_{t+1})}{y} } - \bracket*{ g(x) + \chi_0(x) + \inprod*{\psi(x)}{y_{t+1} } } + \inprod{\psi(x_{t+1} ) } {y_{t+1} - y} - \inprod{s_t} {y_{t+1}- y}\nonumber \\ 
	Q(z_{t+1}, z) & -\inprod{\psi(x_{t+1} )} {y} +\inprod{\ell_f(x_{t+1})+\chi(x_{t+1})}{y_{t+1}} - \inprod{s_t} {y_{t+1}- y} + \inprod{\delta^G_t}{x_{t+1}-x}\nonumber \\ 
	&\le \tfrac{\tau_t}{2} \bracket*{\gnorm{y - y_t}{2}{2} - \gnorm{y_{t+1}- y_t }{2}{2} -\gnorm{y- y_{t+1}}{2}{2}} \nonumber \\
	&\qquad+\eta_t W(x, x_t) - \eta_tW(x_{t+1}, x_t) -(\eta_t+\alpha_{0,t})W(x, x_{t+1}) + O_{t+1}%+ \tfrac{L_0}{2}\gnorm{x_{t+1}-x_t}{}^2 
	\label{eq:int_rel143}. %\label{eq:int_rel72}.
	\end{align}
	\endgroup
	In view of \eqref{eq:A2},
	\begin{equation*}
	f_i(x_{t+1})-\ell_{f_i}(x_{t+1}) \le \tfrac{L_{i}}{2}\gnorm{x_{t+1}-x_t}{}{2} + H_{i}\gnorm{x_{t+1}-x_t}{}{}.
	\end{equation*}
	Then, using Cauchy-Schwarz inequality and noting definitions of $L_f, H_f$, we have
	\begin{equation*}
	\inprod{y}{f(x_{t+1})-\ell_f(x_{t+1}) } \le \gnorm{y}{2}{} \underbracket{\bracket[\big]{\tfrac{L_f}{2}\gnorm{x_{t+1}-x_t}{}{2} + H_f\gnorm{x_{t+1}-x_t}{}{} }}_{C_{t+1}},
	\end{equation*}
	where $C_{t+1} := \tfrac{L_f}{2}\gnorm{x_{t+1}-x_t}{}{2} + H_f\gnorm{x_{t+1}-x_t}{}{} $ is a `Lipschitz'-like term for the constraints.
	In view of the above relation and definitions of $q_t$ and $\delta_{t+1}^F$, we have
	\begingroup
	\allowdisplaybreaks
	\begin{align}
	&\inprod{\ell_f(x_{t+1})+\chi(x_{t+1})}{y_{t+1}}-\inprod{\psi(x_{t+1} )} {y} - \inprod{s_t} {y_{t+1}- y} \nonumber\\
	&\quad \ge \inprod{\ell_f(x_{t+1})+\chi(x_{t+1})}{y_{t+1}} - \inprod{\ell_f(x_{t+1})+ \chi(x_{t+1})}{y} -\inprod{s_t}{y_{t+1}-y} - \gnorm{y}{2}{}C_{t+1}  \nonumber\\
	&\quad = \inprod{\ell_f(x_{t+1})+\chi(x_{t+1})-s_t}{y_{t+1}-y} - \gnorm{y}{2}{}C_{t+1} \nonumber\\
	&\quad = \inprod*{\ell_f(x_{t+1})+\chi(x_{t+1}) - \ell_F(x_t) - \chi(x_{t}) - \theta_tq_t}{y_{t+1}-y}-\gnorm{y}{2}{}C_{t+1} \nonumber \\
	&\quad = \inprod{q_{t+1}}{y_{t+1}-y} - \theta_t \inprod{q_t}{y_t - y} - \theta_t \inprod{q_t}{y_{t+1} - y_t} - \inprod{ \delta^F_{t+1} }{y_{t+1} - y} -\gnorm{y}{2}{}C_{t+1}\label{eq:int_rel144}.% \label{eq:int_rel73},
	\end{align}
	\endgroup
	{Given the non-negative constant $B$ and using the definition of $C_{t+1}$, we have}
	\begin{align}
	\gnorm{y}{2}{}C_{t+1} &= \tfrac{L_f}{2}(\gnorm{y}{2}{}-B)\gnorm{x_{t+1}-x_t}{}{2} + \tfrac{BL_f}{2}\gnorm{x_{t+1}-x_t}{}{2} + \gnorm{y}{2}{}H_f\gnorm{x_{t+1}-x_t}{}{} \nonumber \\
	&\le \tfrac{L_f}{2}\bracket{\gnorm{y}{2}{}-B}_+\gnorm{x_{t+1}-x_t}{}{2} + \tfrac{BL_f}{2}\gnorm{x_{t+1}-x_t}{}{2} + \gnorm{y}{2}{}H_f\gnorm{x_{t+1}-x_t}{}{} \nonumber \\
	&\le \tfrac{BL_f}{2}\gnorm{x_{t+1}-x_t}{}{2} + \paran[\big]{\gnorm{y}{2}{}H_f + \tfrac{L_f D_X}{2}\bracket{\gnorm{y}{2}{}-B}_+}\gnorm{x_{t+1}-x_t}{}{}. \label{eq:int_rel163}
	\end{align}
	{Recall the definition of $D_X$ from \eqref{eq:diameter}.}
	By \eqref{eq:int_rel143}, \eqref{eq:int_rel144}, and \eqref{eq:int_rel163}, noting the definition of $O_{t+1}$, using the relation $\tfrac{1}{2}\gnorm{a-b}{}{2} \le W(a,b)$and replacing index $t$ by $i$, we have
	\begingroup
	\allowdisplaybreaks
	\ifthenelse{\value{num} = 0}{	
		\begin{align}
			&Q(z_{i+1}, z) + \inprod{q_{i+1}}{y_{i+1}-y} - \theta_i \inprod{q_i}{y_i - y} +\inprod{\delta^G_i}{x_i-x} -\inprod{\delta^F_{i+1}}{y_{i+1}-y}\nonumber \\
			&\le \theta_i \inprod{q_i}{y_{i+1} -  y_i} - \inprod{\delta^G_i}{x_{i+1}-x_i} \nonumber \\
			&\quad+\eta_i W(x, x_i) -(\eta_i+\alpha_{0,i})W(x, x_{i+1}) +\tfrac{\tau_i}{2} \bracket*{\gnorm{y - y_i}{2}{2} - \gnorm{y_{i+1}- y_i }{2}{2} -\gnorm{y- y_{i+1}}{2}{2} }  \nonumber \\
			&\quad  - (\eta_i-L_0-BL_f)W(x_{i+1}, x_i)+ \paran[\big]{H_0+\gnorm{y}{2}{}H_f + \tfrac{L_f D_X}{2}\bracket{\gnorm{y}{2}{}-B}_+}\gnorm{x_{i+1}-x_i}{}{}.\label{eq:int_rel145} %\label{eq:int_rel82}
		\end{align}
	}
	{
	\begin{align}
	&Q(z_{t+1}, z) + \inprod{q_{t+1}}{y_{t+1}-y} - \theta_t \inprod{q_t}{y_t - y} +\inprod{\delta^G_t}{x_t-x} -\inprod{\delta^F_{t+1}}{y_{t+1}-y}\nonumber \\
	&\le \theta_t \inprod{q_t}{y_{t+1} -  y_t} - \inprod{\delta^G_t}{x_{t+1}-x_t} \nonumber \\
	&\quad+\eta_t W(x, x_t) -(\eta_t+\alpha_{0,t})W(x, x_{t+1}) +\tfrac{\tau_t}{2} \bracket*{\gnorm{y - y_t}{2}{2} - \gnorm{y_{t+1}- y_t }{2}{2} -\gnorm{y- y_{t+1}}{2}{2} }  \nonumber \\
	&\quad  - (\eta_t-L_0-BL_f)W(x_{t+1}, x_t)+ \paran[\big]{H_0+\gnorm{y}{2}{}H_f + \tfrac{L_f D_X}{2}\bracket{\gnorm{y}{2}{}-B}_+}\gnorm{x_{t+1}-x_t}{}{}.\label{eq:int_rel145} %\label{eq:int_rel82}
	\end{align}
	}
	\endgroup
	%Noting the definition of $Q(\cdot, \cdot)$ in \eqref{eq:saddle_conv1}, 
	Multiplying \eqref{eq:int_rel145} by \ifthenelse{\value{num} = 0}{$\gamma_i$}{$\gamma_t$}, summing them up from \ifthenelse{\value{num} = 0}{$i = 0$ to $t$ with $t \ge 0$}{$t = 0$ to $T-1$ with $T \ge 1$} {and noting that $q_0 = \zero$}, we obtain
	\begingroup
	\allowdisplaybreaks
	\ifthenelse{\value{num} = 0}{
	\begin{align}
		&\tsum_{i = 0}^{t} \gamma_i Q(z_{i+1}, z)  + \tsum_{i = 0}^{t} [\gamma_i \inprod{q_{i+1}}{y_{i+1}-y} - \gamma_i \theta_i \inprod{q_i}{y_i - y}] \nonumber\\
		&+ \tsum_{i =0}^{t}\gamma_i[\inprod{\delta^G_i}{x_i-x} - \inprod{\delta^F_{i+1}}{y_{i+1}-y} ] \nonumber\\
		&\le \tsum_{i = \textcolor{blue}{1}}^{t} [\gamma_i\theta_i \inprod{q_i-\wb{q}_i}{y_{i+1}- y_i} + \gamma_i\theta_i\inprod{\wb{q}_i}{y_{i+1}- y_i}] + \tsum_{ i =0}^t\inprod{\gamma_i\delta^G_i}{x_i-x_{i+1}} \nonumber\\
		&\qquad+ \tsum_{i = 0}^{t} \bracket*{\tfrac{\gamma_i\tau_i}{2} \gnorm{y- y_i}{2}{2} - \tfrac{\gamma_i\tau_i}{2}\gnorm{y- y_{i+1} }{2}{2}} - \tsum_{i = 0}^{t}\tfrac{\gamma_i\tau_i}{2} \gnorm{y_{i+1}- y_i}{2}{2} \nonumber \\
		&\qquad +\tsum_{i = 0}^{t} [\gamma_i\eta_i W(x, x_i) - \gamma_i(\eta_i+\alpha_{0,i})W(x, x_{i+1})] \nonumber \\
		&\qquad -\tsum_{i = 0}^{t}\bracket[\big]{ \gamma_i (\eta_i-L_0 - BL_f)W(x_{i+1}, x_i) - \gamma_i\underbracket{\big(H_0 + \gnorm{y}{2}{}H_f + \tfrac{L_fD_X}{2}\bracket{\gnorm{y}{2}{}-B}_+\big)}_{\HH(y, B)}\gnorm{x_{i+1}-x_i}{}{}}  \label{eq:int_rel146}, %\label{eq:int_rel74}.
	\end{align}
	}
	{	\begin{align}
			&\tsum_{t = 0}^{T-1} \gamma_t Q(z_{t+1}, z)  + \tsum_{t = 0}^{T-1} [\gamma_t \inprod{q_{t+1}}{y_{t+1}-y} - \gamma_t \theta_t \inprod{q_t}{y_t - y}] \nonumber\\
			&+ \tsum_{t =0}^{T-1}\gamma_t[\inprod{\delta^G_t}{x_t-x} - \inprod{\delta^F_{t+1}}{y_{t+1}-y} ] \nonumber\\
			&\le \tsum_{t = 0}^{T-1} [\gamma_t\theta_t \inprod{q_t-\wb{q}_t}{y_{t+1}- y_t} + \gamma_t\theta_t\inprod{\wb{q}_t}{y_{t+1}- y_t} + \inprod{\gamma_t\delta^G_t}{x_t-x_{t+1}}] \nonumber\\
			&\qquad+ \tsum_{t = 0}^{T-1} \bracket*{\tfrac{\gamma_t\tau_t}{2} \gnorm{y- y_t}{2}{2} - \tfrac{\gamma_t\tau_t}{2}\gnorm{y- y_{t+1} }{2}{2}} - \tsum_{t = 0}^{T-1}\tfrac{\gamma_t\tau_t}{2} \gnorm{y_{t+1}- y_t}{2}{2} \nonumber \\
			&\qquad +\tsum_{t = 0}^{T-1} [\gamma_t\eta_t W(x, x_t) - \gamma_t(\eta_t+\alpha_{0,t})W(x, x_{t+1})] \nonumber \\
			&\qquad -\tsum_{t = 0}^{T-1}\bracket[\big]{ \gamma_t (\eta_t-L_0 - BL_f)W(x_{t+1}, x_t) - \gamma_t\underbracket{\big(H_0 + \gnorm{y}{2}{}H_f + \tfrac{L_fD_X}{2}\bracket{\gnorm{y}{2}{}-B}_+\big)}_{\HH(y, B)}\gnorm{x_{t+1}-x_t}{}{}}  \label{eq:int_rel146}, %\label{eq:int_rel74}.
	\end{align}
	}
	\endgroup
	where $\HH(y, B) := H_0 + \gnorm{y}{2}{}H_f + \tfrac{L_fD_X}{2}\bracket{\gnorm{y}{2}{}-B}_+$.
	%Lets define $M(y,B) := \gnorm{y}{2}M_f + \tfrac{L_f D_X}{2}\bracket{\gnorm{y}{2}-B}_+$ for ease of notation.\\
	Now we focus our attention to handle the inner product terms of \eqref{eq:int_rel146}. Noting the definition of $\wb{q}_t$, we have
	\ifthenelse{\value{num} = 0}
	{	\begin{align}
			\gnorm{\wb{q}_i}{2}{} &= \gnorm{\ell_f(x_i)-\ell_f(x_{i-1})+\chi(x_i)-\chi(x_{i-1}) }{2}{} \nonumber\\
			&\le \gnorm{f(x_{i-1}) + \subgrad{ f}(x_{i-1})^T(x_i-x_{i-1}) -f(x_{i-2}) - \subgrad{ f}(x_{i-2})^T(x_{i-1}-x_{i-2})}{2}{} + \gnorm{\chi(x_i)-\chi(x_{i-1})}{2}{}\nonumber\\
			&\le \gnorm{ f(x_{i-1}) -f(x_{i-2}) - \subgrad{ f}(x_{i-2})^T(x_{i-1}-x_{i-2}) }{2}{} + \gnorm{\subgrad{ f}(x_{i-1})^T(x_i-x_{i-1})}{2}{} + M_\chi\gnorm{x_i-x_{i-1}}{}{}\nonumber\\
			&\le 2M_f \gnorm{x_{i-1}-x_{i-2}}{}{}  + (M_f+{M_\chi})\gnorm{x_i-x_{i-1}}{}{} \label{eq:int_rel147},
	\end{align}
	}
	{	
	\begin{align}
			\gnorm{\wb{q}_t}{2}{} &= \gnorm{\ell_f(x_t)-\ell_f(x_{t-1})+\chi(x_t)-\chi(x_{t-1}) }{2}{} \nonumber\\
			&\le \gnorm{f(x_{t-1}) + \subgrad{ f}(x_{t-1})^T(x_t-x_{t-1}) -f(x_{t-2}) - \subgrad{ f}(x_{t-2})^T(x_{t-1}-x_{t-2})}{2}{} + \gnorm{\chi(x_t)-\chi(x_{t-1})}{2}{}\nonumber\\
			&\le \gnorm{ f(x_{t-1}) -f(x_{t-2}) - \subgrad{ f}(x_{t-2})^T(x_{t-1}-x_{t-2}) }{2}{} + \gnorm{\subgrad{ f}(x_{t-1})^T(x_t-x_{t-1})}{2}{} + M_\chi\gnorm{x_t-x_{t-1}}{}{}\nonumber\\
			&\le 2M_f \gnorm{x_{t-1}-x_{t-2}}{}{}  + (M_f+{M_\chi})\gnorm{x_t-x_{t-1}}{}{} \label{eq:int_rel147},
	\end{align}
	}
	where the last relation follows due to \eqref{eq:lipschitz_relations}.
	Using the above relation, %\eqref{eq:int_rel147}
	we have for all $i \ge 1$, %\comment[id=qd]{equations are too long}
	\begingroup
	\allowdisplaybreaks
\ifthenelse{\value{num} = 0}
	{\begin{align}
		&\gamma_{i}\theta_i\inprod{\wb{q}_i}{y_{i+1}-y_i} -\tfrac{\gamma_{i}\tau_i}{3}\gnorm{y_{i+1}-y_i}{2}{2}-\tfrac{\gamma_{i-2}(\eta_{i-2}-L_0-BL_f)}{4}W(x_{i-1},x_{i-2}) \nonumber\\ &-\tfrac{\gamma_{i-1}(\eta_{i-1}-L_0-BL_f)}{4}W(x_i,x_{i-1})\nonumber\\
		& \le \gamma_{i}\theta_i\gnorm{\wb{q}_i}{2}{}\gnorm{y_{i+1}-y_i}{2}{} -\tfrac{\gamma_{i}\tau_i}{3}\gnorm{y_{i+1}-y_i}{2}{2} \nonumber \\
		&\quad-\tfrac{\gamma_{i-2}(\eta_{i-2}-L_0-BL_f)}{4}W(x_{i-1},x_{i-2}) -\tfrac{\gamma_{i-1}(\eta_{i-1}-L_0-BL_f)}{4}W(x_i,x_{i-1})\nonumber\\
		& \le 2M_f \gamma_{i}\theta_i \gnorm{x_{i-1}-x_{i-2}}{}{}\gnorm{y_{i+1}-y_i}{2}{}-\tfrac{\gamma_{i}\tau_i}{6}\gnorm{y_{i+1}-y_i}{2}{2} -\tfrac{\gamma_{i-2}(\eta_{i-2}-L_0-BL_f)}{4}W(x_{i-1},x_{i-2}) \nonumber\\
		&\quad  +(M_f+M_\chi) \gamma_{i}\theta_i \gnorm{x_i-x_{i-1}}{}{}\gnorm{y_{i+1}-y_i}{2}{} - \tfrac{\gamma_{i}\tau_i}{6}\gnorm{y_{i+1}-y_i}{2}{2} -\tfrac{\gamma_{i-1}(\eta_{i-1}-L_0-BL_f)}{4}W(x_i,x_{i-1})\nonumber\\
		& \le 0, \label{eq:int_rel148}
	\end{align}
	}
	{\begin{align}
		&\gamma_{t}\theta_t\inprod{\wb{q}_t}{y_{t+1}-y_t} -\tfrac{\gamma_{t}\tau_t}{3}\gnorm{y_{t+1}-y_t}{2}{2}-\tfrac{\gamma_{t-2}(\eta_{t-2}-L_0-BL_f)}{4}W(x_{t-1},x_{t-2}) \nonumber\\ &-\tfrac{\gamma_{t-1}(\eta_{t-1}-L_0-BL_f)}{4}W(x_t,x_{t-1})\nonumber\\
		& \le \gamma_{t}\theta_t\gnorm{\wb{q}_t}{2}{}\gnorm{y_{t+1}-y_t}{2}{} -\tfrac{\gamma_{t}\tau_t}{3}\gnorm{y_{t+1}-y_t}{2}{2} \nonumber \\
		&\quad-\tfrac{\gamma_{t-2}(\eta_{t-2}-L_0-BL_f)}{4}W(x_{t-1},x_{t-2}) -\tfrac{\gamma_{t-1}(\eta_{t-1}-L_0-BL_f)}{4}W(x_t,x_{t-1})\nonumber\\
		& \le 2M_f \gamma_{t}\theta_t \gnorm{x_{t-1}-x_{t-2}}{}{}\gnorm{y_{t+1}-y_t}{2}{}-\tfrac{\gamma_{t}\tau_t}{6}\gnorm{y_{t+1}-y_t}{2}{2} -\tfrac{\gamma_{t-2}(\eta_{t-2}-L_0-BL_f)}{4}W(x_{t-1},x_{t-2}) \nonumber\\
		&\quad  +(M_f+M_\chi) \gamma_{t}\theta_t \gnorm{x_t-x_{t-1}}{}{}\gnorm{y_{t+1}-y_t}{2}{} - \tfrac{\gamma_{t}\tau_t}{6}\gnorm{y_{t+1}-y_t}{2}{2} -\tfrac{\gamma_{t-1}(\eta_{t-1}-L_0-BL_f)}{4}W(x_t,x_{t-1})\nonumber\\
		& \le 0, \label{eq:int_rel148}
	\end{align}}
	\endgroup
	where the last inequality follows by applying the relation $W(x,y) \ge \tfrac{1}{2}\gnorm{x-y}{}{2}$, Young's inequality ($2ab \le a^2+ {b^2}$) applied twice, once with 
	\[\ifthenelse{\value{num} = 0}
	{a = \paran[\big]{\tfrac{\gamma_i\tau_i}{6}}^{1/2}\gnorm{y_{i+1}-y_i}{}{}, \tab b = \paran[\big]{ \tfrac{\gamma_{i-2}(\eta_{i-2}-L_0-BL_f)}{8}}^{1/2}\gnorm{x_{i-1}-x_{i-2}}{}{}
	}
	{a = \paran[\big]{\tfrac{\gamma_t\tau_t}{6}}^{1/2}\gnorm{y_{t+1}-y_t}{}{}, \tab b = \paran[\big]{ \tfrac{\gamma_{t-2}(\eta_{t-2}-L_0-BL_f)}{8}}^{1/2}\gnorm{x_{t-1}-x_{t-2}}{}{}},
	\] 
	second time with 
	\[\ifthenelse{\value{num} = 0}{	
		a =\paran[\big]{\tfrac{\gamma_i\tau_i}{6}}^{1/2}\gnorm{y_{i+1}-y_i}{}{},\tab b = \paran[\big]{ \tfrac{\gamma_{i-1}(\eta_{i-1}-L_0-BL_f)}{8}}^{1/2}\gnorm{x_i-x_{i-1}}{}{},}
	{a = \paran[\big]{\tfrac{\gamma_t\tau_t}{6}}^{1/2}\gnorm{y_{t+1}-y_t}{}{},\tab b = \paran[\big]{ \tfrac{\gamma_{t-1}(\eta_{t-1}-L_0-BL_f)}{8}}^{1/2}\gnorm{x_t-x_{t-1}}{}{},}
	\] and the fact that \\[1mm]
	\ifthenelse{\value{num} = 0}{	\begin{tabular}{McM}
			(2M_f)\gamma_{i}\theta_i & \braces[\big]{\tfrac{\gamma_{i}\gamma_{i-2}\tau_i(\eta_{i-2}-L_0- BL_f)}{12}}^{1/2} &$\Leftrightarrow$ &(2M_f)^2\tfrac{\theta_i}{\theta_{i-1}}& \tfrac{\tau_i(\eta_{i-2}-L_0-BL_f)}{12},\\
			(M_f+{M_\chi})\gamma_{i}\theta_i & \braces[\big]{\tfrac{\gamma_{i}\gamma_{i-1}\tau_i(\eta_{i-1}-L_0-BL_f)}{12}}^{1/2} &$\Leftrightarrow$ &(M_f+{M_\chi})^2\theta_i & \tfrac{\tau_i(\eta_{i-1}-L_0-BL_f)}{12},
	\end{tabular}}
{	\begin{tabular}{McM}
			(2M_f)\gamma_{t}\theta_t & \braces[\big]{\tfrac{\gamma_{t}\gamma_{t-2}\tau_t(\eta_{t-2}-L_0- BL_f)}{12}}^{1/2} &$\Leftrightarrow$ &(2M_f)^2\tfrac{\theta_t}{\theta_{t-1}}& \tfrac{\tau_t(\eta_{t-2}-L_0-BL_f)}{12},\\
			(M_f+{M_\chi})\gamma_{t}\theta_t & \braces[\big]{\tfrac{\gamma_{t}\gamma_{t-1}\tau_t(\eta_{t-1}-L_0-BL_f)}{12}}^{1/2} &$\Leftrightarrow$ &(M_f+{M_\chi})^2\theta_t & \tfrac{\tau_t(\eta_{t-1}-L_0-BL_f)}{12},
	\end{tabular}}
	\\[1mm]
	where equivalences follow from \eqref{eq:int_rel154-a} {and conditions follow from relations in \eqref{eq:int_rel155-1-a} and \eqref{eq:int_rel155-2}. In particular,} %\eqref{eq:int_rel148} trivially holds for }
	%\ifthenelse{\value{num} = 0}{{$i = 0$}} {{$t = 0$}} {since $\wb{q}_0 = \zero$. Hence, 
	{we require \eqref{eq:int_rel155-1-a} and \eqref{eq:int_rel155-2} for $t \ge 1$  such that \eqref{eq:int_rel148} is satisfied for $i \ge 1$. Moreover,} \ifthenelse{\value{num} =0}{{$\gnorm{x_{i-1}-x_{i-2}}{}{} = 0$  for $i=1$}}{{$\gnorm{x_{t-1}-x_{t-2}}{}{} = 0$  for $t =1$}}. {Hence, we require \eqref{eq:int_rel155-2} for $t \ge 2$}.\\
	%	Noting the definition of $q_t$ and $\wb{q}_t$, we have
	%	\begin{align*}
	%	\gnorm{q_t -\wb{q}_t}{2}^2 &= \gnorm{\ell_{F,t} - \ell_f(x_t) - \ell_{F,t-1} + \ell_f(x_{t-1})}{2}^2\\
	%	&= 2\gnorm{\ell_{F,t}-\ell_f(x_t)}{2}^2 + 2\gnorm{\ell_{F,t-1}-\ell_f(x_{t-1})}{2}^2
	%	\end{align*}
	%	Now, we have
	%	\begin{align*}
	%	\gnorm{\ell_{F,t}-\ell_f(x_t)}{2}^2 \le 2
	%	\end{align*}
	Using Young's inequality, Cauchy-Schwarz inequality and the relation $u^Tv \le \gnorm{u}{}{} \gnorm{v}{\ast}{}$, we have 
	\ifthenelse{\value{num} = 0}
	{\begin{equation}\label{eq:int_rel149}
		\begin{split}
			\gamma_i\theta_i\inprod{q_i-\wb{q}_i}{y_{i+1}-y_i} -\tfrac{\gamma_i\tau_i}{6}\gnorm{y_{i+1}-y_i}{2}{2} &\le \tfrac{3\gamma_{i}\theta_i^2}{2\tau_i}\gnorm{q_i-\wb{q}_i}{2}{2},\\
			\inprod{\gamma_i\delta_i^G}{x_i-x_{i+1}} - \tfrac{\gamma_i(\eta_i-L_0-BL_f)}{4}W(x_{i+1},x_i) &\le \tfrac{2\gamma_i}{\eta_i-L_0-BL_f}\gnorm{\delta_i^G}{*}{2},\\
			\gamma_i\HH(y,B) \gnorm{x_{i+1}-x_i}{}{}- \tfrac{\gamma_i(\eta_i-L_0-BL_f)}{4}W(x_{i+1},x_i) &\le \tfrac{2\gamma_i}{\eta_i-L_0-BL_f}\HH({y},B)^2.
		\end{split}
	\end{equation}
	}
	{\begin{equation}\label{eq:int_rel149}
		\begin{split}
			\gamma_t\theta_t\inprod{q_t-\wb{q}_t}{y_{t+1}-y_t} -\tfrac{\gamma_t\tau_t}{6}\gnorm{y_{t+1}-y_t}{2}{2} &\le \tfrac{3\gamma_{t}\theta_t^2}{2\tau_t}\gnorm{q_t-\wb{q}_t}{2}{2},\\
			\inprod{\gamma_t\delta_t^G}{x_t-x_{t+1}} - \tfrac{\gamma_t(\eta_t-L_0-BL_f)}{4}W(x_{t+1},x_t) &\le \tfrac{2\gamma_t}{\eta_t-L_0-BL_f}\gnorm{\delta_t^G}{*}{2},\\
			\gamma_t\HH(y,B) \gnorm{x_{t+1}-x_t}{}{}- \tfrac{\gamma_t(\eta_t-L_0-BL_f)}{4}W(x_{t+1},x_t) &\le \tfrac{2\gamma_t}{\eta_t-L_0-BL_f}\HH({y},B)^2.
		\end{split}
	\end{equation}
	}
	Using \eqref{eq:int_rel148} and \eqref{eq:int_rel149} for \ifthenelse{\value{num} = 0}{$i = \textcolor{blue}{1}, \dots, t$}{$t = 0, \dots, T-1$} inside \eqref{eq:int_rel146} and noting \eqref{eq:int_rel154}, we have
	\ifthenelse{\value{num} = 0}
	{\begin{align}
			\tsum_{i = 0}^{t} \gamma_i &Q(z_{i+1}, z)  +  \gamma_{t} \inprod{q_{t+1}}{y_{t+1}-y}+ \tsum_{i =0}^{t}\gamma_i[\inprod{\delta^G_i}{x_i-x} - \inprod{\delta^F_{i+1}}{y_{i+1}-y} ] \nonumber\\
			&\le \gamma_0\eta_0W(x,x_0) -\gamma_{t}(\eta_t+\alpha_{0, t})W(x, x_{t+1}) + \tfrac{\gamma_0\tau_0}{2}\gnorm{y-y_0}{2}{2} - \tfrac{\gamma_{t}\tau_{t}}{2}\gnorm{y-y_{t+1}}{2}{2}\nonumber\\
			&\qquad+\tsum_{ i =1}^{t}\tfrac{3\gamma_i\theta_i^2}{2\tau_i}\gnorm{q_i-\wb{q}_i}{2}{2} + \tsum_{ i =0}^t\bracket[\big]{\tfrac{2\gamma_i}{\eta_i-L_0-BL_f}\gnorm{\delta_i^G}{*}{2} +\tfrac{2\gamma_i}{\eta_i-L_0-BL_f}\HH({y},B)^2}\nonumber \\
			&\qquad - \tfrac{\gamma_{t-1}(\eta_{t-1}-L_0-BL_f)}{4}W(x_{t},x_{t-1})  - \tfrac{\gamma_{t}(\eta_{t}-L_0-BL_f)}{2}W(x_{t+1},x_{t}),\label{eq:int_rel150}
	\end{align}}
	{	\begin{align}
			\tsum_{t = 0}^{T-1} \gamma_t &Q(z_{t+1}, z)  +  \gamma_{T-1} \inprod{q_T}{y_T-y}+ \tsum_{t =0}^{T-1}\gamma_t[\inprod{\delta^G_t}{x_t-x} - \inprod{\delta^F_{t+1}}{y_{t+1}-y} ] \nonumber\\
			&\le \gamma_0\eta_0W(x,x_0) -\gamma_{T-1}(\eta_t+\alpha_{0, T-1})W(x, x_T) + \tfrac{\gamma_0\tau_0}{2}\gnorm{y-y_0}{2}{2} - \tfrac{\gamma_{T-1}\tau_{T-1}}{2}\gnorm{y-y_T}{2}{2}\nonumber\\
			&\qquad+\tsum_{ t =0}^{T-1} \bracket[\big]{\tfrac{3\gamma_t\theta_t^2}{2\tau_t}\gnorm{q_t-\wb{q}_t}{2}{2} + \tfrac{2\gamma_t}{\eta_t-L_0-BL_f}\gnorm{\delta_t^G}{*}{2} +\tfrac{2\gamma_t}{\eta_t-L_0-BL_f}\HH({y},B)^2}\nonumber \\
			&\qquad - \tfrac{\gamma_{T-2}(\eta_{T-2}-L_0-BL_f)}{4}W(x_{T-1},x_{T-2})  - \tfrac{\gamma_{T-1}(\eta_{T-1}-L_0-BL_f)}{2}W(x_T,x_{T-1}),\label{eq:int_rel150}
	\end{align}}
	where on the left hand side of the above relation, we use the fact that $q_0 = \ell_{F}(x_0)-\ell_{F}(x_{-1}) + \chi(x_0)-\chi(x_{-1}) = \zero$. %Similarly, we see that $\wb{q}_0 = \zero$. Hence we can ignore $\gnorm{q_0-\wb{q}_0}{2}{2}$ term on the right hand side of the above relation.\\
	Using \eqref{eq:int_rel147}, we have
	\begingroup
	\allowdisplaybreaks
	\ifthenelse{\value{num} = 0}
		{\begin{align}
			&-\gamma_t\inprod{\wb{q}_{t+1}}{y_{t+1}-y} -\tfrac{\gamma_t\tau_t}{3}\gnorm{y-y_{t+1}}{2}{2} \nonumber \\
			&\quad - \tfrac{\gamma_{t-1}(\eta_{t-1}-L_0-BL_f)}{4}W(x_t,x_{t-1})  - \tfrac{\gamma_t(\eta_t-L_0-BL_f)}{2}W(x_{t+1},x_t) \nonumber \\
			& \le (M_f+M_\chi)\gamma_t \gnorm{x_{t+1}-x_t}{}{} \gnorm{y_{t+1}-y}{2}{}-\tfrac{\gamma_t\tau_t}{12}\gnorm{y-y_{t+1}}{2}{2} - \tfrac{\gamma_t(\eta_t-L_0-BL_f)}{2}W(x_{t+1},x_t)\nonumber \\
			& \quad+2M_f\gamma_t \gnorm{x_t-x_{t-1}}{}{} \gnorm{y_{t+1}-y}{2}{}	-\tfrac{\gamma_t\tau_t}{6}\gnorm{y-y_{t+1}}{2}{2} - \tfrac{\gamma_{t-1}(\eta_{t-1}-L_0-BL_f)}{4}W(x_t,x_{t-1})\nonumber \\
			&\quad -\tfrac{\gamma_t\tau_t}{12}\gnorm{y_{t+1}-y}{2}{2}\nonumber\\
			& \le -\tfrac{\gamma_t\tau_t}{12}\gnorm{y_{t+1}-y}{2}{2}, \label{eq:int_rel151}
		\end{align}}
		{\begin{align}
			&-\gamma_{T-1}\inprod{\wb{q}_T}{y_T-y} -\tfrac{\gamma_{T-1}\tau_{T-1}}{3}\gnorm{y-y_T}{2}{2} \nonumber \\
			&\quad - \tfrac{\gamma_{T-2}(\eta_{T-2}-L_0-BL_f)}{4}W(x_{T-1},x_{T-2})  - \tfrac{\gamma_{T-1}(\eta_{T-1}-L_0-BL_f)}{2}W(x_T,x_{T-1}) \nonumber \\
			& \le (M_f+M_\chi)\gamma_{T-1} \gnorm{x_T-x_{T-1}}{}{} \gnorm{y_T-y}{2}{}-\tfrac{\gamma_{T-1}\tau_{T-1}}{12}\gnorm{y-y_T}{2}{2} - \tfrac{\gamma_{T-1}(\eta_{T-1}-L_0-BL_f)}{2}W(x_T,x_{T-1})\nonumber \\
			& \quad+2M_f\gamma_{T-1} \gnorm{x_{T-1}-x_{T-2}}{}{} \gnorm{y_T-y}{2}{}	-\tfrac{\gamma_{T-1}\tau_{T-1}}{6}\gnorm{y-y_T}{2}{2} - \tfrac{\gamma_{T-2}(\eta_{T-2}-L_0-BL_f)}{4}W(x_{T-1},x_{T-2})\nonumber \\
			&\quad -\tfrac{\gamma_{T-1}\tau_{T-1}}{12}\gnorm{y_T-y}{2}{2}\nonumber\\
			& \le -\tfrac{\gamma_{T-1}\tau_{T-1}}{12}\gnorm{y_T-y}{2}{2}, \label{eq:int_rel151}
		\end{align}
	}
	\endgroup
	where the last relation follows from \eqref{eq:int_rel155-1-b} and \eqref{eq:int_rel155-2-b}, %for last iteration index $T$, 
	Young's inequality and the fact that \\[1mm]
	\ifthenelse{\value{num} = 0}
	{\begin{tabular}{McM}
			(2M_f)\gamma_t & \braces[\big]{\tfrac{\gamma_{t-1}\gamma_t\tau_t(\eta_{t-1}-L_0-BL_f)}{12}}^{1/2} &$\Leftrightarrow$ &(2M_f)^2\tfrac{1}{\theta_t} &\tfrac{{\tau_t(\eta_{t-1}-L_0-BL_f)}}{12},\\
			(M_f+M_\chi)\gamma_t & \braces[\big]{\tfrac{\gamma_t^2\tau_t(\eta_t-L_0-BL_f)}{12}}^{1/2} &$\Leftrightarrow$ &(M_f+M_\chi)^2 & \tfrac{\tau_t(\eta_t-L_0-BL_f)}{12}.
	\end{tabular}}
	{\begin{tabular}{McM}
			(2M_f)\gamma_{T-1} & \braces[\big]{\tfrac{\gamma_{T-2}\gamma_{T-1}\tau_{T-1}(\eta_{T-2}-L_0-BL_f)}{12}}^{1/2} &$\Leftrightarrow$ &(2M_f)^2\tfrac{1}{\theta_{T-1}} &\tfrac{{\tau_T(\eta_{T-2}-L_0-BL_f)}}{12},\\
			(M_f+M_\chi)\gamma_{T-1} & \braces[\big]{\tfrac{\gamma_{T-1}^2\tau_{T-1}(\eta_{T-1}-L_0-BL_f)}{12}}^{1/2} &$\Leftrightarrow$ &(M_f+M_\chi)^2 & \tfrac{\tau_{T-1}(\eta_{T-1}-L_0-BL_f)}{12}.
	\end{tabular}}
	\\ [1mm]
	%{As earlier, we require \eqref{eq:int_rel155-2-b} only for the case} \ifthenelse{\value{num} = 0}{{$t\ge 1$ since $\gnorm{x_{t} - x_{t-1}}{}{} = 0$ for $t = 0$.}}{{$T\ge 2$ since $\gnorm{x_{T-1} - x_{T-2}}{}{} = 0$ for $T = 1$.}}\\
	Finally, again using Young's inequality and Cauchy-Schwarz inequality, we have
	\ifthenelse{\value{num} = 0}
	{\begin{equation}\label{eq:int_rel152}
		-\gamma_t\inprod{q_{t+1}-\wb{q}_{t+1}} {y_{t+1}-y} -\tfrac{\gamma_t\tau_t}{6}\gnorm{y-y_{t+1}}{2}{2} \le\tfrac{3\gamma_t}{2\tau_t}\gnorm{q_{t+1}-\wb{q}_{t+1}}{2}{2}.
	\end{equation}
	}
	{\begin{equation}\label{eq:int_rel152}
		-\gamma_{T-1}\inprod{q_T-\wb{q}_T} {y_T-y} -\tfrac{\gamma_{T-1}\tau_{T-1}}{6}\gnorm{y-y_T}{2}{2} \le\tfrac{3\gamma_{T-1}}{2\tau_{T-1}}\gnorm{q_T-\wb{q}_T}{2}{2}.
	\end{equation}
	}
	Using \eqref{eq:int_rel151} and \eqref{eq:int_rel152} in relation \eqref{eq:int_rel150}, noting that $q_0-\wb{q}_0 = \zero$
	and replacing the definition of $\HH(y, B)$, %and taking into account the remaining $\tfrac{\gamma_{T-1}\tau_{T-1}}{12}\gnorm{y-y_T}{2}^2$
	we obtain \eqref{eq:int_rel153}. {Hence, we conclude the proof.}
\end{proofx}

We now aim to convert the bound on the primal-dual gap function $Q$ in Lemma \ref{thm:linearized_conex_1} into bounds on the optimality and infeasibility of the ergodic solution $\wb{x}_T$ according to Definition \ref{define-apprx-opt}. For proving this lemma, we need one more simple result which is stated below.
\begin{lemma} \label{lem:tech_res1}
	Let $\rho_0, \dots, \rho_j$ be a sequence of elements in $\Rbb^n$ and let $S$ be a convex set in $\Rbb^n$. Define the sequence $v_t, t = 0,1,\dots$, as follows: $v_0 \in S$ and 
	\begin{equation*}
	v_{t+1} = \argmin_{x \in S} \inprod{\rho_t}{x} + \tfrac{1}{2}\gnorm{x-v_t}{2}{2}.
	\end{equation*}
	Then for any $x \in S$ and $t \ge 0$, the following inequalities hold
	\begin{equation}\label{eq:int_rel114_1}
	\inprod{\rho_t}{v_t-x} \le \tfrac{1}{2}\gnorm{x-v_t}{2}{2}-\tfrac{1}{2} \gnorm{x-v_{t+1}}{2}{2}+ \tfrac{1}{2}\gnorm{\rho_t}{2}{2},
	\end{equation}
	\begin{equation}\label{eq:int_rel114}
	\tsum_{ t =0}^{j} \inprod{\rho_t}{v_t-x} \le \tfrac{1}{2}\gnorm{x-v_0}{2}{2} + \tfrac{1}{2}\tsum_{t=0}^{j} \gnorm{\rho_t}{2}{2}.
	\end{equation}
\end{lemma}
\begin{proofx}
	Using Lemma \ref{lem:3-point} with $g(x) =  \inprod{\rho_t}{x}$, $W(y,x) = \tfrac{1}{2}\gnorm{y-x}{2}{2}$, $\wt{x} = v_t$ and $\mu = 0$, we have, due to the optimality of $v_{t+1}$,
	\begin{equation*}
	\inprod{\rho_t}{v_{t+1}-x} + \tfrac{1}{2}\gnorm{v_{t+1}-v_t}{2}{2} + \tfrac{1}{2} \gnorm{x-v_{t+1}}{2}{2}\le  \tfrac{1}{2}\gnorm{x-v_t}{2}{2},
	\end{equation*}
	is satisfied for all $x \in S$. The above relation and the fact
	\begin{equation*}
	\inprod{\rho_t}{v_t-v_{t+1}} -\tfrac{1}{2}\gnorm{v_{t+1}-v_t}{2}{2} \le \tfrac{1}{2}\gnorm{\rho_t}{2}{2},
	\end{equation*}
	imply that 
	\begin{equation*}
	\inprod{\rho_t}{v_t-x} \le \tfrac{1}{2}\gnorm{x-v_t}{2}{2}-\tfrac{1}{2} \gnorm{x-v_{t+1}}{2}{2}+ \tfrac{1}{2}\gnorm{\rho_t}{2}{2},
	\end{equation*}
	for all $x \in S$. Summing up the above relations from $t = 0$ to $j$ and noting the nonnegativity of $\gnorm{\cdot}{2}{2}$, we obtain \eqref{eq:int_rel114}. Hence we conclude the proof.
\end{proofx}

Now we are ready to prove the lemma converting bound on the primal-dual gap to infeasibility and optimality gap. %Precise statement of the lemma is as follows:
\begin{lemma} \label{thm:linearized_conex_Main}
	Suppose all assumptions in Lemma \ref{thm:linearized_conex_1} are satisfied. Then, for $T\ge 1$, we have
	\begin{equation}\label{eq:bounding_optimality}
	\begin{split}
	\Ebb&[\psi_0(\wb{x}_T)-\psi_0(x^*)] \le \tfrac{1}{\Gamma_T}\big[\gamma_0\eta_0W(x^*, x_0) + \tfrac{\gamma_0\tau_0}{2}\gnorm{y_0}{2}{2} \\
	&+ \tsum_{ t =0}^{T-1}\tfrac{2\gamma_t}{\eta_t- L_{0}-BL_f}\big(\Ebb[\gnorm{\delta^G_t}{\ast}{2}] + H_0^2 \big) +\paran[\big]{\tsum_{t=1}^{T-1}\tfrac{12\gamma_{t}\theta_t^2}{\tau_t} + \tfrac{12\gamma_{T-1}}{\tau_{T-1}}}(\sigma_{f}^2+D_X^2\gnorm{\sigma}{2}{2})\big],
	\end{split}
	\end{equation}
	\begin{equation}\label{eq:bouding_last_iterate}
	\begin{split}
	\gamma_{T-1}(\eta_{T-1}+&\alpha_{0, T-1})\Ebb[W(x^*, x_T)] \le \tfrac{\gamma_0\tau_0}{2} \gnorm{y^*- y_0}{2}{2} +  \gamma_0\eta_0 W(x^*, x_0) \\
	&+\paran[\big]{\tsum_{t=1}^{T-1}\tfrac{12\gamma_{t}\theta_t^2}{\tau_t}
		+ \tfrac{12\gamma_{T-1}}{\tau_{T-1}} }(\sigma_{f}^2+ D_X^2\gnorm{\sigma}{2}{2}) \\
	&+ \tsum_{ t =0}^{T-1} \tfrac{2\gamma_t}{\eta_t- L_0-BL_f}\big\{\Ebb[\gnorm{\delta^G_t}{\ast}{2}] +(H_0 + \gnorm{y^*}{2}{}H_f + \bracket{ \gnorm{y^*}{2}{}-B}_+)^2\big\}  ,
	\end{split}
	\end{equation}
	and
	\begin{align}
	\Ebb[\gnorm{\relu{\psi(\wb{x}_T)}}{2}{}] &\le \tfrac{1}{\Gamma_T} \Big[ 
	\gamma_0\tau_0\gnorm{y_0}{2}{2}+3(\gnorm{y^*}{2}{}+1)^2\gamma_0\tau_0 + \gamma_0\eta_0W(x^*,x_0) \nonumber \\
	&\quad+\tsum_{t =0}^{T-1}\tfrac{2\gamma_t}{\eta_t-L_0-BL_f} \bracket[\big]{\Ebb[\gnorm{\delta^G_t}{*}{2}] + (H_0+(\gnorm{y^*}{2}{}+1)H_f + \tfrac{L_fD_X[\gnorm{y^*}{2}{}+1-B]_+}{2})^2} \nonumber \\
	&\quad+\paran[\big]{\tsum_{ t =1}^{T-1}	\tfrac{12\gamma_t\theta_t^2}{\tau_t}+\tsum_{t=0}^{T-1} \tfrac{\gamma_t}{\tau_t}+ \tfrac{12\gamma_{T-1}}{\tau_{T-1}} }(\sigma_{f}^2+D_X^2\gnorm{\sigma}{2}{2})\Big].\label{eq:bounding_infeasibility}
	\end{align}
	where $\Gamma_T := \tsum_{t =0}^{T-1}\gamma_t$.
\end{lemma}
\begin{proofx}
	Notice that conditional random variables $[G_i(x_t, \xi_t)\vert \xi_{[t-1]}, \wb{\xi}_{[t-2]}],  i = 0, \dots, m$ %and $[G_i(x_t, \xi_t) \vert \xi_{[t-1]},\wb{\xi}_{[t-2]}]$ 
	satisfy properties of SO in \eqref{eq:stochastic_oracle} because $x_t$ is a constant when conditioned on random variables $\xi_{[t-1]} := (\xi_0, \dots, \xi_{t-1})$ and $\wb{\xi}_{[t-2]} := (\wb{\xi}_0, \dots, \wb{\xi}_{t-2})$. Also, observe that, $y_{t+1}$ is a constant when conditioned on random variables $\xi_{[t-1]}$ and $\wb{\xi}_{[t-1]}$. In particular, using \eqref{eq:stochastic_oracle}, we have
	\begin{equation}\label{eq:int_rel157}
	\begin{split}
	\Ebb[\inprod{\delta^G_{t}}{x_{t}-x}]  &= \Ebb\inprod{\Ebb_{\vert\xi_{[t-1]},\wb{\xi}_{[t-1]} }[\delta^G_t]}{x_t-x}= 0,\\
	%	\Ebb[\gnorm{\delta^G_t}{\ast}^2] &\le C \Ebb[\Ebb_{\vert\xi_{[t-1]}, \wb{\xi}_{[t-1]}, \lambda_{[t-1]} }[\gnor m{\delta^G_t}{2}^2] ]= C\Ebb \paran{\sigma_0^2 + \tsum_{i \in [m]}\paran{\sigma_iy_{t+1}^\br{i}}^2 }\\ &\le C(\sigma_0^2 + \gnorm{\sigma}{2}^2\Ebb\gnorm{y_{t+1}}{2}^2 ) \le C\sigma_0^2 + 2C\gnorm{\sigma}{2}^2\paran[\big]{\gnorm{y^*}{2}^2 + \Ebb\gnorm{y_{t+1}-y^*}{2}^2}.
	%	
	\end{split}
	\end{equation}
	for any non-random $x$. This follows due to the following relation
	\begin{align*}
	\Ebb&_{\vert\xi_{[t-1]}, \wb{\xi}_{[t-1]}}[\delta_t^G] \\
	&= \Ebb_{\vert\xi_{[t-1]},\wb{\xi}_{[t-1]}}[G_{0}(x_t, \xi_t) -\subgrad{ f_0}(x_t)] + \Ebb\tsum_{i =1}^my_{t+1}^\br{i} \Ebb_{\vert\xi_{[t-1]},\wb{\xi}_{[t-1]}}[G_i(x_t, \xi_t)-\subgrad{ f_i}(x_t)] = \zero.
	\end{align*}
	Similarly, using \eqref{eq:stochastic_oracle}, we have 
	\begin{equation}\label{eq:int_rel158}
	\Ebb[\inprod{ \delta^F_{t+1} }{y_{t+1} - y}] = \Ebb [\inprod[\big]{ \Ebb_{\vert\xi_{[t]},\wb{\xi}_{[t-1]}}[\delta^F_{t+1}] }{y_{t+1} - y}] = 0,
	\end{equation}
	for any non-random $y$. Here, we note that
	\begin{equation}\label{eq:int_rel161}
	\begin{split}
	\Ebb_{\vert\xi_{[t]}, \wb{\xi}_{[t-1]} } [\delta_{t+1}^F] &= \Ebb_{\vert\xi_{[t]}, \wb{\xi}_{[t-1]} }[F(x_t, \wb{\xi}_t)]-f(x_t) \\
	& \quad+  \paran[\big]{\Ebb_{\vert\xi_{[t]}, \wb{\xi}_{[t-1]}}[{\Gbf}(x_t, \wb{\xi}_t)] -\subgrad{f}(x_t)} ^T (x_{t+1}-x_t) = \zero,
	\end{split}
	\end{equation}
	where the first term in RHS is $\zero$ due to the third relation in \eqref{eq:stochastic_oracle} applied to $\wb{\xi}_t$, 
	the second term is $\zero$ due to the second relation of \eqref{eq:stochastic_oracle} applied to $\wb{\xi}_t$ and the common fact for both the terms that $x_t, x_{t+1}$ are constants for given $\xi_{[t]}, \wb{\xi}_{[t-1]}$.
	We note that 
	\begin{align}
	\Ebb[\gnorm{\delta^F_t}{2}{2}] &\le 2\Ebb[\gnorm{F(x_{t-1}, \wb{\xi}_{t-1})-f(x_{t-1})}{2}{2}] + 2\Ebb[\gnorm{[{\Gbf}(x_{t-1}, \wb{\xi}_{t-1})-\subgrad{ f}(x_{t-1})]^T(x_t-x_{t-1})}{2}{2} ] \nonumber\\
	&\le 2\sigma_f^2+2\Ebb[\tsum_{ i =1}^{m}\braces[\big]{\paran{G_i(x_{t-1}, \wb{\xi}_{t-1})-\subgrad{ f}_i(x_{t-1})}^T(x_t-x_{t-1})}^2 ] \nonumber\\
	&\le 2\sigma_f^2+2\Ebb[\tsum_{i=1}^m\gnorm{G_i(x_{t-1}, \wb{\xi}_{t-1})-\subgrad{f}(x_{t-1})}{\ast}{2}\gnorm{x_t-x_{t-1}}{}{2} ]\nonumber\\
	&\le 2\sigma_f^2 + 2D_X^2\gnorm{\sigma}{2}{2}.\label{eq:int_rel160}
	\end{align}
	{Recall that $\sigma = [\sigma_1, \dots, \sigma_m]^T$.}
	Then, in view of the above relation and definitions of $q_t, \wb{q}_t$, we have
	\begin{equation}\label{eq:int_rel100}
	\begin{split}
	\Ebb[\gnorm{q_t-\wb{q}_t}{2}{2}] &= \Ebb[\gnorm{\ell_F(x_t)-\ell_f(x_t) + \ell_F(x_{t-1})-\ell_f(x_{t-1})}{2}{2}] \\
	&\le 2\Ebb[\gnorm{\delta^F_t}{2}{2}] + 2\Ebb[\gnorm{\delta^F_{t-1}}{2}{2}] \le 8(\sigma_f^2+ D_X^2\gnorm{\sigma}{2}{2}).
	\end{split}
	\end{equation}
	%Here the first two relations follow due to the properties of SO in \eqref{eq:stochastic_oracle}, the third relation holds because $\gnorm{a+b}{2}^2 \le 2\gnorm{a}{2}^2+ 2\gnorm{b}{2}^2$ for all $a$ and $b$ and the fourth relation holds due to Jensen's inequality.
	In \eqref{eq:int_rel153}, choosing $t = T-1 \ge 0$, then taking expectation on both sides, and using relation \eqref{eq:int_rel157}, \eqref{eq:int_rel158} and \eqref{eq:int_rel100}, we have for all non-random\footnote{This $x,y$ is required to be non-random because we are dropping the inner product terms of the left hand side of \eqref{eq:int_rel153}.} $z \in \braces*{(x,y): x \in X , y \ge \zero}$,
	\begin{align}
	\Ebb \bracket{\tsum_{t = 0}^{T-1}\gamma_tQ(z_{t+1}, z)} &\le \tfrac{\gamma_0\tau_0}{2} \gnorm{y- y_0}{2}{2} +  \gamma_0\eta_0 W(x, x_0) +\paran[\big]{\tsum_{t=1}^{T-1}\tfrac{12\gamma_{t}\theta_t^2}{\tau_t}
		+ \tfrac{12\gamma_{T-1}}{\tau_{T-1}} }(\sigma_{f}^2+D_X^2\gnorm{\sigma}{2}{2}) \nonumber\\
	&+ \tsum_{ t =0}^{T-1} \tfrac{2\gamma_t}{\eta_t- L_0-BL_f} \bracket[\big]{ \Ebb[\gnorm{\delta^G_t}{\ast}{2}]+ (H_0+\gnorm{y}{2}{}H_f + \tfrac{L_fD_X\bracket{\gnorm{y}{2}{}-B}_+}{2})^2 } \nonumber\\
	& -\gamma_{T-1}(\eta_{T-1}+\alpha_{0, T-1})\Ebb[W(x, x_T)], %- \gamma_{T-1}\paran*{\tfrac{\eta_{T-1} - L_{0, T-1}}{2} -\tfrac{2M^2}{\tau_{T-1}}} \Ebb \bracket{W(x_T, x_{T-1})} 
	\label{eq:int_rel80}
	\end{align}
	%where we dropped the term $W(x_T, x_{T-1})$ due to the fact that $\tfrac{\eta_{T-1} - L_{0, T-1}}{2} -\tfrac{2M^2}{\tau_{T-1}} \ge 0$ in view of \eqref{eq:int_rel75}. 
	where we dropped $\gnorm{y-y_T}{2}{2}$.
	Using the convexity of $\psi_0(\cdot)$ and $\psi(\cdot)$, and noting the definition of $\Gamma_T$, we have for all non-random $y \ge \zero$ and $x \in X$,
	\begin{equation}\label{eq:int_rel113}
	\Gamma_T\Ebb\big[ \psi_0(\wb{x}_T)  + \inprod{y}{\psi(\wb{x}_T)} - \psi_0(x) - \inprod{\wb{y}_T}{\psi(x)} \big]\le \Ebb \bracket{\tsum_{t = 0}^{T-1}\gamma_tQ(z_{t+1}, z)}.
	\end{equation}
	%	\begin{align}
	%	\Gamma_T\Ebb\big[ \psi_0&(\wb{x}_T)  + \inprod{y}{\psi(\wb{x}_T)} - \psi_0(x) - \inprod{\wb{y}_T}{\psi(x)} \big]\le \gamma_0\eta_0 W(x, x_0) +\tfrac{\gamma_0\tau_0}{2} \gnorm{y_0-y}{2}^2 \nonumber \\
	%	& + \tsum_{ t =0}^{T-1} \tfrac{2\gamma_t}{\eta_t- L_{0}-BL_f}\big( \zeta^2+M(\gnorm{y}{2},B)^2 \big) +\paran[\big]{\tsum_{t=1}^{T-1}\tfrac{12\gamma_{t}\theta_t^2}{\tau_t} + \tfrac{12\gamma_{T-1}}{\tau_{T-1}}}\sigma_{X,f}^2 \label{eq:int_rel113},
	%	\end{align}
	Combining \eqref{eq:int_rel80} and \eqref{eq:int_rel113}, then choosing $x = x^*, y = \zero$ (which are non-random) throughout the combined relation, observing that $[0-B]_+ = 0$ for any $B \ge 0$, ignoring $W(x, x_T)$ term and noting that $\psi(x^*) \le \zero$ and $\wb{y}_T \ge \zero$ implies $\inprod{\wb{y}_T}{\psi(x^*)} \le 0$, we have \eqref{eq:bounding_optimality}.
	
	Now, we prove a bound on $\Ebb[W(x^*,x_T)]$. %under the assumption that $B \ge \gnorm{y^*}{2}+1 $.\\
	Put $z = z^*: = (x^*, y^*)$ in \eqref{eq:int_rel80}. Then %, in view of assumption on $B$,
	we have that $Q(z_{t+1}, z^*) \ge 0$ for all $t = 0,\dots, T-1$. Hence, using $z = z^*$ in \eqref{eq:int_rel80}, dropping summation of $Q$-terms and taking expectation on both sides, we obtain \eqref{eq:bouding_last_iterate}.
	%Here, we dropped $W(x_T, x_{T-1})$ term due to the fact that $\tfrac{\eta_{T-1} - L_{0, T-1}}{2} -\tfrac{2M^2}{\tau_{T-1}} \ge 0$ in view of \eqref{eq:int_rel75}.
	
	Now, we focus our attention to the infeasibility bound. Let us define $R := \gnorm{y^*}{2}{}+1$ and an auxiliary sequence $\{y^v_t\}$ in the following way:  {$y^v_1 = y_1$} and for all {$t \ge 1$}, define
	\[y^v_{t+1} := \argmin_{y \in \BB_+^2(R)} \tfrac{1}{\tau_{t-1}} \inprod{\delta_t^F}{y} + \tfrac{1}{2} \gnorm{y-y^v_t}{2}{2},\]
	where we recall that $\BB_+^2(R) = \{x \in \Rbb^n: \gnorm{x}{2}{} \le R, x\ge  \zero\}$. Then in view of Lemma \ref{lem:tech_res1}, in particular relation \eqref{eq:int_rel114_1}, for all $y \in \BB_+^2(R)$ we have
	\begin{equation}\label{eq:int_rel115}
	\tfrac{1}{\tau_t}\inprod{\delta^F_{t+1}} {y^v_{t+1} -y} \le \tfrac{1}{2}\gnorm{y-y^v_{t+1}}{2}{2} -  \tfrac{1}{2}\gnorm{y-y^v_{t+2}}{2}{2} + \tfrac{1}{2\tau_t^2} \gnorm{\delta_{t+1}^F}{2}{2}.
	\end{equation}
	Multiplying \eqref{eq:int_rel115} by $\gamma_t\tau_t$, taking a sum from $t = 0$ to $T-1$ and noting \eqref{eq:int_rel154-b}, we obtain
	\begin{equation}\label{eq:int_rel117}
	\tsum_{t =0}^{T-1} \gamma_t \inprod{\delta^F_{t+1}} {y^v_{t+1} -y} \le \tfrac{\gamma_0\tau_0}{2}\gnorm{y-y^v_1}{2}{2} + \tsum_{ t =0}^{T-1} \tfrac{\gamma_t}{2\tau_t} \gnorm{\delta_{t+1}^F}{2}{2},
	\end{equation}
	for all $y \in \BB_+^2(R)$.
	Replacing $i$ and $t$ in \eqref{eq:int_rel153} by $t$ and $T-1$ and summing with \eqref{eq:int_rel117}, we obtain
	\begin{align}
	&\tsum_{t = 0}^{T-1}\gamma_tQ(z_{t+1}, z) +\tsum_{t=0}^{T-1}\gamma_t[ \inprod{\delta^G_t}{x_t-x}-\inprod{\delta_{t+1}^F}{y_{t+1}-y_{t+1}^v}] \nonumber \\% - \tsum_{t=0}^{T-1}\gamma_{t}\inprod{\delta_{t+1}^F}{y_{t+1}^v-y}  \\
	&\le \tfrac{\gamma_0\tau_0}{2} [\gnorm{y-y_0}{2}{2} + \gnorm{y-y^v_1}{2}{2}] + \gamma_0\eta_0 W(x,x_0)+\tsum_{t =1}^{T-1} \tfrac{3\gamma_{t}\theta_t^2}{2\tau_t}\gnorm{q_t-\wb{q}_t}{2}{2} + \tfrac{3\gamma_{T-1}}{2\tau_{T-1}}\gnorm{q_T-\wb{q}_T}{2}{2}\nonumber \\
	&\quad+ \tsum_{ t =0}^{T-1}\bracket[\big]{\tfrac{2\gamma_t}{\eta_t-L_{0}- BL_f}\braces{ \gnorm{\delta^G_t}{*}{2} + (H_0+\gnorm{y}{2}{}H_f + \tfrac{L_fD_X[\gnorm{y}{2}{}-B]_+}{2})^2 }+ \tfrac{\gamma_t}{2\tau_t} \gnorm{\delta_{t+1}^F}{2}{2}}, \label{eq:int_rel119}
	\end{align} 
	for all $z \in \braces*{(x,y): x \in X , y \in \BB_+^2(R)}$.
	Note that given $\xi_{[t]}$ and $\wb{\xi}_{[t-1]}$, we have $y_{t+1}, y^v_{t+1}, x_{t+1}$ and $x_{t}$ are constants. Hence we have 
	\begin{equation}\label{eq:int_rel118}
	\Ebb[\inprod{\delta_{t+1}^F}{y_{t+1}-y_{t+1}^v}] = \Ebb[ \inprod{\Ebb_{\vert\xi_{[t]},\wb{\xi}_{[t-1]}}[\delta_{t+1}^F]}{y_{t+1}-y_{t+1}^v} ] = 0,
	\end{equation}
	where second equality follows from \eqref{eq:int_rel161}.
	Choosing $z = \wh{z} :=(x^*, \wh{y})$ in \eqref{eq:int_rel119} where $\wh{y} := (\gnorm{y^*}{2}{}+1) \relu{\psi(\wb{x}_T)}{\gnorm{\relu{\psi(\wb{x}_T)}}{2}{} }^{-1} \in \BB_+^2(R)$%{$\wh{y} := (\gnorm{y^*}{2}+1) \relu{\psi(\wb{x})}{\gnorm{\relu{\psi(\wb{x})}}{2} }^{-1} \in \BB_+^2(B)$}
	, taking expectation on both sides and noting \eqref{eq:int_rel118}, \eqref{eq:int_rel160}, \eqref{eq:int_rel100}, first relation in \eqref{eq:int_rel157}, we have
	\begin{align}
	\Ebb[\tsum_{ t =0}^{T-1}\gamma_tQ(z_{t+1}, \wh{z})] &\le \tfrac{\gamma_0\tau_0}{2}\Ebb[\gnorm{\wh{y}-y_0}{2}{2} + \gnorm{\wh{y}-y^v_1}{2}{2}] + \gamma_0\eta_0W(x^*,x_0) \nonumber \\ 
	&+\tsum_{t =0}^{T-1}\tfrac{2\gamma_t}{\eta_t-L_0-BL_f} \braces[\big]{\Ebb[\gnorm{\delta^G_t}{*}{2}] + \big( H_0+(\gnorm{y^*}{2}{}+1)H_f + \tfrac{L_fD_X[\gnorm{y^*}{2}{}+1-B]_+}{2} \big)^2 }   \nonumber\\
	&+ \paran[\big]{\tsum_{ t =1}^{T-1}	\tfrac{12\gamma_t\theta_t^2}{\tau_t}+\tsum_{t=0}^{T-1} \tfrac{\gamma_t}{\tau_t}+ \tfrac{12\gamma_{T-1}}{\tau_{T-1}} }(\sigma_{f}^2 + D_X^2\gnorm{\sigma}{2}{2}).\label{eq:int_rel162}
	\end{align}
	Noting the convexity of $Q$ in the first argument, we obtain
	\begin{equation}\label{eq:int_rel121}
	\Ebb[Q(\wb{z}_T, \wh{z})] \le \tfrac{1}{\Gamma_T}\Ebb[\tsum_{ t =0}^{T-1}\gamma_tQ(z_{t+1}, \wh{z})].
	\end{equation}
	Now observe that 
	\begin{align*}
	\LL(\wb{x}_T, y^*) - \LL(x^*, y^*) \ge 0\\
	\Rightarrow \psi_0(\wb{x}_T)+ \inprod{y^*}{\psi(\wb{x}_T)} -\psi_0(x^*) \ge 0,
	\end{align*}
	which in view of the relation
	\[\inprod{y^*}{\psi(\wb{x}_T)} \le \inprod{y^*}{\relu{\psi(\wb{x}_T)}} \le \gnorm{y^*}{2}{} \gnorm{\relu{\psi(\wb{x}_T)}}{2}{}, \]
	implies that 
	\begin{equation} \label{eq:int_rel120}
	\psi_0(\wb{x}_T) + \gnorm{y^*}{2}{} \gnorm{\relu{\psi(\wb{x}_T)}}{2}{} -\psi_0(x^*) \ge 0.
	\end{equation}
	Moreover,
	\begin{align*}
	Q(\wb{z}_T, \wh{z}) &= \LL(\wb{x}_T, \wh{y}) - \LL(x^*, \wb{y}_T) \ge \LL(\wb{x}_T, \wh{y}) - \LL(x^*, y^*)= \psi_0(\wb{x}_T) + (\gnorm{y^*}{2}{}+1) \gnorm{\relu{\psi(\wb{x}_T)}}{2}{} - \psi_0(x^*),
	\end{align*}
	along with \eqref{eq:int_rel120} implies that 
	\[Q(\wb{z}_T, \wh{z}) \ge \gnorm{\relu{\psi(\wb{x}_T)}}{2}{}.\]
	The above relation, \eqref{eq:int_rel121} and \eqref{eq:int_rel162} together yield
	\begin{equation*}
	\begin{split}
	\Ebb[\gnorm{\relu{\psi(\wb{x}_T)}}{2}{}] &\le \tfrac{1}{\Gamma_T} \Big[ \tfrac{\gamma_0\tau_0}{2}\Ebb[\gnorm{\wh{y}-y_0}{2}{2} + \gnorm{\wh{y}-y^v_1}{2}{2}] + \gamma_0\eta_0W(x^*,x_0) \\ 
	&\quad+\tsum_{t =0}^{T-1}\tfrac{2\gamma_t}{\eta_t-L_0-BL_f} \braces[\big]{\Ebb[\gnorm{\delta^G_t}{*}{2}] + \big(H_0+(\gnorm{y^*}{2}{}+1)H_f + \tfrac{L_fD_X[\gnorm{y^*}{2}{}+1-B]_+}{2} \big)^2} \\
	&\quad+\paran[\big]{\tsum_{ t =1}^{T-1}	\tfrac{12\gamma_t\theta_t^2}{\tau_t}+\tsum_{t=0}^{T-1} \tfrac{\gamma_t}{\tau_t}+ \tfrac{12\gamma_{T-1}}{\tau_{T-1}} }(\sigma_{f}^2 + D_X^2\gnorm{\sigma}{2}{2})\Big].
	\end{split}
	\end{equation*}
	Noting the bound $\gnorm{\wh{y} - y^v_1}{2}{} \le 2R$ and $\gnorm{\wh{y}-y_0}{2}{2} \le 2\gnorm{y_0}{2}{2} + 2\gnorm{\wh{y}}{2}{2} \le \gnorm{y_0}{2}{2} + 2R^2$ in the above relation and recalling that $R = \gnorm{y^*}{2}{}+1$, we obtain \eqref{eq:bounding_infeasibility}.
	Hence we conclude the proof.
%	\qed
\end{proofx}
Note that we still need to bound $\Ebb[\gnorm{\delta^G_t}{\ast}{2}]$. Below, we provide a simple lemma which is used to show such a bound.
\begin{lemma}\label{prop:recurrence_bound}
	Let $\{a_t\}_{t\ge 0}$ be a nonnegative sequence, $m_1,m_2 \ge 0$ be constants such that $a_0 \le m_1$ and the following relation holds for all $t \ge 1$:
	\[a_t \le m_1  + m_2\tsum_{ k =0}^{t-1}a_k.\]
	Then we have $a_t \le m_1 (1+m_2)^t$.
\end{lemma}
\begin{proofx}
	We prove this lemma by induction. Clearly, it is true for $t = 0$. Suppose it is true for $a_t$. Then, using inductive hypothesis on $a_k$ for $k= 0, \dots, t$, we have 
	\begin{align*}
	a_{t+1} &\le m_1+ m_2\tsum_{ k =0}^{t}a_t\\
	&\le m_1 \bracket[\big]{1+m_2\tsum_{ k =0}^{t}(1+m_2)^k}\\
	&\le m_1\bracket[\big]{1+m_2\tfrac{(1+m_2)^{t+1}-1}{m_2}} = m_1(1+m_2)^{t+1}.
	\end{align*}
	Hence, we conclude the proof.
\end{proofx}
Now, under some assumptions, we show a bound on $\Ebb[\gnorm{\delta^G_{t}}{\ast}{2}]$.
\begin{lemma}\label{lem:bound_zeta}
	Assume that $\{\gamma_t, \tau_t, \eta_t\}$ satisfy 
	\begin{equation}\label{eq:int_rel159}
	\tfrac{96\gnorm{\sigma}{2}{2}}{\tau_t(\eta_t-L_0-BL_f)} < 1 ,
	\end{equation} for all $t \le T-1$ and there exist constants $R_1$ and $R_2$ satisfying%\comment[id=qd]{grammar; comma in 2.68-2.69}
	\begin{align}
	R_1 \ge \paran[\Big]{1- 	\tfrac{96\gnorm{\sigma}{2}{2}}{\tau_t(\eta_t-L_0-BL_f)} }^{-1}&\Big[ 2\sigma_0^2 +\tfrac{48\gnorm{\sigma}{2}{2}}{\gamma_t\tau_t}\Big\{ \gamma_0\eta_0W(x^*,x_0) + \tfrac{\gamma_0\tau_0}{2}\gnorm{y^*-y_0}{2}{2}+\tfrac{\gamma_t\tau_t}{12}\gnorm{y*}{2}{2}  \nonumber\\
	&+\tsum_{ i =0}^{t}\tfrac{2\gamma_i}{\eta_i-L_0-BL_f} \big( H_0 + H_f\gnorm{y^*}{2}{}+\tfrac{L_fD_X[\gnorm{y^*}{2}{}-B]_+}{2} \big)^2 \nonumber \\
	&+ \paran[\big]{\tsum_{ i =1}^{t}\tfrac{12\gamma_i\theta_i^2}{\tau_i} + \tfrac{12\gamma_t}{\tau_t}} (\sigma_{f}^2+D_X^2\gnorm{\sigma}{2}{2}) \Big\} \Big] \label{eq:R1},
	\end{align} for all $t \le T-1$ and
	\begin{equation}\label{eq:R2}
	R_2 \ge \paran[\Big]{1- 	\tfrac{96\gnorm{\sigma}{2}{2}}{\tau_t(\eta_t-L_0-BL_f)} }^{-1}\tfrac{96\gnorm{\sigma}{2}{2}\gamma_i}{\gamma_t\tau_t(\eta_i-L_0-BL_f)},
	\end{equation}for all $t \le T-1$ and $i \le t-1$.
	Then, we have
	\begin{equation}\label{eq:bound_delta_G_t}
	\Ebb[\gnorm{\delta_t^G}{\ast}{2}] \le R_1(1+R_2)^t,
	\end{equation}
	for all $t \le T-1$.
	In particular, if $\gnorm{\sigma}{2}{} = 0$, then we can set $R_1 = 2\sigma_0^2$ and $R_2 = 0$ implying $\Ebb[\gnorm{\delta^G_t}{\ast}{2}] \le 2\sigma_0^2$.
\end{lemma}
\begin{proofx}
	Observe that $Q(z_{t+1}, z^*) \ge 0$ for all $t = 0, \dots, T-1$ where $z^* = (x^*, y^*)$. 	Choosing $z = z^* $ in \eqref{eq:int_rel153} and %for $T$ substituted by $t+1 (\ge 1)$, 
	taking expectation on both sides, using \eqref{eq:int_rel157} with $x= x^*$ and \eqref{eq:int_rel158} with $y = y^*$ and noting \eqref{eq:int_rel100}, we have
	\begin{align}
	&\tfrac{\gamma_t\tau_t}{12}\Ebb \gnorm{y^*-y_{t+1}}{2}{2} \le \gamma_0\eta_0W(x^*,x_0) + \tfrac{\gamma_0\tau_0}{2}\gnorm{y^*-y_0}{2}{2} % +  \tsum_{i =0}^{T-1} \gamma_i[\inprod{\delta^G_i}{x_i-x^*} - \inprod{\delta^F_{i+1}}{y_{i+1}-y^*} ] 
	\nonumber \\
	&\quad+\tsum_{ i =0}^{t} \tfrac{2\gamma_i}{\eta_i-L_0-BL_f}\bracket[\big]{\Ebb\gnorm{\delta_i^G}{*}{2} +(H_0 + H_f\gnorm{y^*}{2}{}+\tfrac{L_fD_X[\gnorm{y^*}{2}{}-B]_+}{2})^2} \nonumber\\
	& \quad+ \paran[\big]{\tsum_{ i =1}^{t}\tfrac{12\gamma_i\theta_i^2}{\tau_i}+ \tfrac{12\gamma_t}{\tau_t}}(\sigma_{f}^2+D_X^2\gnorm{\sigma}{2}{2}). \label{eq:int_rel156}
	\end{align}
	Now, let us define $\delta^G_{t,i} := G_i(x_t, \xi_t) -f'_i(x_t)$ for $ i = 0, \dots,m$. As a consequence, we have $\delta^G_t = \delta^G_{t,0} + \tsum_{ i =1}^m y^\br{i}_{t+1}\delta^G_{t,i}$. Then, we have 
	\begin{align}
	\Ebb[\gnorm{\delta^G_t}{\ast}{2} ] &=\Ebb[ \gnorm{\delta^G_{t,0}+\tsum_{ i =1}^m y^\br{i}_{t+1} \delta^G_{t,i}}{\ast}{2} ] \nonumber \\[-1.5mm]
	&\mleq{(i)} 2\Ebb\bracket{\gnorm{\delta^G_{t,0} }{\ast}{2} } + 2\Ebb\bracket[\big]{ \gnorm{\tsum_{ i =1}^my^\br{i}_{t+1}\delta^G_{t,i}}{\ast}{2} } \nonumber \\
	&\le 2\Ebb\bracket{\gnorm{\delta^G_{t,0} }{\ast}{2} } + 2\Ebb\bracket[\big]{ \paran{\tsum_{ i =1}^m\gnorm{y^\br{i}_{t+1}\delta^G_{t,i}}{\ast}{} }^2 } \nonumber \\[-1.5mm]
	&\mleq{{(ii)}} 2 \braces[\big]{\sigma_0^2 +  \Ebb\big[ \gnorm{y_{t+1}}{2}{2} \paran[\big]{\tsum_{i =1}^m\gnorm{\delta^G_{t,i}}{\ast}{2} }  \big] } \nonumber \\[-1.5mm]
	&\mleq{(iii)} 2 \braces[\big]{\sigma_0^2 + \Ebb\bracket[\big]{\gnorm{y_{t+1}}{2}{2} \paran[\big]{\tsum_{ i =1}^m\Ebb_{\vert\xi_{[t-1]}, \wb{\xi}_{[t-1]}}[\gnorm{\delta^G_{t,i}}{\ast}{2}]} } }\nonumber\\[-1.5mm]
	&\mleq{(iv)}  2\braces[\big]{\sigma_0^2 + \Ebb\bracket{\gnorm{y_{t+1}}{2}{2}\tsum_{i=1}^m\sigma_i^2} }\nonumber\\ 
	&= 2(\sigma_0^2 + \gnorm{\sigma}{2}{2}\Ebb\gnorm{y_{t+1}}{2}{2} )\nonumber\\
	&\le 2\sigma_0^2 + 4\gnorm{\sigma}{2}{2}\paran[\big]{\gnorm{y^*}{2}{2} + \Ebb\gnorm{y_{t+1}-y^*}{2}{2} }.\label{eq:int_rel167}	
	\end{align}
	Here, relation (i) follows from the fact that $\gnorm{a+b}{\ast}{2} \le (\gnorm{a}{\ast}{}+ \gnorm{b}{\ast}{})^2 \le 2\gnorm{a}{\ast}{2} + 2\gnorm{b}{\ast}{2}$, relation (ii) follows from Cauchy-Schwarz inequality, relation (iii) follows from the fact that $y_{t+1}$, conditioned on random variables $\xi_{[t-1]}, \wb{\xi}_{[t-1]}$, is a constant and relation (iv) follows from fourth and fifth relation in \eqref{eq:stochastic_oracle} and the fact that $x_t$ is a constant when conditioned on random variables $\xi_{[t-1]}, \wb{\xi}_{[t-1]}$.\\
	Adding $\tfrac{\gamma_t\tau_t}{12}\gnorm{y^*}{2}{2}$ to both sides of \eqref{eq:int_rel156}, then multiplying the resulting relation by $\tfrac{48\gnorm{\sigma}{2}{2}}{\gamma_t\tau_t}$ and observing \eqref{eq:int_rel167}, we have 
	\begin{align*}
	\Ebb[\gnorm{\delta_t^G}{\ast}{2}]&\le 2\sigma_0^2 +\tfrac{48\gnorm{\sigma}{2}{2}}{\gamma_t\tau_t}\Big\{\gamma_0\eta_0W(x^*,x_0) + \tfrac{\gamma_0\tau_0}{2}\gnorm{y^*-y_0}{2}{2} + \tfrac{\gamma_t\tau_t}{12}\gnorm{y^*}{2}{2}\\
	&+\tsum_{ i =0}^{t}\tfrac{2\gamma_i}{\eta_i-L_0-BL_f}(H_0 + H_f\gnorm{y^*}{2}{}+\tfrac{L_fD_X[\gnorm{y^*}{2}{}-B]_+}{2})^2 \nonumber\\
	& + \paran[\big]{\tsum_{ i =1}^{t}\tfrac{12\gamma_i\theta_i^2}{\tau_i} + \tfrac{12\gamma_t}{\tau_t}} (\sigma_{f}^2+D_X^2\gnorm{\sigma}{2}{2})  \Big\} + \tsum_{ i =0}^{t}\tfrac{96\gnorm{\sigma}{2}{2}\gamma_i}{\gamma_t\tau_t(\eta_i-L_0-BL_f)}\Ebb\gnorm{\delta_i^G}{*}{2}.
	\end{align*}
	In view of \eqref{eq:int_rel159}, we have that the coefficient of the $\delta_t^G$ term on the right hand side of the above relation is strictly less than $1$. Moving the $\delta_t^G$ term to the left hand side and noting the conditions imposed on constants $R_1, R_2$, we have
	\begin{equation*}
	\Ebb[\gnorm{\delta_t^G}{\ast}{2}] \le R_1+R_2\tsum_{ i =0}^{t-1}\Ebb[\gnorm{\delta_i^G}{\ast}{2}],
	\end{equation*}
	for all $t \le T-1$.
	Using Lemma \ref{prop:recurrence_bound} for the above relation, we have \eqref{eq:bound_delta_G_t}. Hence we conclude the proof.
\end{proofx}
\begin{remark}
	Note that the bound in \eqref{eq:bound_delta_G_t} is still a function of stepsize parameters since $R_1$ and $R_2$ need to satisfy relations \eqref{eq:R1} and \eqref{eq:R2}, respectively. Now, we need to show that there exists a possible selection of stepsize parameters for which we can compute a uniform upper bound on $\Ebb[\gnorm{\delta^G_t}{\ast}{2}]$ for all $t \le T-1$, in particular, we can obtain constants $R_1$ and $R_2$ satisfying \eqref{eq:R1} and \eqref{eq:R2}, respectively. Moreover, the selected stepsize policy is meaningful in the sense that it yields convergence according \eqref{eq:bounding_optimality} and \eqref{eq:bounding_infeasibility}. Below, we show that the stepsize policy in \eqref{eq:step_size} of Theorem \ref{cor:step_size_strong_cvx} and in \eqref{eq:step_size_conv} of Theorem \ref{thm:convergence_convex} is specified in a way such that \eqref{eq:int_rel154}, \eqref{eq:int_rel155-1}, \eqref{eq:int_rel155-2} and \eqref{eq:int_rel159} are satisfied. Moreover, a uniform upper bound according to \eqref{eq:bound_delta_G_t} for all $t \le T-1$ can be obtained and it also leads to the convergence according to \eqref{eq:bounding_optimality} and \eqref{eq:bounding_infeasibility}. In particular, we show the proofs of Theorem \ref{cor:step_size_strong_cvx} and Theorem \ref{thm:convergence_convex} below.
\end{remark}

First, we focus on the setting in which \eqref{main-prob} is strongly convex, i.e., $\alpha_{0} > 0$
and show the proof of Theorem \ref{cor:step_size_strong_cvx}.%In this case, we provide a new and simple stepsize scheme which satisfies \eqref{eq:int_rel154} with $\alpha_{0} = 0$, \eqref{eq:int_rel155} and \eqref{eq:int_rel159}.

\vspace{0.1in}

\begin{proof*}{Theorem \ref{cor:step_size_strong_cvx}}
	Note that $\{\gamma_t, \theta_t, \eta_t, \tau_t\}$ set according to \eqref{eq:step_size} satisfy \eqref{eq:int_rel154}. It is easy to verify 
	\eqref{eq:int_rel154-a} and \eqref{eq:int_rel154-b}. To verify \eqref{eq:int_rel154-c}, note that
	\begin{align*}
	\gamma_{t-1}(\eta_{t-1}+\alpha_{0,t-1}) &\ge \gamma_{t-1}(\eta_{t-1}+\alpha_{0}) \\
	&=(t+t_0+1)\paran[\big]{ \tfrac{\alpha_{0}(t+t_0)}{2} + \alpha_{0} } = \tfrac{\alpha_{0}}{2}(t+t_0+1)(t+t_0+2)
	= \gamma_t\eta_t.
	\end{align*}
	Note that {all relations in \eqref{eq:int_rel155-1} and \eqref{eq:int_rel155-2} are} satisfied if $\tfrac{4}{3}\MM^2 \le \tfrac{\tau_t(\eta_{t-2}-L_0-BL_f)}{12}$. This follows from the fact that $\{\eta_t\}$ is an increasing sequence, { $\{\tau_t\}$ is a deacreasing sequence,} $\tfrac{3}{4} \le \theta_t < 1$ and the definition of $\MM$. Indeed we have,
	\begin{align*}
	\tfrac{\tau_t(\eta_{t-2}-L_0-BL_f)}{12} \ge \tfrac{32\MM^2}{12\alpha_{0}(t+1)}\paran[\big]{\tfrac{\alpha_{0}(t+t_0-1)}{2} -\tfrac{\alpha_{0}(t_0-2)}{4}} =\tfrac{2(2t+t_0)\MM^2}{3(t+1)} \ge  \tfrac{4\MM^2}{3},
	\end{align*}
	where the last inequality follows from $t_0\ge 2$ by definition.
	Also note that
	\begin{align*}
	\tau_t(\eta_t-L_0-BL_f) &\ge \tfrac{384\gnorm{\sigma}{2}{2}T}{\alpha_0(t+1)} \paran[\big]{\tfrac{\alpha_{0}(t+t_0+1)}{2} - \tfrac{\alpha_{0}(t_0-2)}{4}} = \tfrac{96(2t+t_0+4)\gnorm{\sigma}{2}{2}T}{t+1} \ge 192\gnorm{\sigma}{2}{2}
	\end{align*} for all $t \ge 0$.
	Then the above relation implies that
	\begin{equation}\label{eq:int_rel164}
	\tfrac{96\gnorm{\sigma}{2}{2}}{\tau_t(\eta_t-L_0-BL_f)} \le \tfrac{1}{2},
	\end{equation}
	for all $t \ge 0$.
	Finally, we need to show the existence of constants $R_1$ and $R_2$ satisfying \eqref{eq:R1} and \eqref{eq:R2}, respectively.
	Using the fact that $\tau_t \ge \tfrac{384\gnorm{\sigma}{2}{2}T}{\alpha_0(t+1)}$, we observe
	\begin{equation}\label{eq:int_rel165}
	\tfrac{96\gnorm{\sigma}{2}{2}\gamma_i}{\gamma_t\tau_t(\eta_i-L_0-BL_f)} \le \tfrac{384\gnorm{\sigma}{2}{2}(i+t_0+2)}{\alpha_{0}(2i+t_0+4)}\tfrac{\alpha_0(t+1)}{384\gnorm{\sigma}{2}{2}(t+t_0+2)T} \le\tfrac{1}{T},
	\end{equation}
	for all $i \ge 0, t \ge 0$.
	Noting \eqref{eq:int_rel164}, \eqref{eq:int_rel165} and \eqref{eq:R2}, we can set \begin{equation}\label{eq:R2_val}
	R_2 := \tfrac{2}{T}.
	\end{equation} 
	Noting \eqref{eq:R1} along with definition of $\HH_*$ in the theorem statement, setting $y_0 = \zero$, using \eqref{eq:int_rel164},\eqref{eq:int_rel100}, and applying the following relations
	\begin{align*}
	\gamma_t\tau_t &\ge \max\big\{ \tfrac{384\gnorm{\sigma}{2}{2}T}{\alpha_0}, \tfrac{\sigma_{X, f}T^{3/2}}{B(t_0+2)^{1/2}} \big\},\\
	\tsum_{ i =0}^{t}\tfrac{\gamma_i}{\eta_i-L_0-BL_f} &\le \tfrac{4(t+1)}{\alpha_0},\\
	\tsum_{ i =1}^{t} \tfrac{\gamma_i\theta_i^2}{\tau_i} +\tfrac{\gamma_{t}	}{\tau_t} &\le \tfrac{B(t_0+2)^{1/2}}{\sigma_{X, f}T^{3/2}} \bracket[\big]{\tfrac{(t+1)^3}{3} + \tfrac{(t+1)^2(t_0+2)}{2} + \tfrac{(t+1)(9t_0+10)}{6} - (t_0+1) },
	\end{align*}
	we have for all $t \le T-1$, RHS of \eqref{eq:R1} is at most
	\begin{align*}
	2&\Big[2\sigma_0^2 + 48\gnorm{\sigma}{2}{2}\Big\{ \paran[\big]{\tfrac{t_0+2}{2} + \tfrac{1}{12}} \gnorm{y^*}{2}{2}  + \tfrac{8T\HH_*^2}{\alpha_0} \tfrac{T}{T+t_0+1} \tfrac{\alpha_0}{384\gnorm{\sigma}{2}{2}T} \\
	&+ \tfrac{12\sigma_{X,f}^2B(t_0+2)^{1/2}}{\sigma_{X, f}T^{3/2}}  \paran[\Big]{\tfrac{B(t_0+2)^{1/2}}{\sigma_{X, f}T^{3/2}}\tfrac{T^3}{3} + 
	\tfrac{\alpha_0}{384\gnorm{\sigma}{2}{2}T}\paran[\big]{\tfrac{T^2(t_0+2)}{2} + \tfrac{T(9t_0+10)}{6} -(t_0+1)} } \Big\}\Big].
	\end{align*}
	Then, noting $\tfrac{1}{T} \le 1$ and ignoring $-(t_0+1)$ term, we can set 
	\begin{equation}\label{eq:R1_val}
	R_1 := 2\bracket[\Big]{2\sigma_0^2 + 24(t_0+3)\gnorm{\sigma}{2}{2}\gnorm{y^*}{2}{2} + \HH_*^2+  4\times 48(t_0+2)B^2\gnorm{\sigma}{2}{2} + 3\alpha_0B\sigma_{X, f}(t_0+2)^{3/2} }.
	\end{equation}
	Then using Lemma \ref{lem:bound_zeta} and noting \eqref{eq:R2_val}, we have for all $t \le T-1$
	\[\Ebb[\gnorm{\delta_t^G}{\ast}{2}] \le \begin{cases}
	2\sigma_0^2 &\text{if }\gnorm{\sigma}{2}{} = \sigma_{f} =0;\\
	R_1\big(1 + \tfrac{2}{T}\big)^{T-1} \le R_1e^2 &\text{otherwise}.
	\end{cases} .\]
	Noting the above relation, \eqref{eq:R1_val} and the definition of $\zeta$%Note that for $T \ge \begin{cases} 1 &\text{if }\gnorm{\sigma}{2} = \sigma_{f} =0;\\\big( \tfrac{8\sqrt{6}B\gnorm{\sigma}{2}}{\ln{2}} \big)^2(t_0+2) &\text{otherwise.}\end{cases}$
	, we have 
	\begin{equation}\label{eq:int_rel168}
	\Ebb[\gnorm{\delta^G_t}{\ast}{2}] \le \zeta^2, \tab \forall \ t \le T-1.
	\end{equation} 
	So according to \eqref{eq:bounding_optimality} with $y_0 = \zero$ and using \eqref{eq:int_rel168}, we have
	\begin{equation*}
	\begin{split}
	\Ebb [\psi_0(\wb{x}_T) -\psi_0(x^*)]&\le \tfrac{2}{T(T+2t_0+3)} \big[\tfrac{\alpha_h(t_0+1)(t_0+2)}{2}W(x^*, x_0) + \tfrac{8(\zeta^2+H_0^2)T}{\alpha_0} \\
	&\quad + 12B(t_0+2)^{1/2}\sigma_{X,f} \braces[\big]{\tfrac{T^{1/2}(T+2)}{3} + \tfrac{(t_0+1)T^{-1/2}(T+3)}{2}}
	\big].
	\end{split}
	\end{equation*} 
	Here we used the bound
	\begin{equation}\label{eq:int_rel138}
	\begin{split}
	\tfrac{\gamma_t}{\eta_t-L_{0}-BL_f} &\le \tfrac{4}{\alpha_{0}} \text{ for all }t \ge 0,\\
	\tsum_{ t =1}^{T-1} \tfrac{\gamma_t\theta_t^2}{\tau_t} + \tfrac{\gamma_{T-1}}{\tau_{T-1}} &\le \tfrac{B(t_0+2)^{1/2}}{\sigma_{X,f} T^{3/2}} \bracket[\big]{\tfrac{T^2(T+2)}{3} + (t_0+1)\tfrac{T(T+3)}{2}}.
	\end{split}
	\end{equation}
	Noting the bound on $W(x^*, x_0)$ in the earlier relation, we obtain \eqref{eq:bounding_optimality_cor}.
	Using \eqref{eq:bounding_infeasibility}, \eqref{eq:int_rel168} and the bounds in \eqref{eq:int_rel138}, we have
	\begin{align} 
	\Ebb \gnorm*{\relu{\psi(\wb{x}_T)}}{2}{} &\le \tfrac{2}{T(T+2t_0+3)} \big[ 3(t_0+2)(\gnorm{y^*}{2}{}+1)^2\max\braces[\big]{ \tfrac{32\MM^2}{\alpha_0}, \tfrac{\sigma_{X,f} T^{3/2}}{B(t_0+2)^{1/2}}, \tfrac{384\gnorm{\sigma}{2}{2}T}{\alpha_0} }\nonumber \\
	&+ \tfrac{\alpha_{0}(t_0+1)(t_0+2)}{2}W(x^*,x_0) + 13B(t_0+2)^{1/2}\sigma_{X,f}\braces[\big]{\tfrac{T^{1/2}(T+2)}{3} + \tfrac{(t_0+1)T^{-1/2}(T+3)}{2}}  \nonumber \\
	&+\tfrac{8T}{\alpha_{0}}\braces[\big]{\zeta^2+ \bracket[\big]{H_0+(\gnorm{y^*}{2}{}+1)H_f + \tfrac{L_fD_X[\gnorm{y^*}{2}{}+1-B]_+}{2}}^2} 
	\big].\label{eq:int_rel106}
	\end{align}
	Noting the bound on $W(x^*, x_0)$ in \eqref{eq:int_rel106}, the definition of $\HH_*$, using the fact that $\tfrac{T^{1/2}(T+2)}{3} \le T^{3/2}$ and combining 
	the $T^{3/2}$ order terms, we obtain \eqref{eq:bounding_infeasibility_cor}.
	From \eqref{eq:bouding_last_iterate}, we have 
	\begin{align*}
	\Ebb[W(x_T, x^*)] &\le \tfrac{2}{\alpha_{0}(T+t_0+1)(T+t_0+2)}\big[ \tfrac{(t_0+2)\gnorm{y^*}{2}{2}}{2} \max\braces[\big]{ \tfrac{32\MM^2}{\alpha_0}, \tfrac{\sigma_{X,f} T^{3/2}}{B(t_0+2)^{1/2}}, \tfrac{384\gnorm{\sigma}{2}{2}T}{\alpha_0} } \nonumber\\ 
	&+ \tfrac{\alpha_{0}(t_0+1)(t_0+2)}{2}W(x^*,x_0) + 12B(t_0+2)^{1/2}\sigma_{X,f}\braces[\big]{\tfrac{T^{1/2}(T+2)}{3} + \tfrac{(t_0+1)T^{-1/2}(T+3)}{2}} \nonumber \\
	&+\tfrac{8T}{\alpha_{0}}\braces[\big]{\zeta^2+ \bracket[\big]{H_0+\gnorm{y^*}{2}{}H_f + \tfrac{L_fD_X[\gnorm{y^*}{2}{}-B]_+}{2}}^2} \big].
	\end{align*}
	With similar replacements in the above relation as in \eqref{eq:int_rel106}, we obtain \eqref{eq:bounding_last_iterate_cor}.
	Hence we conclude the proof.
%	\qed
\end{proof*}

Now, we show proof of Theorem \ref{thm:convergence_convex}.

\begin{proof*}{Theorem~\ref{thm:convergence_convex}}
	It is easy to verify that $\{\gamma_{t}, \theta_t, \eta_t, \tau_t\}$ set according to \eqref{eq:step_size_conv} satisfy \eqref{eq:int_rel154} with $\alpha_{0} = 0$.
	Note that {all relations in \eqref{eq:int_rel155-1} and \eqref{eq:int_rel155-2} are }%{\eqref{eq:int_rel155} is}
	satisfied if $\MM^2 \le \tfrac{\tau_t(\eta_{t-2}-L_0-BL_f)}{12}$. This follows due to the fact that $\{\eta_t\}$ is an non-decreasing sequence, {$\{\tau_t\}$ is a non-increasing sequence}, $\theta_t = 1$ for all $t \ge 0$ and the definition of $\MM$.  Then we have
	\begin{equation*}
	\tfrac{\tau_t(\eta_{t-2}-L_0-BL_f)}{12} \ge \tfrac{6\MM B}{D_X} \tfrac{2\MM D_X}{B} \times\tfrac{1}{12} = \MM^2.
	\end{equation*}
	Also, since $(\eta_t-L_0-BL_f) \ge \tfrac{24B\gnorm{\sigma}{2}{} }{D_X}$ and $\tau_t \ge\tfrac{8D_X\gnorm{\sigma}{2}{} }{B}$, we have \begin{equation*}
	\tau_t(\eta_t-L_0-BL_f) \ge 192\gnorm{\sigma}{2}{2}
	\end{equation*} for all $t \ge 0$. In view of the above relation, we have 
	\begin{equation}\label{eq:int_rel166}
	\tfrac{96\gnorm{\sigma}{2}{2}}{\tau_t(\eta_t-L_0-BL_f)} \le \tfrac{1}{2},
	\end{equation}
	hence \eqref{eq:int_rel159} is satisfied.
	We also need to show the existence of $R_1$ and $R_2$ satisfying \eqref{eq:R1} and \eqref{eq:R2}, respectively.
	Using the fact that $\gamma_t, \eta_t$ and $\tau_t$ are constants for all $t \ge 0$, $\tau\eta \ge \tfrac{96T\sigma_{X, f}\gnorm{\sigma}{2}{}}{D_X}$ and noting \eqref{eq:int_rel166}, we obtain
	\begin{equation*}
	\paran[\big]{1-\tfrac{96\gnorm{\sigma}{2}{2}}{\tau_t(\eta_t-L_0-BL_f)}}^{-1} \tfrac{96\gnorm{\sigma}{2}{2}\gamma_i}{\gamma_t\tau_t(\eta_i-L_0-BL_f)} \le 2\tfrac{96\gnorm{\sigma}{2}{2}}{\tau\eta} \le 2\tfrac{\gnorm{\sigma}{2}{}D_X}{T\sigma_{X, f}} \le \tfrac{{2}}{T},
	\end{equation*}
	where in the last relation, we used the fact that $\sigma_{X, f} \ge D_X\gnorm{\sigma}{2}{}$.
	In view of the above relation and \eqref{eq:R2}, we can set 
	\begin{equation}\label{eq:R2_val_cvx}
	R_2 := \tfrac{{2}}{T}.
	\end{equation}
	Noting \eqref{eq:R1} along with the fact that $\HH_* \ge H_0 + H_f\gnorm{y^*}{2}{} +\tfrac{L_fD_X[\gnorm{y^*}{2}{}-B]_+}{2} $, setting $y_0 =\zero$, using \eqref{eq:int_rel166}, \eqref{eq:int_rel100}, $\gamma_t\tau_t = \tau \ge \sqrt{96T}\sigma_{X, f}$, $\tsum_{ i =0}^{t} \tfrac{\gamma_i}{\eta_i-L_0-BL_f} = \tfrac{t+1}{\eta} \le \tfrac{\sqrt{T}D_X}{\sqrt{2[\HH_*^2+\sigma_0^2+48B^2\gnorm{\sigma}{2}{2}] } }$, and $\tsum_{ i =1}^{t} \tfrac{\gamma_i\theta_i^2}{\tau_i} + \tfrac{\gamma_t}{\tau_t} =\tfrac{t+1}{\tau}\le \tfrac{T}{\tau}$ %\le \tfrac{T(\gnorm{y^*}{2}+1)}{2\sqrt{T}\sigma_{X, f}}
	for all $t \le T-1$, we can see that the RHS of \eqref{eq:R1} is at most
	\begin{align*}
	&2\big[2\sigma_0^2+ 48\gnorm{\sigma}{2}{2}\big\{ \tfrac{7}{12}\gnorm{y^*}{2}{2} + \tfrac{\eta}{\tau}D_X^2 + \tfrac{\sqrt{2T}D_X\HH_*^2}{\sqrt{\HH_*^2+\sigma_0^2+ 48B^2\gnorm{\sigma}{2}{2}}} \tfrac{B}{\sqrt{96T}\sigma_{X, f}} + 12\sigma_{X, f}^2 \tfrac{T}{\tau^2}\big\} \big]\\
	&\le2\big[2\sigma_0^2 + 48\gnorm{\sigma}{2}{2}\big\{ \tfrac{7}{12}\gnorm{y^*}{2}{2} + \tfrac{\eta}{\tau}D_X^2+ \tfrac{D_XB\HH_*}{\sqrt{48}\sigma_{X, f}} + 12T\sigma_{X, f}^2\tfrac{B^2}{96T\sigma_{X, f}^2}\big\}\big]\\
	&\le 2\big[2\sigma_0^2 + 48\gnorm{\sigma}{2}{2}\big\{\tfrac{7}{12}\gnorm{y^*}{2}{2}+ \tfrac{D_X}{\sigma_{X, f}}\paran[\big]{B\sqrt{\tfrac{[\HH_*^2+\sigma_0^2 + 48B^2\gnorm{\sigma}{2}{2}]}{48}} + \tfrac{B\HH_*}{\sqrt{48}}} +\tfrac{6\max\{\MM, 4\gnorm{\sigma}{2}{}\}BD_X}{2\max\{\MM, 4\gnorm{\sigma}{2}{}\}}\tfrac{B}{D_X}+\tfrac{B^2}{8} \big\}\big]\\
	&\le 2\big[2\sigma_0^2+ 28\gnorm{\sigma}{2}{2}\gnorm{y^*}{2}{2} + 75B^2\gnorm{\sigma}{2}{2} +\sqrt{48} \gnorm{\sigma}{2}{}[2B\HH_*+ (B\sigma_0+\sqrt{48}B^2\gnorm{\sigma}{2}{})] \big]
	\end{align*}
	where in the last inequality, we used the fact that $\tfrac{\gnorm{\sigma}{2}{}D_X}{\sigma_{X, f}} \le 1$. Note that the last term in the above sequence of relations is a constant %vis-à-vis the problem parameters 
	satisfying the requirement in \eqref{eq:R1}. Hence we can set 
	\begin{equation}\label{eq:R1_val_cvx}
	R_1:=2\big[2\sigma_0^2+ 28\gnorm{\sigma}{2}{2}\gnorm{y^*}{2}{2} + 75B^2\gnorm{\sigma}{2}{2} +\sqrt{48} \gnorm{\sigma}{2}{}[2B\HH_*+ (B\sigma_0+\sqrt{48}B^2\gnorm{\sigma}{2}{})] \big].
	\end{equation} 
	%Noting the definition of $\zeta$ and using Lemma \ref{lem:bound_zeta},  for all $t \le T-1$, we 	
	Then using Lemma \ref{lem:bound_zeta} and noting \eqref{eq:R2_val_cvx}, we have for all $t \le T-1$
	\[\Ebb[\gnorm{\delta_t^G}{\ast}{2}] \le \begin{cases}
	2\sigma_0^2 &\text{if }\gnorm{\sigma}{2}{} = \sigma_{f} =0;\\
	R_1\big(1 + \tfrac{2}{T}\big)^{T-1} \le R_1e^2 &\text{otherwise}.
	\end{cases} .\]
	Noting the above relation, \eqref{eq:R1_val_cvx} and the definition of $\zeta$, we have 
	\begin{equation}\label{eq:int_rel170}
	\Ebb[\gnorm{\delta^G_t}{\ast}{2}] \le \zeta^2, \tab \forall \ t \le T-1.
	\end{equation}
	So according to \eqref{eq:bounding_optimality} with $y_0 = \zero$ and using \eqref{eq:int_rel170}, we have
	\begin{align*}
	\Ebb &[\psi_0(\wb{x}_T) -\psi_0(x^*)]\le \tfrac{1}{T} \big[(\eta+L_0+BL_f)W(x^*, x_0) +  \tfrac{2(\zeta^2+H_0^2)T}{\eta}  + 12\sigma_{X,f}^2 \tfrac{T}{\tau} \big].
	\end{align*} 
	Using the bound $W(x^*, x_0)\le D_X^2$, we obtain \eqref{eq:bounding_opt_convex}.
	From \eqref{eq:bounding_infeasibility} and \eqref{eq:int_rel170}, we have for $T\ge 1$
	\begin{align*} 
	\Ebb &\gnorm*{\relu{\psi(\wb{x}_T)}}{2}{} \le \tfrac{1}{T} \big[ 3(\gnorm{y^*}{2}{}+1)^2\tau + (\eta+L_0+BL_f)W(x^*,x_0) +\tfrac{2\paran{\zeta^2+ \HH_*^2}T}{\eta} + \tfrac{13\sigma_{X,f}^2T}{\tau}\big].
	\end{align*}
	Using bounds $W(x^*, x_0)\le D_X^2$, we obtain \eqref{eq:bounding_infeas_convex}.
	Using \eqref{eq:bounding_infeas_convex} and \eqref{eq:T_convex}, we have
	\begin{equation*} 
	\begin{split}
	\Ebb \gnorm*{\relu{\psi(\wb{x}_T)}}{2}{} &\le \tfrac{\eps}{3} +\tfrac{\eps}{3}+\tfrac{\eps}{3} = \eps,
	\end{split}
	\end{equation*}
	Similarly, using \eqref{eq:bounding_opt_convex} and \eqref{eq:T_convex}, it is easy to observe that $\Ebb \bracket*{\psi_0(\wb{x}_T) - \psi_0(x^*)} \le \eps$.
	Hence we conclude the proof.
%	\qed
\end{proof*}

\section{Proximal Point Methods for Nonconvex Functional Constrained Porblems} \label{sec-alg1}
{Our goal in this section is to extend
the ConEx method to the nonconvex setting by leveraging the framework of proximal point method. To achieve this goal, we first need to understand the convergence properties of the proximal point method for nonconvex functional constrained optimization.}
We first recall the assumptions mentioned briefly in Section \ref{sec:intro}
for the nonconvex case.
\begin{enumerate}
	\item $f_i:X\to \Rbb$ are nonconvex and Lipschitz-smooth functions satisfying the lower curvature condition in \eqref{eq:weak-convexity} with parameters $\mu_i (> 0),\ i = 0, \dots, m$. 
	%\item $\chi_0: X \to \Rbb$ is a proper convex \added{continuous}  function.
	\item $\chi_i: X \to \Rbb, \ i = 0, \dots, m$ are convex and continuous functions.
\end{enumerate}
Let $x^{*}\in X$ be a global optimal solution and $\psi^*_0=\psi_0(x^*)$ be the optimal value of problem \eqref{main-prob}. According to the above assumptions and compactness of $X$, we have $\psi^*_0>-\infty$.

It should be noted, however, that solving nonconvex problem \eqref{main-prob} to the optimality condition in Definition \ref{define-apprx-opt} is generally difficult. 
%Hence, we discuss a new convergence criterion and some new notation specific for this section below.
Due to the hardness of the problem, we focus on the necessary condition for guaranteeing local optimality.
%In order to discuss the termination criterion for solving this problem, 
For this purpose, we need to properly generalize the \emph{subdifferential} for the objective function $\psi_0$ and constraints $\psi_i$
because they are possibly nonconvex and nonsmooth.
Let $\pgrad \chi_0$ and $\pgrad \chi_i, i \in [m]$ be the subdifferentials of $\chi_0$ and $\chi_i$.
We define %\vspace{-1mm}
\begin{align*}
%\pgrad \psi_0(x) &\equiv \{\grad f_0(x)\} + \pgrad \chi_0(x)\\
\pgrad \psi_i(x) &\equiv \{\grad f_i(x)\} + \pgrad \chi_i(x), \tab i = 0, \dots, m.
\end{align*}
Note that $\pgrad \psi_i = \{ \grad f_i\}$ when $\psi$ is a ``purely" differentiable nonconvex function and $\pgrad \psi_i = \pgrad \chi_i$ when $\psi_i$ is a nonsmooth convex function. 
%So $\pgrad f$ and $\pgrad \phi^\br{i}$ are considered as ``subdifferential-like" objects. 

%Hence, we use the notation of $\pgrad f$ henceforth. \\
Using these objects, we can define a Karush-Kuhn-Tucker (KKT) condition for the  nonsmooth nonconvex problem~\eqref{main-prob} 
as follows.
\begin{definition}\label{defi:kkt}
	We say that $x^* \in X $ is a critical KKT point of \eqref{main-prob} if 
	$\psi_i(x^*) \le 0$ and $\exists$ $y^*=\bracket{
		y^{*\br{1}}, \dots, y^{*\br{m}} }^T \ge \zero$ s.t.  
	\begin{equation}\label{eq:criterion-main-prob}
	\begin{split}
	y^\sbr{i}\psi_i(x^*) &= 0, \tab i\in[m],\\[-1mm ]
	d\paran[\big]{\pgrad \psi_0(x^*) + \tsum_{i=1}^{m}y^\sbr{i}\pgrad \psi_i(x^*)+ N_X(x^*), \zero} &= 0.
	\end{split}
	\end{equation}
\end{definition}
The parameters $\{y^{*\br{i}}\}_{i \in [m]}$ are called \emph{Lagrange multipliers}. For brevity, we use the notation $y^*$ and $\bracket{
	y^{*\br{1}}, \dots, y^{*\br{m}} }^T$ interchangeably.

{To ensure that the KKT condition is necessary to achieve optimality, we need a constraint qualification (CQ). A well-known CQ for smooth nonlinear optimization is the classical Mangasarian-Fromovitz constraint qualification (MFCQ, see \cite{manga67}). A slightly extended MFCQ to deal with nonsmooth functions is defined as follows.}

\begin{definition}
[MFCQ]\label{eq:mangafromo} 
	Let $x\in X$ be a point such that $\psi_i(x) \le 0$ for all $i \in [m]$. We say that $x$ satisfies MFCQ if there exists a direction $z \in -N^*_X(x)$ such that %\vspace{-1mm}	
	\begin{equation}\label{eq:mangafromo-2}
		\max_{v \in \pgrad \psi_i(x)} v^Tz < 0, \tab i \in \AA(x) ,
	\end{equation}
	where $\AA(x) $ denotes the indicator set of all active constraints. {Moreover, if $\AA(x) = \emptyset$ then we say that $x$ satisfies MFCQ.}
\end{definition}

{ The following result ensures that the KKT condition in \eqref{eq:criterion-main-prob}
is a first-order necessary optimality condition for the composite nonconvex 
optimization. The proof of this result is given in the appendix.}

\begin{proposition}\label{thm:MFCQ}
%\deleted{
	Let $x^*$ be a local optimal solution of the problem \eqref{main-prob}. If $x^*$ satisfies MFCQ (Definition~\ref{eq:mangafromo}),
	then there exists $y^\sbr{i} \ge 0, \ i\in[m]$ such that 
	\eqnok{eq:criterion-main-prob} holds.
\end{proposition}

{In order to describe the stationary condition at the limit points of the solutions generated by the algorithm, we assume that MFCQ holds on an enlarged domain containing all the limit points of the algorithm.}
\begin{assumption}[Strong MFCQ]\label{ass:strong-mangafromo}
	{All the points in the feasible set of problem~\eqref{main-prob} satisfy MFCQ.}
\end{assumption}

{
\begin{definition}\label{defi:approx-kkt-v1}
	We say that a  point $x \in X$ is an  $\eps$-KKT point for problem~\eqref{main-prob} if it is feasible, i.e., $\psi(x)\le\zero$, and there exists $y\ge \zero$ such that 
	\begin{equation}\label{epsi-kkt}
	\begin{aligned}
	\tsum_{i=1}^m \abs{y^{\br{i}}\psi_i(x)} &\le \epsilon, \\
	\bracket*{d\paran*{\pgrad \psi_0(x) + \tsum_{i=1}^m y^{(i)}\pgrad\psi_i(x)+N_X(x), \zero}}^2  &\le \epsilon	
	\end{aligned}
	\end{equation}
	Moreover, we call {$x$ a stochastic $\eps$-KKT point if  $x$ is a feasible random vector, and $y$ is a random vector such that} \eqref{epsi-kkt} is satisfied under expectation
w.r.t. the random variables involved in the process generating $x$.
\end{definition}
Definition~\ref{defi:approx-kkt-v1} describes a type of points which approximately satisfy the KKT condition to a specific accuracy.
Particularly, a $0$-KKT point satisfies the KKT condition \eqref{eq:criterion-main-prob} exactly and hence is a KKT point. 
However, as will be shown in our convergence analysis, often it is more convenient to describe convergence in terms of the following measure.
\begin{definition} \label{def:approx_kkt}
	We say that a point $\wh{x} \in X$ is an $(\epsilon, \delta)$-KKT point for problem~\eqref{main-prob} if there exists a  $\epsilon$-KKT point $x$ ($\psi(x)\le \zero$) and 
	\begin{equation} \label{eq:delta-proximity-1}
	\gnorm{x-\wh{x}}{}{2} \le \delta.
	\end{equation}
	Moreover, $\wh{x}$ is a stochastic $(\epsilon,\delta)$-KKT point if $x$ is a {stochastic $\epsilon$}-KKT point and 
	\begin{equation} \label{eq:delta-proximity-2}
		\Ebb \big[\gnorm{x-\wh{x}}{}{2}\big]\le \delta.
	\end{equation} 
\end{definition}
}

\subsection{Basic proximal point method} \label{sec_exact_ppt}
The main idea of the proximal point method (see Algorithm~\ref{alg:exact-pp}) is to transform the nonconvex problem into a sequence of convex subproblems by adding strongly convex terms to the objective and to the constraints. %We first consider the algorithm which solves the convex subproblems exactly.
Specifically, each step of the proximal point algorithm involves a convex subproblem~\eqref{inner-convex-prob} with convex constraints. It can be observed that, by adding a strongly convex proximal term, $\psi_0(x; x_{k-1})$ is $\mu_0$-strongly convex and $\psi_{i}(x; x_{k-1})$ is $\mu_i$-strongly convex relative to $W(\cdot,\cdot)$. {Hence, if each subproblem is feasible, it must have a unique globally optimal solution.}

{
One way to ensure the well-definedness of each subproblem is always keeping the solutions $\{x_k\}$  feasible. Nevertheless, feasibility is insufficient to guarantee that the KKT condition holds for the original problem. For illustration, we consider the following example:
\begin{equation}\label{example-no-kkt} % \tag{3.3(k)}
	\begin{split} 
	\min_{x \in [-2,2]} \tab &x - x^2\\[-2mm]
	\text{s.t.} \  \tab  &\frac{1}{2}(x-1)^2 \le 0.
	\end{split}
\end{equation} 
% Let $\psi_0(x)=x$, $\psi_1(x)=\frac{1}{2}(x-1)^2\le 0$ and $X=[-2,2]$. 
 Consider Algorithm~\ref{alg:exact-pp} applied to the above problem with input $x_0=1$, $\mu_0 = 2, \mu_1 \ge 0$ and $W(x,y) = \tfrac{1}{2} (x-y)^2$. Clearly, $x_0$ is the only feasible solution and also the output of 
 the algorithm. However, the main issue in this setting is that the KKT condition fails at $x_0$, since there does not exist a Lagrange multiplier  $y\ge 0$  such that the stationarity condition
\[
  -1 + y (x-1)=0
\]
holds at $x=1$. 
%The main issue is that $x = 1$ is the only feasible solution of \eqref{example-no-kkt} which leads to the nonexistence of the KKT multipliers. 
Due to this issue, it is desired to generate a sequence of strictly feasible solutions $\{x_k\}$, which motivates the following strict feasibility assumption on the initial point.
\begin{assumption}\label{ass:strict-feasibility}
	$x_0$ is a strictly feasible solution to the nonlinear functional constraints of problem~(\ref{main-prob}), namely, $\psi(x_0)<\zero$.
\end{assumption}
}

\begin{algorithm}[h]
	\caption{Exact Constrained Proximal Point Algorithm }
	\label{alg:exact-pp}
	\begin{algorithmic}[1]
		\REQUIRE Input $x_0$
		\FOR{ $k = 1, \dots, K$} 
		\STATE Set \vspace{-2em}
		\begin{align*}
		%\psi_0(x;x_{k-1})&:= {\psi_0(x)} + 2 \mu_0 W(x, x_{k-1}), %\label{eq:fk}\\
		{\psi_i(x;x_{k-1})} &:= {\psi_i(x)} + 2 \mu_i W(x, x_{k-1}), \tab i = 0, \dots, m.%\label{eq:phik}
		\end{align*}\vspace{-1.5em}
		\STATE  Obtain $x_k$ as the optimal solution of the following problem \vspace{-0.5em}
		\begin{equation}\label{inner-convex-prob} % \tag{3.3(k)}
			\begin{split} 
			\min_{x \in X} \tab &\psi_0(x; x_{k-1})\\[-2mm]
			\text{s.t.} \  \tab  &\psi_{i}(x; x_{k-1}) \le 0, \tab i \in [m].
			\end{split}
		\end{equation} \vspace{-1em}
		\STATE \textbf{If} $x_{k-1} = x_k$ then \textbf{return} $x_k$.
		%\IF{$x_k = x_{k-1}$}
		%	\STATE \textbf{return} $x_k$
		%\ENDIF
		\ENDFOR
		\STATE \textbf{return} $x_{K} $
		\end{algorithmic}
\end{algorithm}

Strict feasibility of the generated solutions, as well as some other properties of Algorithm~\ref{alg:exact-pp} are established by the following theorem.
%, namely, the square summability
%of $x_{k-1} - x_k$ and sufficient descent property. 

\begin{theorem}\label{thm:exact-pp}
%{
Suppose that Assumption~\ref{ass:strict-feasibility} holds. Then, 
	\begin{enumerate}
		\item [a)] all the generated points $x_1,x_2,...,x_k...$ are strictly feasible for problem \eqref{main-prob}, and $\braces*{\psi_0(x_k)}$ is a monotonically decreasing sequence;
		\item [b)] either there exists a $\wh{k}$ such that $x_\wh{k}=x_{\wh{k}-1}$ and $x_\wh{k}$ is a KKT point of \eqref{main-prob},   or $\braces*{\psi_0(x_k)}$ is a strictly decreasing sequence that has a limit point $\wt{\psi}_0>-\infty$, and we have 
\[\lim_{k\rightarrow +\infty} \gnorm{x_k - x_{k-1}}{}{} { =} 0.\]
	\end{enumerate}
%	}
\end{theorem}\vspace{-0.5em}

\begin{proofx}
Part a).  
We prove the property of strict feasibility  by induction. First, the strict feasibility of $x_0$ is by Assumption~\ref{ass:strict-feasibility}. Next, assume that our claim holds for $x_{k-1}$, namely $\psi_{i}(x_{k-1})<0$, then $x_{k-1}$ is strictly feasible for the $k$-th subproblem $\eqref{inner-convex-prob}$ with $\psi_0(\cdot; x_{k-1})$ and $\psi(\cdot;x_{k-1})$.  
If $x_k = x_{k-1}$, the claim holds by the induction assumption.
Otherwise, by the feasibility of $x_k$ for \eqref{inner-convex-prob}, we have $\psi_{i}(x_k) < \psi_i(x_k; x_{k-1}) \le 0$ for all $i \in [m]$.\\
To show the monotonicity of $\{\psi_0(x_k)\}$, we establish some sufficient descent property. Due to the optimality of $x_k$ for solving subproblem \eqref{inner-convex-prob} and  the strong convexity of objective {$\psi_0(\cdot; x_{k-1})$}, we have for all feasible $x$ that 
\[
\psi_0(x; x_{k-1}) \ge \psi_0(x_k; x_{k-1}) + \mu_0 W(x, x_{k})\ge \psi_0(x_k)+2\mu_0 W(x_k, x_{k-1})+\mu_0 W(x, x_k).
\] 
In above, taking $x=x_{k-1}$ and using the strong convexity of $\omega(x)$ we have %\vspace{-0.5em}
\begin{equation} \label{pp:square-bound}
\gnorm{x_{k-1}-x_{k}}{}{2} \le \tfrac{2}{3\mu_0} \bracket{\psi_0(x_{k-1})-\psi_0(x_k)},	
\end{equation}
%Summing up \eqref{pp:square-bound} for $k=1,2,3,...K$ yields the result in part a).
which implies that  $\braces{\psi_0(x_k)}$ is a decreasing sequence. 

Part b). Since we already show that $x_{\wh{k}-1}$ is strictly feasible, if $x_\wh{k}=x_{\wh{k}-1}$, then we conclude from Slater's condition  that $x_\wh{k}$ is a KKT point. If  $x_k\neq x_{k-1}$  for all $k$, 
 then \eqref{pp:square-bound} implies that $\braces*{\psi_0(x_k)}$ is strictly decreasing.
 Since this sequence is lower-bounded, we conclude that $\lim_{k\rightarrow +\infty} \psi_0(x_k)=\wt{\psi}_0$ for some $\wt{\psi}_0\ge \psi^*_0$. Taking limit $k\rightarrow \infty$ in \eqref{pp:square-bound} we have $\lim_{k\rightarrow +\infty} \gnorm{x_k - x_{k-1}}{}{}=0$.
%\qed
\end{proofx}

{Strict feasibility established in Theorem \ref{thm:exact-pp} along with Slater's condition guarantees that there exist Lagrange multipliers $\{y_k\}$ such that   $(x_k, y_k)$ satisfies the KKT condition of subproblem~\eqref{inner-convex-prob} for each $k\ge 1$.}
\begin{lemma}
	\label{prop:strong-cvx}
	Let $\paran{x_k,y_k}$ be a KKT point of the subproblem \eqref{inner-convex-prob}.
	Then
	\begin{equation}
	\psi_0(x;x_{k-1})-\psi_0(x_k; x_{k-1})+\left\langle y_k,\psi(x; x_{k-1})\right\rangle 
	\ge  \big(\mu_0+\mu^Ty_k\big) W(x, x_k),\quad x\in X.\label{eq:strong-cvx}
	\end{equation}
\end{lemma}\vspace{-0.5em}
\begin{proofx}
	{Due to the KKT condition, there exist $\psi'_0(x_k, x_{k-1})\in\pgrad\psi_0(x_k, x_{k-1})$, $\psi'_i(x_k,x_{k-1}) \in \pgrad \psi_i(x_k, x_{k-1})$ and $z^{*}\in N_X(x_k)$  satisfying the condition 
	\begin{equation}\label{eq:strong-cvx-mid-01}
		\psi'_0(x_k, x_{k-1}) + \tsum_{i=1}^m y_k^{(i)}\psi'_i(x_k,x_{k-1}) + z^*=0.
	\end{equation}
	}
%	\eqref{eq:criterion-main-prob}.
%	{Due to the KKT condition, $x_k$ is the minimizer of the Lagrangian function $\psi_0(\cdot; x_{k-1})+\langle y_k, \psi(x,x_{k-1})\rangle$}
	According to the strong convexity of $\psi_0(\cdot; x_{k-1})$, $\psi_i(\cdot; x_{k-1})$, and the fact that $y_k\ge \zero$, we have
	\begin{align*}
	\psi_0(x; x_{k-1})+\inprod{y_k}{\psi(x; x_{k-1})} 
	&\ge \psi_0(x_k; x_{k-1})+\inprod{ \psi'_0(x_k; x_{k-1})}{x-x_k} + \mu_0 W(x, x_k) \\
	& \quad + \inprod{y_k}{\psi(x_k; x_{k-1})} + \inprod{\tsum_{i=1}^my_k^{(i)}\psi'_i(x_k; x_{k-1})}{x-x_k}+ \paran{\mu^Ty_k} W(x, x_k) \\ 
	& =\psi_0(x_k; x_{k-1})+\inprod{\psi'_0(x_k; x_{k-1})+\tsum_{i=1}^my_k^{(i)}\psi'_i(x_k;x_{k-1})}{x-x_k}\\
	 &\quad+ (\mu_0+ \mu^Ty_k) W(x, x_k),
	\label{strong-cvx-kkt-mid}
	\end{align*}
	where the last equality follows from the complementary slackness part of the KKT condition. {Moreover, appealing to the definition of normal cone and property~(\ref{eq:strong-cvx-mid-01}),  we have that
	\[
	\inprod{\psi'_0(x_k;x_{k-1})+\tsum_{i}y_k^{(i)}\psi'_i(x_k; x_{k-1})}{x-x_k}=-\langle z^*, x-x_k\rangle \ge 0,
	\]
	Putting the above two inequalities together, we arrive at relation \eqref{eq:strong-cvx}.}
\end{proofx}

In view of Lemma~\ref{prop:strong-cvx} and Theorem~\ref{thm:exact-pp}, {we  develop a boundedness condition on the sequence $\{y_k\}$.}

\begin{proposition}\label{prop:bound_y_exact_pp}
%	Assume that  $x_0$ is strictly feasible for \eqref{main-prob} in Algorithm \ref{alg:exact-pp}. 
	{Suppose that Assumption~\ref{ass:strict-feasibility} holds.}
	Then, for all $k \ge 1$, there exists $y_k = \bracket {y_k^{\br{1}},  \dots, y_k^{\br{m}}}^T$ such that $y_k \ge \zero$, and 
\begin{equation} \label{eq:criterion-inner-prob}
\begin{aligned}
y_k^\br{i}\psi_i(x_k; x_{k-1}) &= 0, \tab  i =1, \dots, m,\\%[-1mm]
\pgrad \psi_0(x_k; x_{k-1}) + \tsum_{i\in[m]}y_k^\br{i}\pgrad \psi_i(x_k; x_{k-1})+N_X(x_k) &\ni \zero.
\end{aligned}
\end{equation}
and we have the following boundedness condition:
 \begin{equation}
\|y_k\|_1 \le \tfrac{\psi_0(x_{k-1})-\psi_0(x_k)}{\min_{1\le i\le m} \braces{-\psi_{i}(x_{k-1})}}, \quad k=1,2,3,\ldots \label{eq:bound-yk} 	
 \end{equation}
\end{proposition}

\begin{proofx}
Strict feasibility of $x_0$ along with Theorem~\ref{thm:exact-pp}.a) imply that the subproblem \eqref{inner-convex-prob} in Algorithm \ref{alg:exact-pp} satisfies {Slater's condition for all $k \ge 1$, which ensures that the KKT condition \eqref{eq:criterion-inner-prob} holds at optimality.} In particular, the first relation in \eqref{eq:criterion-inner-prob} is a direct application of the KKT complementary slackness and the second relation is an application of the KKT stationarity. 
Similarly, applying Lemma~\ref{prop:strong-cvx} and setting $x = x_{k-1}$ in \eqref{eq:strong-cvx} yield
\begin{align*}
	\psi_0(x_{k-1}) - \psi_0(x_k) & 
	\ge\paran{\mu_0+\mu^T y_k} W(x_{k-1}, x_k) + 2\mu_0W(x_k, x_{k-1})  -\tsum_{i=1}^m y_k^\br{i}\psi_{i}(x_{k-1})	 \\
	& \ge \gnorm{y_k^\br{i}}{1}{} \min_{1\le i \le m} \braces{-\psi_{i}(x_{k-1})}.
\end{align*}
 Thus relation \eqref{eq:bound-yk} immediately follows.
\end{proofx}

{
We have two comments about the above results. 
First, in order to show that the KKT solutions $\{(x_k,y_k)\}_{k \ge 1}$ are well-defined, we need to ensure that Algorithm~\ref{alg:exact-pp} generates a path of strictly feasible solutions $\{x_k\}_{k \ge 0}$. 
However, to achieve this goal, Algorithm~\ref{alg:exact-pp} requires an oracle that solves the convex subproblem~(\ref{inner-convex-prob}) exactly. Since computation of the optimal solution  can be impractical for large-scale or stochastic optimization, it is desirable to develop inexact variants of the proximal point method which can deal with approximate solutions for the subproblem \eqref{inner-convex-prob}.
Second, we note that the bound on the optimal dual solution $\|y_k\|_1$, as provided in Proposition~\ref{prop:bound_y_exact_pp}, depends on the algorithm. 
%Conceptually, we hope that the boundedness of $\{y_k\}$ and proximity of $\{x_k\}$ lead to convergence to the KKT condition of problem \eqref{main-prob}.
For some special cases, the bound is uniform for the whole sequence. 
For example, if the algorithm generates $x_k=x_{k-1}$ for some $k > 1$, then the stationary point is interior to the inequality constraints and we have $y_k=\zero$. 
However, \eqref{eq:bound-yk} does not rule out the possibility that  $\gnorm{y_k}{1}{}$ tends to infinity when $x_k$ converges to boundary points. 
Therefore, it is difficult to claim that
the KKT conditions of problem \eqref{main-prob} will be satisfied at the limit points of $\{x_k\}$.
 Due to the above two concerns, we turn our attention to inexact proximal point methods which can deal with approximate solutions for the subproblem \eqref{inner-convex-prob}.
 In particular, we will discuss additional constraint qualifications to ensure that the optimal dual solutions have a uniform bound,
 and thus to %edness of the optimal dual solutions. 
allow us to establish the convergence of the inexact proximal point method to the KKT solutions. 
}

\subsection{Inexact proximal point method under MFCQ} \label{sec:inexact_ppt}
In Algorithm~\ref{alg:inexact-pp}, we present an inexact extension of the proximal point method, which  generalizes Algorithm~\ref{alg:exact-pp} by allowing the subproblems to be solved approximately.
To distinguish  exact solutions from  approximate ones, we denote the exact solution as $x_k^*$ and the corresponding dual solution as $y_k^*$ hereafter.  Since each subproblem \eqref{inner-convex-prob} is solved inexactly, the sequence generated by Algorithm~\ref{alg:inexact-pp} can become infeasible with respect to the original problem. If $x_{k-1}$ is infeasible with respect to \eqref{main-prob}, then
we can not guarantee feasibility of the subproblem \eqref{inner-convex-prob} in general.
This also implies obtaining bounds on Lagrange multipliers is more challenging for inexact case. 
However, we show that if the successive problems are solved accurately enough then 
both the strict feasibility of the iterates and the boundedness of $\{y_k^*\}$ can be ensured.

\begin{algorithm}[h]
	\caption{Inexact Constrained Proximal Point Algorithm}
	\label{alg:inexact-pp}
	\begin{algorithmic}[1]
		\STATE Input $x_0$ is strictly feasible.
		\FOR { $k = 1, \dots, K$}
		\STATE $x_k \gets $ an approximate solution of subproblem \eqref{inner-convex-prob}. % $(\delta_k, \wb{\delta}_k)$-optimal solution (c.f. Definition \ref{define-apprx-opt}) of \ref{inner-convex-prob}.
		\ENDFOR
		\STATE $\wh{k} = \argmin_{1\le k\le K} \psi_0(x_{k-1}) - \psi_0(x_k)$. %Randomly choose $\wh{k}$ from $\braces{1,2,...,K}$.
		\RETURN $x_{\wh{k}}$.
	\end{algorithmic}
\end{algorithm}

Throughout the rest of this subsection, we assume that 
$\psi_i(x; x_{k-1})$ is Lipschitz continuous with constant $M_i$, $i = 0, \dots, m$. Proposition~\ref{prop:ipp-error-assum} shows 
that the sequence $\{x_k\}$ is strictly feasible if the subproblem \eqref{inner-convex-prob} is solved accurately enough.

\begin{proposition}\label{prop:ipp-error-assum} In Algorithm~\ref{alg:inexact-pp}, suppose that Assumption~\ref{ass:strict-feasibility} holds and that
the subproblem \eqref{inner-convex-prob}, if solvable, is returned an approximate solution  $x_k$ satisfying
		\begin{equation}
		\sqrt{\tfrac{M_i}{\mu_i}\|x_{k}-x_{k}^{*}\|}+\|x_{k}-x_{k}^{*}\| \le \frac{1}{2} \|x_{k-1}-x_{k}^{*}\|,\ i\in[m], \label{ipp:dist-bound-const}
		\end{equation}
		 where $x_k^*$ (which depends on $x_{k-1}$) is the optimal solution, {then the sequence $\{x_k\}$ generated by Algorithm~\ref{alg:inexact-pp} is strictly feasible}, i.e. $\psi(x_k)<\zero$, $k=0,1,2,3,\ldots$. 		
		Moreover, there exists subproblem solutions $(x_k^*,y_k^*)$ satisfying the KKT condition for \eqref{inner-convex-prob}.
		
		If we further assume that $\{x_k\}$ satisfies:
		\begin{equation}
		\sqrt{\tfrac{2M_0}{\mu_0}\|x_{k}-x_{k}^{*}\|}+\|x_{k}-x_{k}^{*}\|\le \|x_{k-1}-x_{k}^{*}\|,\label{ipp:dist-bound-obj}
		\end{equation}
		then $\{\psi_0(x_k)\}$ is monotonically decreasing  and converges to a limit point $\wt{\psi}_0$, and
		\begin{equation}\label{eq:ipp-dist-conv-zero}
		\lim_{k\rightarrow \infty} W(x_k, x_{k-1})= \lim_{k\rightarrow \infty}W(x_{k-1}, x_k^*)  = 0.
		\end{equation}
\end{proposition}

\begin{proofx} 
Note that the strict feasibility of $x_{k-1}$ trivially implies that subproblem~(\ref{inner-convex-prob}) is solvable. We want to show that the strict feasibility of $x_{k-1}$, along with condition~(\ref{ipp:dist-bound-const}),   implies the strict feasibility of $x_k$. Therefore, using $\psi(x_0)<\zero$, it is easy to prove the strict feasibility of the whole sequence by induction.

 Suppose that $x_{k-1}$ is strictly feasible. 
From the definition of $\psi_i(x; x_{k-1})$ and feasibility of $x_k^*$, we have 
	\[ 
\psi_i(x_{k})+2\mu_i W(x_{k}, x_{k-1})=\psi_i(x_{k}; x_{k-1})\le\psi_i(x_{k}^{*}; x_{k-1})+M_i \|x_k-x_k^*\| \le M_i \|x_k-x_k^*\|,
\]
where the first inequality follows from the Lipschitz continuity of $\psi_i(x;x_{k-1})$.
Using the triangle inequality and (\ref{ipp:dist-bound-const}),  we have
\begin{align*}
\sqrt{2\mu_i W(x_{k}, x_{k-1})}\ge \sqrt{\mu_i} \|x_{k}-x_{k-1}\|
& \ge \sqrt{\mu_i}  (\|x_{k-1}-x_{k}^*\|-\|x_{k}-x_{k}^*\|) \\
& \ge  2\sqrt{M_i \|x_k-x_k^*\|},
\end{align*}
for $i\in[m]$. Combining the above two results, we have $\psi_i(x_{k})+\frac{3}{2}\mu_i W(x_k,x_{k-1})\le0$. If $x_k \neq x_{k-1}$, then we have strict feasibility: $\psi_i(x_{k})<0$. Otherwise,  $x_k=x_{k-1}$ is strictly feasible. {Using the induction we complete the proof of strict feasibility of the whole sequence $\{x_k\}$.} 
 Moreover, the existence of $(x_k^*, y_k^*)$ immediately follows from Slater's condition and the KKT condition.

Applying Lemma~\ref{prop:strong-cvx} with $x= x_{k-1}$ and replacing the saddle point $(x_k, y_k)$ therein by $(x_k^*, y_k^*)$, we  deduce
\begin{align}
\psi_0(x_{k-1})& = \psi_0(x_{k-1}; x_{k-1})\nonumber \\
	& \ge \psi_0(x_k^*; x_{k-1})-\langle y_k^*, \psi(x_{k-1}; x_{k-1})\rangle +(\mu_0+\mu^Ty_k^*)W(x_{k-1}, x_k^*)\nonumber \\
 & \ge \psi_0(x_k; x_{k-1})-M_0 \|x_k-x_k^*\| +(\mu_0+\mu^Ty_k^*)W(x_{k-1}, x_k^*).\nonumber \\
 & = \psi_0(x_k)+2\mu_0 W(x_k, x_{k-1})-M_0 \|x_k-x_k^*\| +(\mu_0+\mu^Ty_k^*)W(x_{k-1}, x_k^*). \label{eq:ipp-mid-01}
\end{align}
where the second inequality follows from the Lipschitz continuity of $\psi_0(\cdot; x_{k-1})$ and the basic property $\langle y_k^*, \psi(x_{k-1}; x_{k-1})\rangle = \langle y_k^*, \psi(x_{k-1})\rangle \le 0$. 
Moreover, using \eqref{ipp:dist-bound-obj} and similar argument beforehand, we have 
\begin{equation*}
\sqrt{2\mu_0 W(x_{k}, x_{k-1})}\ge \sqrt{2M_0 \|x_k-x_k^*\|}.
\end{equation*}
Putting this result in (\ref{eq:ipp-mid-01}), we have
	\begin{equation}\label{ipp:fk-diff}
			\psi_0(x_k)+\mu_0 W(x_k, x_{k-1}) +(\mu_0+\mu^Ty_k^*) W(x_{k-1}, x_k^*) \le \psi_0(x_{k-1}).	
	\end{equation}
 We immediately observe that $\psi_0(x_k)$ is decreasing. Since $\psi_0$ is bounded below, we have the convergence $\lim_k \psi_0(x_k)=\wt{\psi}_0$ for some $\wt{\psi}_0>-\infty$.
		Summing up the above relation for $k=1,2,...,$ we have
	\begin{equation}\label{ipp:sum-square}
		\sum_{k=1}^\infty[ \mu_0W(x_k,x_{k-1})+ (\mu_0+\mu^Ty_k^*)W(x_{k-1}, x_k^*) ] \le \psi_0(x_0)-\wt{\psi}_0<+\infty.
	\end{equation}
Therefore, the convergence results in (\ref{eq:ipp-dist-conv-zero}) immediately follow.
%\qed
\end{proofx}

\begin{remark}
	It should be noted that the inexactness criteria {include the subproblem} optimality criteria for the exact proximal point method (Algorithm~\ref{alg:exact-pp}) as a special case. Specifically, if we set $x_k=x_k^*$, then (\ref{ipp:dist-bound-const}) and (\ref{ipp:dist-bound-obj}) hold trivially. Hence all our convergence analysis applies to the exact proximal point discussed in subsection~\ref{sec_exact_ppt}.
\end{remark}

The following theorem establishes the asymptotic convergence of the proposed inexact proximal point method under some mild constraint qualification.
\begin{theorem}\label{thm:ipp-subseq-bound}
Suppose that all the assumptions in Proposition \ref{prop:ipp-error-assum} hold, $x^*$ is a limit point of the solution sequence and it satisfies MFCQ.  %(\ref{eq:mangafromo-2}) holds at $x^*$. 
If $\{x_{j_k}\}$ is a subsequence converging to $x^*$, then the dual sequence $\{y_{j_k}^*\}$ is bounded. Moreover, if $y^*$ is a limit point of $\{y_{j_k}^*\}$ ,then $(x^*, y^*)$ satisfies the KKT condition \eqref{eq:criterion-main-prob}.
\end{theorem}

\begin{proofx}
First, we establish the convergence of $\{x_{j_k}^*\}$ to $x^*$. It immediately follows from the assumption $\lim_{k\rightarrow\infty} x_{j_k}=x^*$ and  Proposition \ref{prop:ipp-error-assum} that $\lim_{k\rightarrow\infty} x_{j_k-1}=x^*$. Applying Proposition \ref{prop:ipp-error-assum}  and the triangle inequality, we have
\begin{align*}
\lim_{k\rightarrow\infty} \|x_{j_k}^*-x^*\| & \le \lim_{k\rightarrow\infty} \big[\|x_{j_k}^*-x_{j_k-1}\| + \|x_{j_k-1}-x^*\|\big]=0,
\end{align*}
which implies that 
\begin{equation}\label{eq:ipp-mid-02}
	\lim_{k\rightarrow\infty} x_{j_k}^*
%	=\lim_{k\rightarrow\infty} x_{j_k-1}^*
	 = x^*.
\end{equation}	

We prove the boundedness of the dual subsequence by contradiction.
	Suppose that $\{y_{j_k}^*\}$ is unbounded. Passing to any subsequence if necessary, we have $\lim_{k\rightarrow \infty}\gnorm{y_{j_k}^*}{1}{}=\infty$. In view of the KKT condition,  we have
	\begin{equation}
		\psi_0(x_{j_k}^*)+(y_{j_k}^*)^T\psi(x_{j_k}^*) \le \psi_0(x)+(y_{j_k}^*)^T\psi(x) + 2(\mu_0+\mu^Ty_{j_k}^*) [W(x,x_{j_k-1}) - W(x_{j_k}^*, x_{j_k-1})],\label{subprob-bound}
	\end{equation}
for any $ x \in X$. Let $v_{j_k}=y_{j_k}^*/\gnorm{y_{j_k}^*}{1}{}$, then $\gnorm{v_{j_k}}{1}{}=1$, hence $\{v_{j_k}\}$ must have a convergent subsequence. Without loss of generality, we assume $\lim_{k\rightarrow \infty} v_{j_k}=v^*$.
Let us divide both sides of \eqref{subprob-bound} by $\gnorm{y_{j_k}^*}{1}{}$ and take $k\rightarrow \infty$. 
In view of (\ref{eq:ipp-mid-02}) and $\lim_{k\rightarrow\infty}1/\|y_{j_k}^*\|_1=0$, we have
\begin{equation}
	(v^*)^T\psi(x^*)=\lim_{k\rightarrow \infty}({v^*})^T\psi(x_{j_k}^*) \le ({v^*})^T\psi(x) + 2(\mu^Tv^*) W(x, x^*), \quad \forall x \in X, \label{subprob-bound2}
\end{equation}
which means that $x^*$ is optimal for minimizing the right side of (\ref{subprob-bound2}).
Therefore, the first-order  condition implies that
\begin{equation}
d\paran*{\tsum_{i=1}^m \pgrad\psi_{i}(x^*)v^\sbr{i} +N_X(x^*), \zero}=0. \label{dist-xstar}
\end{equation}
Let $\AA(x^*)$ be the set of active constraints at $x^*$. 
By this definition, for any $i \notin \AA(x^*)$, we have $\psi_{i}(x^*)<0$. 
Since $\psi_{i}$ is continuous and  $\|x_{j_k}^*-x_{j_{k}-1}\|^2$ converges to $0$, there exists $k_0$ such that for all $k>k_0$, 
we have $\psi_i(x_{j_k}^*; x_{j_k-1}) < 0$. Hence, according to the KKT complementary slackness condition for the subproblem, $y_{j_k}^\sbr{i}=0$ for $k > k_0$. Taking $k\rightarrow \infty$ 
we obtain $v^{*\br{i}}=0$ for any $i \notin \AA(x^*)$. 
So we can rewrite the equation \eqref{dist-xstar} as 
\[ d\paran*{\tsum_{i \in \AA(x^*)} \pgrad\psi_{i}(x^*)v^\sbr{i} +N_X(x^*), \zero}=0. \]
Let $\psi'_i(x^*) \in \pgrad \psi_{i}(x^*)$ for  $i \in [m]$ and $u \in N_X(x^*)$ be such that $u + \tsum_{i\in\AA(x^*)} {\psi'_i}(x^*)v^\sbr{i} = \zero$. {Let $z$ be the direction vector defined in MFCQ~\eqref{eq:mangafromo-2}.} Then,
\begin{align*}
0 = z^Tu + \tsum_{i \in \AA(x^*)}v^\sbr{i} z^T\psi'_i(x^*) &\le\tsum_{i \in \AA(x^*)}v^\sbr{i} z^T\psi'_i(x^*) \\
& \le \tsum_{i \in \AA(x^*)}v^\sbr{i} \max_{v \in \pgrad \psi_{i}(x^*)}z^Tv < 0,
\end{align*} 
where the first inequality follows since $z \in -N_X^*(x^*)$ and $u \in N_X(x^*)$ implies $z^Tu \le 0$, 
the second inequality follows due to the fact that $v^\sbr{i} \ge 0$ and $\psi'_i(x^*) \in \pgrad \psi_{i}(x^*)$ and
the last strict inequality follows from MFCQ and $v^\sbr{i} > 0$ for at least one $i \in \AA(x^*)$. 
Hence we obtain a contradiction and conclude that  $\{y_{j_k}^*\}$ is a bounded sequence.
Since $\{y_{j_k}^*\}$ is bounded, it must have a limit point $y^*$.   Passing to any subsequence if necessary, we have $\lim_{k\rightarrow \infty} y_{j_k}^*=y^*$. 

To prove that $(x^*, y^*)$ satisfies the KKT condition, we first show that $(x^*, y^*)$ satisfies complementary slackness. 
Applying Lemma \ref{prop:strong-cvx} with $x=x_{k-1}$ and replacing $(x_k, y_k)$ by $(x_k^*, y_k^*)$, we have
\begin{equation}
\label{eq:int_rel128}
\psi_0(x_{k-1})-\psi_0(x_k^*) \ge 2\mu_0 W(x_k^*, x_{k-1})+(\mu_0+\mu^Ty_k^*)\,W(x_{k-1}, x_k^*).
\end{equation}
Moreover, the KKT condition for the subproblem (\ref{inner-convex-prob}) and the assumption (\ref{wxy-bound}) imply that
\begin{equation}
	-\tsum_{i=1}^m y_k^\sbr{i}\psi_i(x_k^*) = 2 (\mu^T y_k^*) W(x_k^*, x_{k-1}) \le 2L_\omega (\mu^T y_k^*)\,W(x_{k-1}, x_k^*). \label{eq:ipp-mid-03}
\end{equation}
Combining (\ref{eq:int_rel128}) and (\ref{eq:ipp-mid-03}) and taking the limit $k\rightarrow \infty$, we have
\begin{align*}
	0 & \le -\lim_{k\rightarrow \infty} \tsum_{i=1}^m y_k^\sbr{i}\psi_{i}(x_k^*) \\
	& \le \lim_{k\rightarrow \infty} 2L_\omega \big[\psi_0(x_{k-1})-\psi_0(x_k^*)\big] \\
	& \le \lim_{k\rightarrow \infty} 2L_\omega \big[\psi_0(x_{k-1},x_{k-1})-\psi_0(x_k^*,x_{k-1})+2\mu_0 W(x_k^*, x_{k-1})\big] \\
	& \le \lim_{k\rightarrow \infty} 2L_\omega \big[ M_0 \|x_{k-1}-x_k^*\|+2\mu_0 W(x_k^*, x_{k-1})\big]\\
	& = 0,
\end{align*}
where the last equality is implied by Proposition~\ref{prop:ipp-error-assum} and (\ref{wxy-bound}). 
Hence we have 
\begin{equation}\label{eq:ipp-mid-04}
\lim_{k\rightarrow \infty} y_k^\sbr{i}\psi_{i}(x_k^*)=0, \tab i=1,2,...,m.
\end{equation}
Passing to the subsequence $(x_{j_k}, y_{j_k}^*)$ that converges to $(x^*, y^*)$,
we conclude from (\ref{eq:ipp-mid-04}) that
 \begin{equation}
 y^\sbr{i}\psi_{i}(x^*)=0, \tab i=1,2,...,m. \label{eq:ipp-mid-06}
 \end{equation}

Next, we check the stationarity condition. Since $(x_{j_k}, y_{j_k}^*)$ converges to $(x^*, y^*)$, taking  $k\rightarrow\infty$ in (\ref{subprob-bound}) yields
	\begin{equation}
		\psi_0(x^*)+(y^*)^T\psi(x^*) \le \psi_0(x)+(y^*)^T\psi(x) + 2(\mu_0+\mu^Ty^*) W(x,x^*), \label{subprob-bound-2}
	\end{equation}
implying that $x^*$ minimizes the function $\psi_0(x)+(y^*)^T\psi(x)+2(\mu_0+\mu^Ty^*) W(x,x^*)$ over $x\in X$. In other words, we have
\begin{equation}\label{eq:ipp-mid-05}
0\in N_X(x^*) + \partial \psi_0(x^*) + \tsum_{i=1}^m{y}^\sbr{i}\partial\psi_{i}(x^*).
\end{equation} 
In view of (\ref{eq:ipp-mid-06}) and (\ref{eq:ipp-mid-05}),  we complete the proof.
%\qed
\end{proofx}

The following theorem shows the asymptotic convergence and 
the rate of convergence to stationarity of all the limit points of $\{x_k\}$. An ingredient vital to the latter result is the uniform boundedness of the dual variables under Strong MFCQ.
\begin{theorem}\label{thm:uniform-bound-and-rate}
Suppose that  all the assumptions in Proposition~\ref{prop:ipp-error-assum} and Assumption~\ref{ass:strong-mangafromo} hold. Then,
\begin{enumerate}
\item [a)] all the limit points of the sequence $\{x_k\}$ are critical KKT points;
\item [b)] the whole sequence $\{y_k^*\}$ is uniformly bounded, namely, there exists some constant $B>0$  that $\|y_k^*\|_1 \le B$, for $k=1,2,3,..$;
\item [c)]  after $K$ iterations, the solution sequence contains an  $(\varepsilon_K, \wb{\varepsilon}_K)$-KKT point  with 
\begin{align}
\varepsilon_K &= \max\big\{2L_\omega, 8L_\omega^2(\mu_0 + \|\mu\|_\infty B)\big\}\, \frac{\psi_0(x_0)-\psi_0^*}{K}, \label{eq:ipp-epsi-k}\\
\wb{\varepsilon}_K &= \min\Big\{\frac{\psi_0(x_0)-\psi_0^*}{{\wt{M}}^2 K^2},\frac{1}{2\mu_0 K} \Big\}\, \big[\psi_0(x_0)-\psi_0^*\big].\label{eq:ipp-epsi-bar-k}
\end{align}
where $\wt{M}=\max_{0\le i\le m} \frac{\mu_0}{\mu_i}M_i$. Moreover, if  Algorithm~\ref{alg:inexact-pp} generates exact solutions $x_k=x_k^*$, i.e., it is reduced to Algorithm~\ref{alg:exact-pp}, then the solution sequence contains an $\varepsilon_K$-KKT point.
\end{enumerate}
\end{theorem}

\begin{proofx} Part a): A natural consequence of Assumption \ref{ass:strong-mangafromo} is that every limit point of $\{x_k\}$ satisfies MFCQ (Definition \ref{eq:mangafromo}). Then, applying Theorem~\ref{thm:ipp-subseq-bound}, we immediately get part a).

Part b):	We show the boundedness of $\{y_k^*\}$ by contradiction. Suppose that the sequence is unbounded, then there exists a subsequence $\{j_k\}$ such that 
\begin{equation}
\lim_{k\rightarrow \infty}\|y_{j_k}^*\|=\infty, \label{eq:yjk-diverge}
\end{equation}	
Since $\mathcal{X}$ is compact and $\{x_{j_k}\}$ is bounded, there exists a limit point $x^*$ and a subsequence $\{x_{i_k}\}\subseteq \{x_{j_k}\}$ that $ \lim_{k\rightarrow\infty}x_{i_k}=x^*$. Due to Assumption \ref{ass:strong-mangafromo}, we obtain that $x^*$ satisfies MFCQ.
However, according to Theorem~\ref{thm:ipp-subseq-bound}, when $x^*$ satisfies MFCQ,  $\|y_{i_k}^*\|_1$ must be bounded, thereby leading to a contradiction to (\ref{eq:yjk-diverge}). This completes the proof of the existence of a constant $B>0$ such that 
\[\|y_k^*\|_1\le B \tab  k=1,2,...\]

Part c): Next we establish the rate of convergence to KKT condition.
Due to the optimality of $x^\ast_k$ in subproblem~\eqref{inner-convex-prob}, we have 
\[d \paran[\big]{\pgrad \psi_0(x^\ast_k; x_{k-1}) + \tsum_{i=1}^{m}y_{k}^\sbr{i}\pgrad \psi_i(x_k^\ast; x_{k-1}) +N_X(x_{k}^*), \zero} \ni 0.\]
Plugging the definition of $\pgrad \psi_0(; x_{k-1})$ and  $\pgrad \psi_{i}(; x_{k-1}), i\in [m],$ into the above inequality yields
\begin{equation}\label{ipp-kkt}
	d\paran[\big]{\pgrad \psi_0(x^\ast_k)  + \tsum_{i=1}^{m}y_k^{\sbr{i}}\pgrad \psi_{i}(x_k^\ast) +2\paran{\mu_0 +\mu^Ty_k^*} (\grad\omega(x_k^*)-\grad\omega(x_{k-1})) + N_X(x_k^\ast), \zero}  = 0.
\end{equation}
It follows that
\begin{align}
& d\paran{\pgrad \psi_0(x^\ast_k)  + \tsum_{i=1}^{m}y_k^{\sbr{i}}\pgrad \psi_{i}(x_k^\ast) + N_X(x_k^\ast), \zero}^2 \nonumber \\
\le{}&  4\paran{\mu_0 + \mu^T y_k^\ast}^2 \gnorm{\grad\omega(x_k^*)-\grad\omega(x_{k-1})}{}{2} \nonumber \\
\le{}&  8 L_\omega^2 \paran{\mu_0 + \mu^T y_k^\ast}^2 W(x_{k-1}, x^*_k)\nonumber \\
\le{}&  8L_\omega^2(\mu_0 + \|\mu\|_\infty B) \bracket*{\psi_0(x_{k-1})-\psi_0(x_k)}, \label{eq:ipp-stationary-bound-1}
\end{align}
where the second inequality uses the smoothness and strong convexity of $\omega(x)$, and the last inequality follows from \eqref{ipp:fk-diff} and $\mu_0+\mu^Ty_k^* \le \mu_0+\|\mu\|_\infty B$. 
In addition, by complementary slackness and \eqref{ipp:fk-diff} we have
\begin{align}
\tsum_{i=1}^m \abs{y_k^\sbr{i}\psi_{i}(x_k^*)}&=2(\mu^Ty_k^*) W(x_k^*, x_{k-1}) 
\le 2L_\omega (\mu^Ty_k^*) W(x_{k-1}, x_k^*) \nonumber\\
& \le 2 L_\omega \bracket*{\psi_0(x_{k-1})-\psi_0(x_k)}.\label{eq:ipp-comp-slack-bound-1}	
\end{align}
Furthermore, by the distance contraction property \eqref{ipp:dist-bound-const}, \eqref{ipp:dist-bound-obj} and relation \eqref{ipp:fk-diff}, we have  
\begin{equation}
	\|x_k-x_k^*\|^2\le  \frac{1}{4}\|x_{k-1}-x_k^*\|^2 \le \tfrac{1}{2\mu_0} [\psi_0(x_{k-1})-\psi_0(x_k)], \label{eq:ipp-dist-square-bound-1}
\end{equation}
and 
\begin{align}
\|x_k-x_k^*\| & \le \min \Big\{\frac{\mu_0}{2M_0}, \min_{1\le i\le m}\frac{\mu_i}{4M_i} \Big\} \|x_{k-1}-x_k^*\|^2\nonumber \\
 & \le \min \Big\{\frac{1}{M_0}, \min_{1\le i\le m}\frac{\mu_i}{2\mu_0 M_i} \Big\} [\psi_0(x_{k-1})-\psi_0(x_k)]\nonumber \\
  & \le \frac{1}{\wt{M}} [\psi_0(x_{k-1})-\psi_0(x_k)]. \label{eq:ipp-dist-square-bound-2}
\end{align}
Noting that $\wh{k}=\argmin_{1\le k\le K} {\psi_0(x_{k-1}) - \psi_0(x_k)}$,  we have 
\begin{equation}
\min_{1\le k \le K} [\psi_0(x_{k-1})-\psi_0(x_k)]\le \frac{1}{K} \sum_{k=1}^K[\psi_0(x_{k-1})-\psi_0(x_k)]\le \frac{\psi_0(x_0)-\psi_0^*}{K}. \label{eq:bound-func-diff}	
\end{equation}
Let $x^{\hat{k}}$ be the desired approximate KKT point. Combining (\ref{eq:ipp-stationary-bound-1}), (\ref{eq:ipp-comp-slack-bound-1}) and (\ref{eq:bound-func-diff}) we obtain (\ref{eq:ipp-epsi-k}), 
and combining (\ref{eq:ipp-dist-square-bound-1}), (\ref{eq:ipp-dist-square-bound-2}) and (\ref{eq:bound-func-diff}) gives (\ref{eq:ipp-epsi-bar-k}). 
Finally, noting that an $(\epsilon,\delta)$-KKT point is an $\epsilon$-KKT point when $\delta=0$, we immediately see that the exact proximal point method generates an $\varepsilon_K$-KKT solution.
%\qed

\end{proofx}
\begin{remark}
We leave several comments about the above convergence results.
First, imposing MFCQ type assumption is quite common in the traditional nonlinear optimization for justifying the search of KKT solutions (see \cite{bertsekasnonlinear}). By means of MFCQ, we  not only ensure asymptotic convergence to the KKT solution but also obtain a non-asymptotic convergence rate result. To the best of our knowledge, this appears to be the first complexity result and efficiency analysis of proximal point method under the MFCQ type assumption
 Second,  while Assumption~\ref{ass:strong-mangafromo} requires MFCQ on the whole feasible domain, we only need every limit point of $\{x_k\}$ to satisfy MFCQ throughout the proof. Strong MFCQ is an algorithm independent sufficient condition to achieve MFCQ on every limit point. 
 Third, it should be noted that while Algorithm~\ref{alg:inexact-pp} generates approximate KKT points that are strictly feasible, feasibility is actually not essential for the underlying optimality measure in Definition~\ref{defi:kkt}. The next subsection will describe inexact proximal point method with possibly infeasible solution sequence. (Also see Remark~\ref{rem:ipp-4}).
\end{remark}

\begin{remark}\label{rem:cvx_case_exact}
	All the results in this section can be easily extended to the case when $\psi_{i}, i\in[m]$ are convex functions. In that case, we can replace $\mu_i=0$ for all $i \in [m]$. 
	As a result, the subproblem \eqref{inner-convex-prob} of Algorithm~\ref{alg:exact-pp} becomes
	\begin{equation}\label{inner-conv-prob-conv-algoirhtm}
	\begin{split}
	\min_{x \in X} \tab &\psi_0(x; x_{k-1})\\[-2mm]
	\text{s.t.} \tab &\psi_{i}(x) \le 0, \tab i \in[m].
	\end{split}
	\end{equation} 
	Hence, constraints are fixed for all iterations. This implies that we do not need \eqref{ipp:dist-bound-const} for ensuring the strict feasibility of iterates  for every subproblem \eqref{inner-conv-prob-conv-algoirhtm}. It is given for free provided that there exists some (possibly unknown) strictly feasible solution for problem \eqref{inner-conv-prob-conv-algoirhtm}. However, we still need \eqref{ipp:dist-bound-obj} for ensuring the convergence result in \eqref{eq:ipp-dist-conv-zero}. After this modification in Proposition \ref{prop:ipp-error-assum}, {it would be easy to obtain results similar to those in Theorem~\ref{thm:ipp-subseq-bound} and Theorem~\ref{thm:uniform-bound-and-rate}.}  
\end{remark}
\begin{remark}
It is interesting to compare the complexity of proximal point iterations in Theorem~\ref{thm:uniform-bound-and-rate}  with  that of proximal point  for unconstrained nonconvex optimization (e.g. \cite{lan2018accelerated}). A quantitative difference is that for functional constrained problem the complexity bound has an unknown parameter $B$ which could possibly depend on the solution path, while for unconstrained problem, the parameters in the complexity bound usually globally depend on the initial solution, optimality gap or distance to the optimal solutions, and hence are easier to estimate. This distinction appears to indicate some substantial difficulty in developing complexity analysis for nonconvex optimization with nonconvex constraints.  	
\end{remark}
\begin{remark}\label{rem:ipp-4}
%While the complexity result allow us to have inexact solutions, 
The inexactness criteria (\ref{ipp:dist-bound-const}) and (\ref{ipp:dist-bound-obj}) describe convergence to the optimal solution for each convex subproblem. These criteria can be satisfied eventually if we employ ConEx for the subproblems, thanks to the last iterate convergence (\ref{eq:bounding_last_iterate_cor}).
However, {it is still difficult to estimate the total complexity, since ConEx does not guarantee convergence purely in terms of initial distance to the optimal solution. This difficulty holds for deterministic as well as stochastic cases.} 
Due to these issues, it is desirable to exploit more practical scenarios where proximal point methods can be combined with first order methods (such as ConEx) for large-scale and stochastic optimization.
\end{remark}

\subsection{Inexact proximal point under the strong feasibility assumption}\label{sec:inexact-strong-feas}
In this section, we present a variant of the inexact proximal point in which the subproblem
is approximately solved by ConEx (see Algorithm~\ref{alg:inexact-pp-2}).
To understand our motivation, consider the case when the objective function is given in the form 
of $f(x)=\Ebb_\xi\bracket*{F(x,\xi)}$,
where $F(x,\xi)$ is a stochastic function on some random variable $\xi$ and is possibly nonconvex with respect to the parameter $x$.
Consequently, the objective function in the subproblem \eqref{inner-convex-prob} is given by $\Ebb_\xi\bracket*{F(x,\xi)} + \mu_0 \gnorm{x-\wb{x}}{}{2}$.
As discussed in the previous section, stochastic
optimization algorithms (such as ConEx) for solving this type of problem 
will exhibit
a sublinear rate of convergence in expectation, 
which does not fit the inexactness criterion raised in Proposition~\ref{prop:ipp-error-assum}.
To alleviate this issue, we propose a new assumption as follows.

\begin{assumption}[Strong feasibility]\label{ass:strong-feasi} There exists $\wb{x} \in X$ such that
	\begin{equation}\label{eq:strict-feasibility}%\tag{\textbf{A0}}
	\psi_i(\wb{x}) \le -2\mu_i D_X^2, \tab i =1, \dots, m,
	\end{equation} 
%	where  \added{$D_X$ is given by (\ref{eq:diameter})}.
	where $D_X$ is defined in \eqref{eq:diameter}. 
\end{assumption}

The following proposition states that the KKT condition in \eqref{defi:kkt} is a first-order necessary optimality condition when Assumption \ref{ass:strong-feasi} holds. This result is similar to Proposition \ref{thm:MFCQ} where MFCQ is replaced by strong feasibility.  We defer the proof of this result to the appendix.
\begin{proposition}\label{prop:strong-feas-suff-cond}
	Let $x^*$ be a local optimal solution of the problem~\eqref{main-prob}. If Assumption \ref{ass:strong-feasi} is satisfied then there exists $y^\sbr{i} \ge 0,\ i \in [m]$ such that 
		\eqnok{eq:criterion-main-prob} holds.
\end{proposition}
%}
\begin{algorithm}[h]
	\caption{Inexact Constrained Proximal Point Algorithm with ConEx}
	\label{alg:inexact-pp-2}
	\begin{algorithmic}[1]
		\STATE Input $x_0$
		\FOR { $k = 1, \dots, K$}
		\STATE $x_k \gets $ a (stochastic) approximate solution of subproblem \eqref{inner-convex-prob} by ConEx. % $(\delta_k, \wb{\delta}_k)$-optimal solution (c.f. Definition \ref{define-apprx-opt}) of \ref{inner-convex-prob}.
		\ENDFOR
		\STATE Choose $\wh{k}$ uniformly at random from $\braces{1,2,...,K}$.
		\RETURN $x_{\wh{k}}$.
	\end{algorithmic}
\end{algorithm}

Note that \eqref{eq:strict-feasibility} is a local and a verifiable condition. Moreover, we will show in Lemma \ref{lem:strict-feasible-bound} that it provides a computable uniform bound $B$ on the dual solutions. %,  as in accordance with the result of Lemma \ref{lem:strict-feasible-bound}. 
While it appears that \eqref{eq:strict-feasibility} is quite distinct from MFCQ  (Definition~\ref{eq:mangafromo}), we would like to point out certain similarities between these two conditions. To understand this connection better, let us assume that $\psi_{i}$ is smooth function. Then, for all $x \in X$, we have 
\begin{align*}
	\psi_{i}(\wb{x}) &\ge \psi_{i}(x) + \inprod{\grad \psi_{i}(x)}{\wb{x}-x} - \tfrac{\mu_i}{2}\gnorm{\wb{x}-x}{}{2} \\
	\Rightarrow \inprod{\grad \psi_{i}(x)}{x-\wb{x}} &\ge \psi_{i}(x) - \psi_{i}(\wb{x}) -\tfrac{\mu_i}{2} \gnorm{\wb{x}-x}{}{2},%\label{eq:int_rel41}
\end{align*}
which implies that 
\begin{equation} \label{weak_cq}
	\inprod{\grad \psi_{i}(x)}{x-\wb{x}} \ge 0, \tab \forall x \in X \cap \{\psi_{i} \ge -\tfrac{3}{2}\mu_i D_X^2 \}.
\end{equation}
Recall that the existence of a Minty solution, $\wb{x}$, for variational inequality problem on mapping $\grad \psi_{i}$, is the following condition
%In particular, we note that the  assumption of the existence of a Minty solution for MVI(X, $\grad\psi_{i}$) is equivalent to existence of $\wb{x}$ such that 
\begin{equation}\label{eq:minty_sol}
	\inprod{\grad \psi_{i}(x)}{x-\wb{x}} \ge 0, \tab \forall x \in X,
\end{equation}
which is stronger than \eqref{weak_cq}. Hence $\psi$ satisfying \eqref{eq:strict-feasibility} is not necessarily quasi-convex. However, existence of Minty solution, $\wb{x}$, gives an `almost' sufficient condition for ensuring \eqref{eq:mangafromo} in the following way. Set $x = x^*$ in \eqref{eq:minty_sol}. Then we obtain that $z = \wb{x}-x^*$ satisfies \eqref{eq:mangafromo} with strict inequality replaced by nonstrict inequality. Since there is no implication from \eqref{weak_cq} to \eqref{eq:minty_sol} (in fact, the implication is in the opposite direction), so a direct comparison for the weaker among the two conditions \eqref{eq:strict-feasibility} and \eqref{eq:mangafromo-2}, can not be made as such.

Now, we state an important lemma which allows us to obtain dual boundedness under the strong feasibility assumption.
\begin{lemma}\label{lem:strict-feasible-bound}
Let $\psi_i(x;\hat{x})\coloneqq {\psi_i(x)} + 2 \mu_i W(x, \hat{x})$, $0\le i \le m$, where $\hat{x}\in X$ is  an arbitrary proximal center.
Then under Assumption~\ref{ass:strong-feasi}, the convex problem:
\begin{equation}\label{inner-convex-prob-v2} 
			\begin{split} 
			 \min_{x \in X} \tab & {\psi_0(x; \hat{x})}\\
			\text{s.t.} \  \tab  &{\psi_i(x; \hat{x})} \le 0, \tab i \in [m].
			\end{split}
		\end{equation}
is always feasible, there exists a solution $(x^+, y^+)$ satisfying the KKT-condition, and the  variable $y^+$ is uniformly bounded by:
\begin{equation}
	\|y^+\|_1 \le B :=\tfrac{\psi_0(\wb{x})-\psi_0^* + \mu_0 D_X^2}{\mu_{\min} D_X^2}, \label{ipp-bound-yplus}
	\end{equation}
	where $\mu_{\min}=\min_{1\le i\le m}\mu_i$.
\end{lemma}
%}
\begin{proofx}
	Based on Assumption~\ref{ass:strong-feasi}, we have from subproblem~(\ref{inner-convex-prob-v2}) that % \vspace{-0.5em}
	\begin{equation}\label{eq:strong-feasi-mid-01}
		\psi_{i}(\wb{x}, \hat{x}) \le -2\mu_i D_X^2 + 2\mu_i W(\wb{x}, \hat{x})\le -\mu_i D_X^2 < 0.
	\end{equation}
	Then the existence of KKT solution $(x^+, y^+)$  immediately follows from Slater's condition. 
	
	Moreover, notice that problem~(\ref{inner-convex-prob-v2}) differs from problem~(\ref{inner-convex-prob}) only at the choice of proximal center.
	Therefore, the same argument to prove Lemma~\ref{prop:strong-cvx} implies that
	\begin{equation}
		\psi_0(x,\hat{x})-\psi_0(x^+;\hat{x}) +\langle y^+, \psi(x;\hat{x})\rangle \ge (\mu_0+\mu^Ty^+) W(x, x^+), \quad x\in X.	
	\end{equation}
Let us place $x=\bar{x}$ in the above inequality. In view of the non-negativity of  $(\mu_0+\mu^Ty^+)$ and the definition of $\psi_0(x,\hat{x})$, one has 
\begin{equation} \label{eq:strong-feasi-mid-02}
	\psi_0(\wb{x})+2\mu_0 W(\wb{x}, \hat{x}) - \psi_0(x^+)-2\mu_0 W(x^+, \hat{x}) \ge \inprod{y^+}{-\psi(\wb{x}, \hat{x})}.
\end{equation}
	Combining the result (\ref{eq:strong-feasi-mid-01}) and (\ref{eq:strong-feasi-mid-02}) together, we  deduce	
	\begin{align*}
		\mu_{\min} \gnorm{y^+}{1}{}D_X^2 & \le (\mu^T y^+) D_X^2 \\
		& \le -\tsum_{i=1}^m y^{+\br{i}}\psi_{i}(\wb{x}, \hat{x}) \\
		&\le  \psi_0(\wb{x})-\psi_0(x^+) +\mu_0 D_X^2. 
	\end{align*}
	Finally, since $x^+$ is feasible to (\ref{inner-convex-prob-v2})  and  the feasible region of (\ref{inner-convex-prob-v2}) is a subset of the feasible region of the original problem~(\ref{main-prob}), we have $\psi_0(x^+)\ge \psi_0^*$. The result (\ref{ipp-bound-yplus}) immediately follows.
%	\qed
\end{proofx}

Note that in view of Lemma~\ref{lem:strict-feasible-bound}, the strong feasibility assumption ensures two key requirements for the analysis of the convergence rate of the inexact proximal point: 1) feasibility of proximal point subproblems  and 2)
 a uniform bound on the optimal dual variable $y_k^*$. To see the second part,  placing $x^{k-1}=\hat{x}$ in (\ref{inner-convex-prob-v2}), we immediately obtain the the bound 
\begin{equation}\label{eq:bound-y-strong-feas}
	\gnorm{y_k^*}{1}{} \le B=\tfrac{\psi_0(\wb{x})-\psi_0^* + \mu_0 D_X^2}{\mu_{\min} D_X^2},\tab k=1,2,\ldots.
\end{equation}
In this case, we only need to assume that $x_k$ satisfies the functional optimality gap and constraint violation  given in Definition~\ref{define-apprx-opt} as compared to \eqref{ipp:dist-bound-const} and \eqref{ipp:dist-bound-obj} in the previous section.

We develop the convergence result of Algorithm~\ref{alg:inexact-pp-2} in the following theorem.
\begin{theorem}\label{thm:inexact-pp}
In Algorithm~\ref{alg:inexact-pp-2}, suppose that Assumption~\ref{ass:strong-feasi} holds such that $\gnorm{y_k^*}{1}{} \le B$ and $B$ is given in Lemma~\ref{lem:strict-feasible-bound}. %\added{and \eqref{inner-convex-prob} is feasible} for $k=0,1,\ldots$.
Moreover, assume that the definition of $x_k$ in Algorithm~\ref{alg:inexact-pp-2}  is given by
\begin{equation} \label{def_x_k_spec}
x_k \gets  \mbox{ a stochastic} (\delta_k, \wb{\delta}_k)\mbox{-optimal solution (c.f. Definition \ref{define-apprx-opt}) of \eqref{inner-convex-prob}.}
\end{equation} 
Then $x_\wh{k}$ is a stochastic $(\varepsilon_K, \wb\varepsilon_K)$-KKT point of Problem \eqref{main-prob} with %\vspace{-0.8em}
\begin{equation} \label{eq:inexact_KKT_thm}
  	\varepsilon_K=\max\braces*{2L_\omega, 8L_\omega^2 \paran*{\mu_0+\mu_{\max}B}}\tfrac{\Gamma_K}{K},\quad\text{and } \wb{\varepsilon}_K= \tfrac{2}{\mu_0K}\Omega_K,
\end{equation}
where 
$\mu_{\max} := \max_{i \in [m]} \mu_i$, 
$\Gamma_K :=\Delta_{\psi_0} + B\wb{\Delta}_0 + \Omega_K$, $\Delta_{\psi_0} := \psi_0(x_0) - \min_{x \in X}\psi_0(x)$, $\wb{\Delta}_0=\|[\psi(x_0)]_+\|_2$ and $\Omega_K= \tsum_{k=1}^{K}\delta_k + B\tsum_{k=1}^{K}\wb{\delta}_{k}$.
\end{theorem}
\begin{proofx}
	Let $\Delta_k = \psi_0(x_k; x_{k-1}) -\psi_0(x_k^*; x_{k-1})$ and $\wb{\Delta}_k=\gnorm{\relu{\psi(x_k; x_{k-1})}}{2}{}$. 
	Using Definition \ref{define-apprx-opt} we have $\Ebb[\abs{\Delta_k}]\le \delta_k$ and $\Ebb[\wb{\Delta}_k] \le \wb{\delta}_k$.
	In view of Lemma \ref{prop:strong-cvx} and  the strong convexity of $\psi_0(; x_{k-1})$ and $\psi(; x_{k-1})$, we have 
	\begin{equation}\label{ipp:middle-result}
	\begin{aligned}
	\psi_0(x; x_{k-1}) + \tsum_{i=1}^{m} y_k^\sbr{i} \psi_i(x; x_{k-1}) &\ge \psi_0(x^*_k; x_{k-1}) + (\mu_0 + \mu^Ty_k^\ast) W(x, x_k^*) \\
	&= \psi_0(x_k; x_{k-1}) - \Delta_k +(\mu_0 + \mu^Ty_k^\ast)\, W(x, x_k^*)\\	
	&= \psi_0(x_k) + 2\mu_0 W(x_k, x_{k-1})-\Delta_k +(\mu_0 + \mu^Ty_k^\ast)\, W(x, x_k^*).
	\end{aligned} 
	\end{equation}
	Setting $x=x_k$ in \eqref{ipp:middle-result} yields
	\begin{align*}
\psi_0(x_k; x_{k-1}) + \tsum_{i=1}^{m} y_k^\sbr{i} \psi_i(x_k; x_{k-1}) &\ge \psi_0(x^*_k; x_{k-1}) + (\mu_0 + \mu^Ty_k^\ast) W(x_k, x_k^*)
\end{align*}
Setting $k=\wh{k}$ in the above relation and taking expectation, we have
\begin{align*}
\Ebb \bracket{\gnorm{x_{\wh{k}}-x_\wh{k}^*}{}{2}} \le 2\Ebb\bracket{W(x_\wh{k},x_\wh{k}^*) }
& \le \tfrac{2}{\mu_0K}\tsum_{k=1}^K\Ebb\bracket{\psi_0(x_{k}; x_{k-1}) - \psi_0(x_{k}^*;x_{k-1})+ \tsum_{i=1}^{m} y_k^\sbr{i} \psi_i(x_k; x_{k-1})} \\
& \le \tfrac{2}{\mu_0K}\tsum_{k=1}^K\Ebb\bracket{\psi_0(x_{k}; x_{k-1}) - \psi_0(x_{k}^*; x_{k-1}) + \tsum_{i=1}^{m} y_k^\sbr{i} [\psi_i(x_k; x_{k-1})}_+] \\
& \le  \tfrac{2}{\mu_0K}\tsum_{k=1}^K\Ebb\bracket{ \Delta_k + B\wb{\Delta}_k} \\
& \le \tfrac{2}{\mu_0K}\tsum_{k=1}^K(\delta_k+B\wb{\delta}_k).
\end{align*}
where the third inequality above is due to the Cauchy-Schwarz inequality and the boundedness of $\|y_k^*\|_2$: $\gnorm{y_k^*}{2}{} \le \gnorm{y_k^*}{1}{}\le B$.\\
Analogously, by setting $x = x_{k-1}$ in \eqref{ipp:middle-result} and noticing $\psi_0(x_{k-1}; x_{k-1})=\psi_0(x_{k-1})$ we have
\begin{equation}\label{ipp-strong-cvx-rela2}
	\begin{aligned}
		\psi_0(x_{k-1}) + B\wb{\Delta}_{k-1} 
		& \ge \psi_0(x_{k-1}; x_{k-1}) + \|y_k^*\|_2 \wb{\Delta}_{k-1}\\
		& \ge \psi_0(x_{k-1}; x_{k-1}) + \tsum_{i=1}^{m} y_k^\sbr{i}\psi_i(x_{k-1}; x_{k-1})   \\
		& \ge \psi_0(x_k^*; x_{k-1}) + (\mu_0 + \mu^Ty_k^\ast)\ W(x_{k-1}, x_k^*) \\
		& \ge \psi_0(x_k) -\Delta_k + 2\mu_0 W(x_k, x_{k-1}) +(\mu_0 + \mu^Ty_k^\ast)\ W(x_{k-1}, x_k^*).
	\end{aligned}
\end{equation}
Here the second inequality use the following property:
for $k>1$, 
	\begin{equation} 
	\tsum_{i=1}^{m} y_k^\sbr{i}\psi_i(x_{k-1}; x_{k-1}) \le \tsum_{i=1}^{m} \bracket{y_k^\sbr{i}\psi_i(x_{k-1}; x_{k-1})}_+ \le \tsum_{i=1}^{m} y_k^\sbr{i}\bracket{\psi_i(x_{k-1}; x_{k-1})}_+ \le \gnorm{y^\ast_k}{2}{} \wb{\Delta}_{k-1},
\end{equation}
and $\tsum_{i=1}^{m} y_1^\sbr{i}\psi_i(x_0; x_0)= \tsum_{i=1}^{m} y_1^\sbr{i}\psi_{i}(x_0)\le \|y_1^*\|_2 \wb{\Delta}_0.$\\
Summing up the inequality \eqref{ipp-strong-cvx-rela2} for $k = 1, \dots,  K$, we obtain 
\begin{equation}\label{ipp-sum-bound}
\begin{split}
	2\mu_0 \tsum_{k=1}^{K}W(x_k, x_{k-1}) &+\tsum_{k=1}^{K}(\mu_0 +\mu^Ty_k^*) W(x_{k-1}, x_k^*) \\
	& \le \psi_0(x_0) - \psi_0(x_K) +\tsum_{k=1}^{K}\Delta_k + B\tsum_{k=1}^{K}\wb{\Delta}_{k-1}  \\
	& \le \Delta_f + \tsum_{k=1}^{K}\Delta_k + B\tsum_{k=1}^{K}\wb{\Delta}_{k-1}, 	
\end{split}
\end{equation}
Furthermore, due to the  KKT condition for \eqref{inner-convex-prob}, we have 
\[d \paran[\big]{\pgrad \psi_0(x^\ast_k; x_{k-1}) + \tsum_{i=1}^{m}y_{k}^{\ast \br{i}}\pgrad \psi_{i}(x_k^\ast; x_{k-1}) +N_X(x_{k}^*), \zero} = 0.\]
Plugging the definition of $\pgrad \psi_0(x; x_{k-1})$ and  $\pgrad \psi_{i}(x; x_{k-1}), i\in [m],$ into the above inequality yields
\begin{equation}\label{ipp-kkt}
	d\paran[\big]{\pgrad \psi_0(x^\ast_k)  + \tsum_{i=1}^{m}y_k^{\ast \br{i}}\pgrad \psi_{i}(x_k^\ast) +2\paran{\mu_0 +\mu^Ty_k^*} (\grad\omega(x_k^*)-\grad\omega(x_{k-1})) + N_X(x_k^\ast), \zero}  = 0.
\end{equation}
Let $\wh{k}$ be the random index from $1,\dots,K$. Then, in view of \eqref{ipp-kkt}, \eqref{ipp-sum-bound} and bound on $\gnorm{y_k^*}{1}{}$, we have
\begingroup
\allowdisplaybreaks
\begin{align}
\Ebb\big[d&\paran[\big]{\pgrad \psi_0(x^\ast_\wh{k})  + \tsum_{i=1}^{m}y_\wh{k}^{\ast \br{i}}\pgrad \psi_{i}(x_\wh{k}^\ast) + N_X(x_\wh{k}^\ast) , \zero}^2 \big] \nonumber \\
&=	\tfrac{1}{K} \Ebb\braces*{\tsum_{k=1}^K d\paran{\pgrad \psi_0(x^\ast_k)  + \tsum_{i=1}^{m}y_k^{\sbr{i}}\pgrad \psi_{i}(x_k^\ast) + N_X(x_k^\ast), \zero}^2} \nonumber \\
&\le \tfrac{4}{K} \Ebb\braces*{\tsum_{k=1}^K \paran{\mu_0 + \mu^T y_k^\ast}^2 \gnorm{\grad\omega(x_k^*)-\grad\omega(x_{k-1})}{}{2}} \label{ipp:bound_grad} \\
& \le \tfrac{8L_\omega^2(\mu_0 + \mu_{\max} B)}{K} \Ebb\braces*{\tsum_{k=1}^K \paran{\mu_0 + \mu^T y_k^\ast}W(x_{k-1}, x^*_k)} \nonumber \\
& \le \tfrac{8L_\omega^2(\mu_0 + \mu_{\max} B)}{K} \bracket*{\Delta_f + \wb{\Delta}_0 + \tsum_{k=1}^{K}\delta_k + B\tsum_{k=2}^{K}\wb{\delta}_{k-1}}\nonumber\\
& \le \tfrac{8L_\omega^2(\mu_0 + \mu_{\max} B)}{K}\Gamma_K \nonumber
\end{align}
\endgroup
Moreover, using the complimentary slackness for the subproblem and the relation \eqref{ipp-sum-bound}, we have
\begin{align}
\tsum_{k=1}^K\tsum_{i=1}^m  \abs{y_k^{*\br{i}}\psi_{i}(x_k^*)} 
&=2\tsum_{k=1}^K\paran{\mu^Ty_k^*} W(x_k^*, x_{k-1}) \label{ipp:bound-cs} \\
&\le 2L_\omega \tsum_{k=1}^K\paran{\mu^Ty_k^*} W(x_{k-1}, x_k^*) \nonumber \\
& \le 2 L_\omega \bracket[\big]{\Delta_f + \tsum_{k=1}^{K}\Delta_k + B\tsum_{k=1}^{K}\wb{\Delta}_{k-1}}.	\nonumber
\end{align}
Therefore 
\[ \Ebb \bracket*{\tsum_{i=1}^m  \abs{y_\wh{k}^{*\br{i}}\psi_{i}(x_\wh{k}^*)}} = \tfrac{1}{K}\Ebb \bracket*{\tsum_{k=1}^{K}\tsum_{i=1}^m \abs{y_k^{*\br{i}}\psi_{i}(x_k^*)}} \le \tfrac{2L_\omega}{K}\Gamma_K.
\]
Hence we conclude the proof.
%\qed
\end{proofx}
\begin{remark}
	We should note that when $\psi_{i}, i \in [m]$, are convex functions then we can obtain a variant of Algorithm~\ref{alg:inexact-pp-2} where $x_k$ is a (stochastic) $(\delta_k, \wb{\delta}_k)$-optimal solution of \eqref{inner-conv-prob-conv-algoirhtm}. For this variant of Algorithm~\ref{alg:inexact-pp-2}, we can easily obtain the result of Theorem~\ref{thm:inexact-pp} under uniform boundedness of the Lagrange multiplier. Moreover, since constraints remain the same in \eqref{inner-conv-prob-conv-algoirhtm} for all $ k \ge 1$, we just need Slater's condition to ensure the uniform boundedness of $\{y_k^*\}$.
\end{remark}
In the following corollary, we state an immediate consequence of Theorem \ref{thm:inexact-pp} as well as the final complexity when using the ConEx method as subroutine to solve subproblem \ref{inner-convex-prob}. Before proceeding to the details of the corollary, we need to properly
redefine $B$ such that it satisfies $B\ge \max\{ \gnorm{y_k^*}{1}{}, \gnorm{y_k^*}{2}{}+1\}$. This allows the use of $B$ in the sense of Theorem \ref{thm:inexact-pp} as well as in the stepsize policy for 
the ConEx method in \eqref{eq:step_size}. 
\begin{corollary} \label{cor:iteration_complexity_inexact}
Under the assumptions of Theorem \ref{thm:inexact-pp},
	suppose that in Algorithm \ref{alg:inexact-pp-2}, we set $ \delta_k = c\wb{\delta}_k$ for some $c>0$, and $\wb{\delta}_k = \eps/(2c_1c_2)$, where 
	\begin{equation} \label{eq:c_params_appx_KKT}
	\begin{split}
	%\wb{c} &=  \max\braces*{\frac{\mu_{\max}}{\mu_0}(c_2-c), c_1}\\
	c_1 &= \max\braces*{2L_w, 8L_w^2(\mu_0+\mu_{\max}B)}\\
	c_2 &=c + B
	\end{split}
	\end{equation} 
	Then after running at most $K =  2c_1(\Delta_f+B\wb{\Delta}_0)/\eps$ iterations, we obtain an $\paran{\eps, \tfrac{2\eps}{\mu_0c_1} }$-KKT point of Problem \eqref{main-prob}. 
	In particular, if we run Algorithm \ref{alg-without-guess} for subproblem \eqref{inner-convex-prob}, then we obtain an $(\eps, \tfrac{2\eps}{\mu_0c_1} )$-KKT point in $O(\tfrac{1}{\eps}T_\eps)$ iterations, where $T_\eps$ is  defined in \eqref{eq:bounding_T_cor}.
\end{corollary}
\begin{proofx}
	Suppose $\delta_k$ and $\wb{\delta}_k$ are constants throughout Algorithm~\ref{alg:inexact-pp-2}. Then, according to \eqref{eq:inexact_KKT_thm}, we have $\eps_K \le c_1\Gamma_K/K$. Choosing given values of $\delta_k, \wb{\delta}_k$ and $K$, we have
	\begin{equation*}
	\eps_K \le c_1\tfrac{\Gamma_K}{K} =c_1 \bracket*{\tfrac{\Delta_f+B\wb{\Delta}_0}{K} + (c+B)\wb{\delta}_K} =c_1 \bracket*{\tfrac{\eps}{2c_1} + c_2\tfrac{\eps}{2c_1c_2}}=\eps.
	\end{equation*}
	Moreover, we have 
	\begin{equation*}
	\wb\varepsilon_K =\tfrac{2}{\mu_0 K}\Omega_K \le \tfrac{2}{\mu_0K}\Gamma_K \le \tfrac{2\eps}{\mu_0c_1}.	
	\end{equation*}
	Now noting that $\delta_k = \wb{\delta}_k = O(\eps)$ is a constant and using Corollary \ref{cor:iteration_counts}, we obtain $(\delta_{k}, \wb{\delta}_k)$-approximate solution of subproblem \eqref{inner-convex-prob} in $T_\eps$ iterations. Noting the definition $K$ in the statement of the corollary, we conclude the proof.
%	\qed
\end{proofx}

In the above corollary, we assume that the subproblem \eqref{inner-convex-prob} is solved using the ConEx method. Since $\gnorm{y_k^*}{2}{} \le \gnorm{y_k^*}{1}{}$, we have an upper bound $B$ on $\gnorm{y_k^*}{2}{}$ in \eqref{eq:bound-y-strong-feas}. Consequently, we can set the parameter $B$ of ConEx method, (call it $B_{\text{ConEx}}$ to avoid confusion with $B$ in \eqref{eq:bound-y-strong-feas}) in Theorem~\ref{cor:step_size_strong_cvx} appropriately.  In particular, we can set $B_\text{ConEx} = B +1$ which gives accelerated convergence for smooth problems. %with B replaced by $B+1$.
Moreover, if $\chi_{i}(x)$ is a simple function such that we can compute \prox{} operator in \eqref{eq:W-prox} for functions 
$\mu_iW(x, x_{k-1}) + \chi_{i}(x), \ i = 1, \dots, m,$ efficiently, then we solve each subproblem in the smooth strongly convex setting,
since $f_i, i = 1, \dots, m$ are smooth functions. Otherwise, we must include the nonsmooth convex function $\chi_{i}(x)$ in totality (or part thereof) with $f_i$, 
and then we can assume $\mu_iW(x, x_{k-1})$ is a simple function. In this case, we solve the subproblems in a nonsmooth strongly convex setting. 
We can derive from Corollary \ref{cor:iteration_complexity_inexact} 
and the definition of $T_\eps$ in \eqref{eq:bounding_T_cor} the final complexity bounds for different problem settings.
	\begin{itemize}
		\item \textbf{Smooth nonconvex case}: In this case, $T_\eps$ can be bounded $O(1/\eps^{1/2})$ in the deterministic case, 
		$O(1/\eps)$ in the semi-stochastic case and $O(1/\eps^2)$ in the fully-stochastic case. 
		Hence, in view of Corollary~\ref{cor:iteration_complexity_inexact}, we can compute an $(\eps, 2\eps/(\mu_0c_1))$-KKT point of 
		the nonconvex problem \eqref{main-prob} in $O(1/\eps^{3/2})$,  $O(1/\eps^2)$,
		and $O(1/\eps^3)$ iterations for the deterministic case, semi-stochastic case 
		and fully-stochastic cases, respectively.	
		\item \textbf{Nonsmooth nonconvex case}: In this case, $T_\eps$ can be bounded by $O(1/\eps)$ in the deterministic case, $O(1/\eps)$ in the semi-stochastic case 
		and $O(1/\eps^2)$ in the fully-stochastic case. Hence, in view of Corollary \ref{cor:iteration_complexity_inexact}, we 
		can compute an $(\eps, 2\eps/(\mu_0c_1))$-KKT point of the nonconvex problem \eqref{main-prob} in $O(1/\eps^{2})$ iterations for 
		the deterministic and semi-stochastic cases, and $O(1/\eps^3)$ iterations for the fully-stochastic case.		
	\end{itemize}
	Note that the dependence of these complexity bounds on different problem parameters can be made more precise
	in view of the definition of $T_\eps$ in \eqref{eq:bounding_T_cor}.

\section{Conclusion}
This paper focuses on stochastic first-order methods for solving both convex and nonconvex functional constrained problems.
For the convex case, we present a novel ConEx method which utilizes linear approximations of constraint functions to 
define the extrapolation step. This method is a simple and unified algorithm that can
achieve the best-known convergence rates for solving different functional constrained convex composite problems.
In particular, we show that ConEx attains a few new complexity results especially for 
the stochastic constrained setting and/or when the objective/constraint functions contains smooth components.
{
For the nonconvex case, we present new proximal point  methods which successively generate strictly feasible solutions from a sequence of strongly convex subproblems. 
Under some standard MFCQ type assumptions, we establish both asymptotic convergence to the KKT condition and iteration complexities to attain some approximate KKT points.
Under a different strong feasibility assumption, we establish the convergence of inexact proximal point methods without requiring feasible solutions. This is particularly attractive for large-scale and stochastic optimization where high accuracy is unachievable.
Efficiencies of the proximal point method which uses  ConEx to solve the subproblems are established under different problem settings.
}

\vspace{0.1in}

{\bf Acknowledgement.} The authors would like to thank Dr. Qihang Lin for a few inspiring discussions that help to improve the initial version of this work. {The authors would also like to thank an anonymous reviewer whose comments helped in a significant streamlining of the presentation of the paper.}
\bibliographystyle{acm}
%\bibliography{nccp}
\bibliography{ref-qdeng}

\begin{thebibliography}{10}

\bibitem{allen-zhu2016variance}
{\sc {Allen-Zhu}, Z., and {Hazan}, E.}
\newblock Variance reduction for faster non-convex optimization.
\newblock {\em International Conference on Machine Learning\/} (2016),
  699--707.

\bibitem{andreani2011on}
{\sc {Andreani}, R., {Haeser}, G., and {Martínez}, J.~M.}
\newblock On sequential optimality conditions for smooth constrained
  optimization.
\newblock {\em Optimization 60}, 5 (2011), 627--641.

\bibitem{andreani2018strict}
{\sc {Andreani}, R., {Martínez}, J.~M., {Ramos}, A., and {Silva}, P. J.~S.}
\newblock Strict constraint qualifications and sequential optimality conditions
  for constrained optimization.
\newblock {\em Mathematics of Operations Research 43\/} (2018), 693--717.

\bibitem{aravkin2018level}
{\sc {Aravkin}, A.~Y., {Burke}, J.~V., {Drusvyatskiy}, D., {Friedlander},
  M.~P., and {Roy}, S.}
\newblock Level-set methods for convex optimization.
\newblock {\em Mathematical Programming\/} (2018), 1--32.

\bibitem{Ben-Tal05}
{\sc Ben-Tal, A., and Nemirovski, A.}
\newblock Non-euclidean restricted memory level method for large-scale convex
  optimization.
\newblock {\em Mathematical Programming 102\/} (2005), 407--456.

\bibitem{bertsekasnonlinear}
{\sc Bertsekas, D.~P.}
\newblock {\em Nonlinear programming}.
\newblock Athena Scientific, 1999.

\bibitem{bertsekas2015convex}
{\sc Bertsekas, D.~P.}
\newblock {\em Convex optimization algorithms}.
\newblock Athena Scientific Belmont, 2015.

\bibitem{cartis2014on}
{\sc {Cartis}, C., {Gould}, N.~I., and {Toint}, P.~L.}
\newblock On the complexity of finding first-order critical points in
  constrained nonlinear optimization.
\newblock {\em Mathematical Programming 144}, 1 (2014), 93--106.

\bibitem{chambolle2011first}
{\sc {Chambolle}, A., and {Pock}, T.}
\newblock A first-order primal-dual algorithm for convex problems with
  applications to imaging.
\newblock {\em Journal of Mathematical Imaging and Vision 40}, 1 (2011),
  120--145.

\bibitem{chen2014pd}
{\sc Chen, Y., Lan, G., and Ouyang, Y.}
\newblock Optimal primal-dual methods for a class of saddle point problems.
\newblock {\em SIAM Journal on Optimization 24}, 4 (2014), 1779--1814.

\bibitem{davis2017proximally}
{\sc Davis, D., and Grimmer, B.}
\newblock Proximally guided stochastic subgradient method for nonsmooth,
  nonconvex problems.
\newblock {\em arXiv preprint arXiv: 1707.03505v4\/} (2017).

\bibitem{Quoc2011}
{\sc Dinh, Q.~T., Gumussoy, S., Michiels, W., and Diehl, M.}
\newblock Combining convex-concave decompositions and linearization approaches
  for solving bmis, with application to static output feedback.
\newblock {\em arXiv preprint arXiv:1109.3320\/} (2011).

\bibitem{Facchinei19}
{\sc Facchinei, F., Kungurtsev, V., Lampariello, L., and Scutari, G.}
\newblock Ghost penalties in nonconvex constrained optimization: Diminishing
  stepsizes and iteration complexity.
\newblock {\em arXiv preprint arXiv:1709.03384\/} (2017).

\bibitem{FangLiLinZhang18-1}
{\sc Fang, C., Li, C.~J., Lin, Z., and Zhang, T.}
\newblock Spider: Near-optimal non- convex optimization via stochastic
  path-integrated differential estimator.
\newblock {\em Advances in Neural Information Processing Systems\/} (2018),
  687--697.

\bibitem{RN304}
{\sc Frostig, R., Ge, R., Kakade, S., and Sidford, A.}
\newblock Un-regularizing: approximate proximal point and faster stochastic
  algorithms for empirical risk minimization.
\newblock In {\em International Conference on Machine Learning\/} (2015),
  pp.~2540--2548.

\bibitem{GhaLan12}
{\sc Ghadimi, S., and Lan, G.}
\newblock Stochastic first- and zeroth-order methods for nonconvex stochastic
  programming.
\newblock {\em SIAM Journal on Optimization 23(4)\/} (2013), 2341--2368.

\bibitem{RN97}
{\sc Ghadimi, S., and Lan, G.}
\newblock Accelerated gradient methods for nonconvex nonlinear and stochastic
  programming.
\newblock {\em Mathematical Programming 156}, 1-2 (2016), 59--99.

\bibitem{guler1992new}
{\sc G{\"u}ler, O.}
\newblock New proximal point algorithms for convex minimization.
\newblock {\em SIAM Journal on Optimization 2}, 4 (1992), 649--664.

\bibitem{aybat2018primal}
{\sc {Hamedani}, E.~Y., and {Aybat}, N.~S.}
\newblock A primal-dual algorithm for general convex-concave saddle point
  problems.
\newblock {\em arXiv preprint arXiv:1803.01401\/} (2018).

\bibitem{kong2018complexity}
{\sc Kong, W., Melo, J.~G., and Monteiro, R.~D.}
\newblock Complexity of a quadratic penalty accelerated inexact proximal point
  method for solving linearly constrained nonconvex composite programs.
\newblock {\em arXiv preprint arXiv:1802.03504\/} (2018).

\bibitem{Lan19}
{\sc Lan, G.}
\newblock {\em First-order and Stochastic Optimization Methods for Machine
  Learning}.
\newblock Springer-Nature, 2020.

\bibitem{Lan2018comm}
{\sc Lan, G., Lee, S., and Zhou, Y.}
\newblock Communication-efficient algorithms for decentralized and stochastic
  optimization.
\newblock {\em Mathematical Programming\/} (2018).

\bibitem{LanMon13-1}
{\sc Lan, G., and Monteiro, R. D.~C.}
\newblock Iteration-complexity of first-order penalty methods for convex
  programming.
\newblock {\em Mathematical Programming 138\/} (2013), 115--139.

\bibitem{LanMon16-1}
{\sc Lan, G., and Monteiro, R. D.~C.}
\newblock Iteration-complexity of first-order augmented lagrangian methods for
  convex programming.
\newblock {\em Mathematical Programming 155(1-2)\/} (2016), 511–547.

\bibitem{lan2018accelerated}
{\sc {Lan}, G., and {Yang}, Y.}
\newblock Accelerated stochastic algorithms for nonconvex finite-sum and
  multi-block optimization.
\newblock {\em arXiv preprint arXiv:1805.05411\/} (2018).

\bibitem{lan2018optrand}
{\sc Lan, G., and Zhou, Y.}
\newblock An optimal randomized incremental gradient method.
\newblock {\em Math. Program. 171}, 1-2 (2018), 167--215.

\bibitem{lan-zhou2016stoch}
{\sc Lan, G., and Zhou, Z.}
\newblock Algorithms for stochastic optimization with expectation constraints.
\newblock {\em arXiv preprint arXiv:1604.03887\/} (2016).

\bibitem{LNN}
{\sc Lemar\'{e}chal, C., Nemirovski, A.~S., and Nesterov, Y.~E.}
\newblock New variants of bundle methods.
\newblock {\em Mathematical Programming 69\/} (1995), 111--148.

\bibitem{lin_svrg}
{\sc Lin, Q., Ma, R., and Yang, T.}
\newblock Level-set methods for finite-sum constrained convex optimization.
\newblock In {\em Proceedings of the 35th International Conference on Machine
  Learning\/} (2018), vol.~80, pp.~3112--3121.

\bibitem{lin2018level}
{\sc {Lin}, Q., {Nadarajah}, S., and {Soheili}, N.}
\newblock A level-set method for convex optimization with a feasible solution
  path.
\newblock {\em SIAM Journal on Optimization 28}, 4 (2018), 3290--3311.

\bibitem{MaLinYang19}
{\sc Ma, R., Lin, Q., and Yang, T.}
\newblock Proximally constrained methods for weakly convex optimization with
  weakly convex constraints.
\newblock {\em arXiv preprint arXiv:1908.01871\/} (2019).

\bibitem{manga67}
{\sc Mangasarian, O., and Fromovitz, S.}
\newblock The fritz john necessary optimality conditions in the presence of
  equality and inequality constraints.
\newblock {\em Journal of Mathematical Analysis and Applications 17\/} (1967),
  37--47.

\bibitem{martinez2003a}
{\sc {Martínez}, J.~M., and {Svaiter}, B.~F.}
\newblock A practical optimality condition without constraint qualifications
  for nonlinear programming.
\newblock {\em Journal of Optimization Theory and Applications 118}, 1 (2003),
  117--133.

\bibitem{Nemirovski2005mirror}
{\sc Nemirovski, A.}
\newblock Prox-method with rate of convergence o(1/t) for variational
  inequalities with lipschitz continuous monotone operators and smooth
  convex-concave saddle point problems.
\newblock {\em SIAM Journal on Optimization 15}, 1 (2004), 229--251.

\bibitem{nesterov2018lectures}
{\sc Nesterov, Y.}
\newblock {\em Lectures on convex optimization}.
\newblock Springer, 2018.

\bibitem{NguyeLiuSchTak17-1}
{\sc Nguyen, L.~M., Liu, J., Scheinberg, K., and {c}, M.~T.}
\newblock A novel method for machine learning problems using stochastic
  recursive gradient.
\newblock {\em Proceedings of the 34th International Conference on Machine
  Learning 70\/} (2017), 2613--2621.

\bibitem{lee19nonconvexminmax}
{\sc Nouiehed, M., Sanjabi, M., Lee, J.~D., and Razaviyayn, M.}
\newblock Solving a class of non-convex min-max games using iterative first
  order methods.
\newblock {\em arXiv preprint arXiv:1902.08297\/} (2019).

\bibitem{PhamNguyenPhanTran18-1}
{\sc Pham, N.~H., Nguyen, L.~M., Phan, D.~T., and Tran-Dinh, Q.}
\newblock Proxsarah: An efficient algorithmic framework for stochastic
  composite nonconvex optimization.
\newblock {\em arXiv preprint arXiv:1902.05679\/} (2019).

\bibitem{Polyak67}
{\sc Polyak, B.}
\newblock A general method of solving extremum problems.
\newblock {\em Soviet Mathematics Doklady 8(3)\/} (1967), 593–597.

\bibitem{rafique2018convex}
{\sc {Rafique}, H., {Liu}, M., {Lin}, Q., and {Yang}, T.}
\newblock Non-convex min-max optimization: Provable algorithms and applications
  in machine learning.
\newblock {\em arXiv preprint arXiv:1810.02060\/} (2018).

\bibitem{reddi2016stochastic}
{\sc {Reddi}, S.~J., {Hefny}, A., {Sra}, S., {P\'{o}cz\'{o}s}, B., and {Smola},
  A.~J.}
\newblock Stochastic variance reduction for nonconvex optimization.
\newblock {\em International Conference on Machine Learning\/} (2016),
  314--323.

\bibitem{rockafellar00cvar}
{\sc Rockafellar, R.~T., and Uryasev, S.}
\newblock Optimization of conditional value-at-risk.
\newblock {\em Journal of Risk 2\/} (2000), 21--41.

\bibitem{shapirobook}
{\sc Shapiro, A., Dentcheva, D., and Ruszczynski, A.}
\newblock {\em Lectures on Stochastic Programming: Modeling and Theory, Second
  Edition}.
\newblock Society for Industrial and Applied Mathematics, Philadelphia, PA,
  USA, 2014.

\bibitem{WangMaYuan17-1}
{\sc Wang, X., Ma, S., and Yuan, Y.}
\newblock Penalty methods with stochastic approximation for stochastic
  nonlinear programming.
\newblock {\em Mathematics of Computation 86 (306)\/} (2017), 1793--1820.

\bibitem{WangJiZhouLiangTa18-1}
{\sc Wang, Z., Ji, K., Zhou, Y., Liang, Y., and Tarokh, V.}
\newblock Spiderboost: A class of faster variance-reduced algorithms for
  nonconvex optimization.
\newblock {\em arXiv preprint arXiv:1810.10690\/} (2018).

\bibitem{Xu2019-1}
{\sc Xu, Y.}
\newblock Iteration complexity of inexact augmented lagrangian methods for
  constrained convex programming.
\newblock {\em Mathematical Programming\/} (2019).

\bibitem{YuNeelyWei17}
{\sc Yu, H., Neely, M., and Wei, X.}
\newblock Online convex optimization with stochastic constraints.
\newblock {\em Advances in Neural Information Processing Systems\/} (2017),
  1428--1438.

\bibitem{ZhouPanGu2008}
{\sc Zhou, D., Xu, P., and Gu, Q.}
\newblock Stochastic nested variance reduction for nonconvex optimization.
\newblock In {\em Proceedings of the 32Nd International Conference on Neural
  Information Processing Systems\/} (USA, 2018), NIPS'18, Curran Associates
  Inc., pp.~3925--3936.

\end{thebibliography}

\appendix
\section{Proof of Proposition \ref{thm:MFCQ} \label{sec:mfcq-proof}}
Let us denote 
\begin{equation}\label{eq:convexified-functions}
	\begin{split}
		\wb{\psi}_i(x)&:= \psi_{i}(x) {+ \mu_iW(x, x^*)}, \quad i = 0, \dots, m.%\gnorm{x-x^*}{2}{2}.
	\end{split}
\end{equation}
It is easy to see that $\wb{\psi}_0(x)$ and $\wb{\psi}_i(x),\ i \in [m]$, are convex functions. Moreover, their respective subdifferentials can be written as%\vspace{-2mm}
\begin{align*}
%\pgrad \wb{\psi}_0(x) &= \{\grad f_0(x) + \mu_0 (x-x^*)\} + \pgrad \chi_0(x),\\
\pgrad \wb{\psi}_i(x) &= \{\grad f_i(x) + \mu_i  \grad W(x,x^*)\} + \pgrad \chi_i(x),
\end{align*}
{where $\grad W$ is the gradient in the first variable. Note that $\grad W(x^*, x^*) = \zero$.}
Consider the constrained convex optimization problem:
\begin{align}
\min_{x \in X} \tab &\wb{\psi}_0(x) \label{eq:pb_cvx} \\[-2mm] 
\text{s.t.} \tab &\wb{\psi}_i(x) \le 0, \tab i \in [m].\nonumber
\end{align}
Note that $x^*$ is a feasible solution of this problem. For sake of this proof, define $\Psi_k(x) := \wb{\psi}_0(x) + \tfrac{k}{2}\tsum_{i=1}^m \relu{\wb{\psi}_i(x)}^2 +\tfrac{1}{2}\gnorm{x-x^*}{2}{2}$. 
Let $S=\{x \in X: \gnorm{x-x^*}{2}{} \le \eps \}$ for some $\eps > 0$ such that any $x \in S$ which is feasible for \eqref{eq:pb_cvx}
satisfies $\wb{\psi}_0(x) \ge \wb{\psi}_0(x^*)$. Let $x_k :=\argmin_{x\in S} \Psi_k(x)$. {This is well-defined since $\Psi_k$ is a strongly convex function. Note that}
\begin{equation}\label{eq:upp-bd }
	\liminf_{k\rightarrow \infty} \Psi_k(x_k) \le\limsup_{k\rightarrow \infty}\Psi_k(x_k) \le \limsup_{k\rightarrow \infty} \Psi_k(x^*) = \wb{\psi}_0(x^*) < \infty,
\end{equation} 
{where second inequality follows from the fact that $x_k$ is optimal and $x^* \in S$ is a feasible point. 
Note that as $k \to \infty$, we have
%the optimality of $x_k$ and existence of $x^*\in S$, we have 
$\limsup_{k\rightarrow \infty} \Psi_k(x_k) < \infty \Rightarrow \limsup_{k\to \infty}\wb{\psi}(x_k) \le 0$.} \\[2mm]
{Moreover, note that $\dom(\liminf_{k\rightarrow \infty} \Psi_k) \subseteq \{x: \psi_{i}(x) \le 0, i \in [m]\}$. Also note that $\dom(\liminf_{k\rightarrow \infty} \Psi_k) \cap S \ne \emptyset$ since both sets contain $x^*$. Then, definition of set $S$ implies $\wb{\psi}_0(x) \ge \wb{\psi}_0(x^*)$ for all  $x \in \dom(\liminf_{k\to \infty} \Psi_k) \cap S$. Hence, $
\liminf_{k\to \infty} \Psi_k(x_k) \ge \liminf_{k\rightarrow \infty} \wb{\psi}_0(x_k) \ge \wb{\psi}_0(x^*)$. This inequality with \eqref{eq:upp-bd } implies that $\lim_{k\rightarrow\infty} \Psi_k(x_k) = \psi_0(x^*)$ and $x_k \to x^*$}. Hence, there exists $\wb{k}$ such that for all $k > \wb{k}$, $x_k \in \inte (S)$. So for such $k$ we can write the following first-order criterion for convex optimization  ($\bracket{\wb{\psi}_{i}}_+^2$ is a convex function):
\[ \zero \in N_X(x_k) + \pgrad \wb{\psi}_0(x_k) + k\bracket{\wb{\psi}(x_k)}_+ \pgrad \wb{\psi}(x_k) + x_k-x^*. \]
This implies that $x_k$ is also the optimal solution of
\[
\min_{x\in X} \bar{\psi}_0(x)+k\,\bracket{\wb{\psi}(x_k)}_+^T {\wb{\psi}(x)}+\tfrac{1}{2}\|x-x^*\|_2^2.\]
For simplicity, let us denote $v_k=k\,\bracket{\wb{\psi}(x_k)}_+$. Due to the optimality of $x_k$ of solving the above, we have
\begin{equation}\label{mfcq:penalty-opt}
\bar{\psi}_0(x_k)+v_k^T\bar{\psi}(x_k)+\tfrac{1}{2}\|x_k-x^*\|^2\le
\bar{\psi}_0(x)+v_k^T\bar{\psi}(x)+\tfrac{1}{2}\|x-x^*\|_2^2,\quad \forall x\in X.
\end{equation}

We claim that $\{v_k\}$ is a bounded sequence. Indeed, if this is true, then we can find a convergent subsequence $\{i_k\}$ with $\lim_{k\rightarrow\infty} v_{i_k}=v^*$. Taking $k\rightarrow \infty$ in \eqref{mfcq:penalty-opt}, we have
\begin{equation}\label{ipp:sub-optimum-limit}
\limsup_{k\rightarrow\infty}\bar{\psi}_0(x_{i_k}) + {v^*}^T\bar{\psi}(x^*)\le
\bar{\psi}_0(x)+ {v^*}^T\bar{\psi}(x)+\tfrac{1}{2}\|x-x^*\|_2^2,\tab \forall x\in X.
\end{equation}
Placing $x=x^*$, we have $\bar{\psi}_0(x^*)\ge\limsup \bar{\psi}_0(x_{i_k})$, thus $\lim_{k\rightarrow \infty}\bar{\psi}_0(x_{i_k})=\bar{\psi}_0(x^*)$ based on the lower semicontinuity of $\bar{\psi}_0$. In view of this discussion, $x^*$ optimizes the right side of \eqref{ipp:sub-optimum-limit}. Thus, applying the first order criterion, we have
\[
0\in \partial\bar{\psi}_0(x^*)+\sum_{i\in [m]}{v^{(i)}}^*\partial\bar{\psi}(x^*) + N_X(x^*).
\]
It remains to apply $\partial\bar{\psi}_0(x^*)=\partial \psi_0(x^*)$ and $\partial\bar{\psi}_i(x^*)=\partial \psi_i(x^*)$.

In addition, to prove complimentary slackness, it suffices to show when $\bar{\psi}_i(x^*)=\psi_i(x^*)<0$, we must have ${v^{(i)}}^*=0$. 
Since $x_k$ converges to $x^*$ and $\bar{\psi}_i$ is continuous,  there exists some $\exists k_0>0$, such that $\bar{\psi}_i(x_{i_k})<0$  when $k>k_0$. 
Hence ${v^{(i)}_{i_k}}^*=0$ by its definition. Taking the limit, we have ${v^{(i)}}^*=0$.

It remains to show the missing piece, that $\{v_k\}$ is a bounded sequence. We will prove by contradiction. If this is not true, we may assume $\lim_{k\rightarrow \infty}\|v_k\| = \infty$, passing to a subsequence if necessary.
Moreover, define $y_k=v_k/\|v_k\|$, since $y_k$ is a unit vector, it has some limit point, let us assume $\lim_{k\rightarrow \infty}y_{j_k}=y^*$ for a subsequence $\{j_k\}$. Dividing both sides of \eqref{mfcq:penalty-opt} by $\|v_k\|$ and then passing it to  the subsequence $\{j_k\}$, we have
\[
\bar{\psi}_0(x_{j_k})/\|v_{j_k}\|+y_{j_k}^T\bar{\psi}(x_{j_k})+\tfrac{1}{2\|v_{j_k}\|}\|x_{j_k}-x^*\|^2\le
\bar{\psi}_0(x){/\|v_{j_k}\|}+y_{j_k}^T\bar{\psi}(x)+\tfrac{1}{2\|v_{j_k}\|}\|x-x^*\|^2,\tab \forall x\in X.
\]
Taking $k\rightarrow\infty$, we have 
\[
{y^*}^T\bar{\psi}(x^*)\le {y^*}^T\bar{\psi}(x),\tab \forall x\in X.
\]

Since subsequence $x_{j_k}$ converges to $x^*$ and $\wb{\psi}_i$ is continuous, we see that $\wb{\psi}_i(x_{j_k}) < 0$ for any $i \notin \AA(x^*)$ for $k \ge k_0$. This implies $y_{j_k}= j_k \relu{\wb{\psi}_i(x_{j_k})} = 0$ for all $ k \ge k_0$ and for all $i \notin \AA(x^*)$. So we must have 
$\zero \in N_X(x^*) +  \tsum_{i \in \AA(x^*)} y^\sbr{i} \pgrad \psi_i(x^*)$. {Here, we have used the fact that $\grad W(x^*, x^*) = \zero$, implying that $\partial \wb{\psi}_i(x^*) = \partial \psi_{i}(x^*)$ for all $i = 0, \dots, m$.}
%where $\pgrad \psi_\AA(x^*)$ is a short notation for $\sum_{i \in \AA(x^*)} y^\sbr{i} \pgrad \psi_i(x^*)$ and $v^\sbr{i} > 0$ for at least one $i \in \AA(x^*)$. 
Let $u \in N_X(x^*)$ and $g_i(x^*) \in \pgrad \psi_i(x^*), i \in \AA(x^*)$ be such that 
\[ u + \tsum_{i \in \AA(x^*)} y^\sbr{i}g_i(x^*) = \zero.\] 
Then we can derive a contradiction by using MFCQ (Definition~\ref{eq:mangafromo}). Assume that $z$ satisfies  MFCQ  \eqref{eq:mangafromo}. Therefore, we have
\begin{align*}
0 = z^Tu + \tsum_{i \in \AA(x^*)}y^\sbr{i} z^Tg_i(x^*) &\le\tsum_{i \in \AA(x^*)}y^\sbr{i} z^Tg_i(x^*) \\
& \le \tsum_{i \in \AA(x^*)}y^\sbr{i} \max_{v \in \pgrad \psi_i(x^*)}z^Tv < 0,
\end{align*}
where first inequality follows since $z \in -N_X^*(x^*)$ and $u \in N_X(x^*)$ hence $z^Tu \le 0$, second inequality follows due to the fact that $y^\sbr{i} \ge 0$ and $g_i(x^*) \in \pgrad \psi_i(x^*)$ and last strict inequality follows since \eqref{eq:mangafromo} and $y^\sbr{i} > 0$ for at least one $i \in \AA(x^*)$.

\section{{Proof of Proposition \ref{prop:strong-feas-suff-cond}}}
{Let us define $\wb{\psi}_i, i = 0, \dots, m$ as in \eqref{eq:convexified-functions} where} $x^*$ is a local solution of \eqref{main-prob} then, 
\begin{alignat*}{2}
	&\exists\ \eps > 0 \quad \text{ s.t.} \quad \psi_0(x) \ge \psi_0(x^*) \quad &&\text{for all } x \in \{x \in X: \psi_i(x) \le 0, i \in [m],\ \gnorm{x-x^*}{}{} < \eps\}\\
	\Rightarrow &\exists\ \eps > 0 \quad \text{ s.t.} \quad \psi_0(x) \ge \psi_0(x^*) \quad &&\text{for all } x \in \{x \in X: \wb{\psi}_i(x) \le 0, i \in [m],\ \gnorm{x-x^*}{}{} < \eps\}\\
	\Rightarrow &\exists\ \eps > 0 \quad \text{ s.t.} \quad \wb{\psi}_0(x) \ge \psi_0(x^*) = \wb{\psi}_0(x^*) \quad &&\text{for all } x \in \{x \in X: \wb{\psi}_i(x) \le 0, i \in [m],\ \gnorm{x-x^*}{}{} < \eps\},
\end{alignat*}
where the first implication follows from the fact that $\wb{\psi}_{i}(x) \ge \psi_{i}(x)$ for all $ i \in [m]$ or equivalently $\{x \in X: \wb{\psi}_i(x) \le 0, i \in [m],\ \gnorm{x-x^*}{}{} < \eps\} \subseteq \{x \in X: \psi_i(x) \le 0, i \in [m],\ \gnorm{x-x^*}{}{} < \eps\}$, and second implication follows from the fact that $\wb{\psi}_i(x) \ge \psi_{i}(x)$. 

The last statement implies that $x^*$ is a local optimal solution for the convex problem \eqref{eq:pb_cvx}. Hence, it is also a global optimal solution. Based on \eqref{eq:strict-feasibility} from Assumption {\ref{ass:strong-feasi}}, we have,
\[\wb{\psi}_{i}(\wb{x}) = \psi_{i}(\wb{x}) + \mu_iW(\wb{x}, x^*) \le -2\mu_iD_X^2 + \mu_iD_X^2 = -\mu_iD_X^2 < 0.\]
Hence, by Slater condition, we have that there exists $y^* \ge \zero$ such that $(x^*, y^*)$ satisfy first order KKT-condition for the convex problem \eqref{eq:pb_cvx}. Thus, we have 
\begin{align*}
	 \partial\bar{\psi}_0(x^*)+\sum_{i\in [m]}{y^{(i)}}^*\partial\bar{\psi}_i(x^*) + N_X(x^*) &\ni \zero,\\
	 y^\sbr{i}\wb{\psi}_i(x^*) &=0, \quad i \in [m].
\end{align*}
It remains to apply $\partial\bar{\psi}_i(x^*)=\partial \psi_i(x^*)$ and $\wb{\psi}_i(x^*) = \psi_i(x^*)$ for all $i \in 0, \dots, m$. Hence, we conclude the proof.

\end{document}